\def\version{}

\documentclass[reqno,twoside,11pt]{amsart}
\usepackage{cite}
\usepackage{amsmath}
\usepackage{amsfonts}
\usepackage{amssymb}
\usepackage{epsfig}
\usepackage{verbatim}
\usepackage{color}

\IfFileExists{epsf.def}{\input epsf.def}{\usepackage{epsf}}
\IfFileExists{mathptmx.sty}{\usepackage{mathptmx}}{\usepackage{mathpazo}}
\usepackage{mathrsfs}

\DeclareFontFamily{OT1}{eusb}{} \DeclareFontShape{OT1}{eusb}{m}{n} {<5> <6> <7> <8> <9> <10> <11> <12> <14.4> eusb10}{}
\DeclareMathAlphabet{\eusb}{OT1}{eusb}{m}{n}

\DeclareFontFamily{OT1}{eusm}{} \DeclareFontShape{OT1}{eusm}{m}{n} {<5> <6> <7> <8> <9> <10> <11> <12> <14.4> eusm10}{}
\DeclareMathAlphabet{\eusm}{OT1}{eusm}{m}{n}

\DeclareFontFamily{OT1}{eufm}{} \DeclareFontShape{OT1}{eufm}{m}{n} {<5> <6> <7> <8> <9> <10> <11> <12> <14.4> eufm10}{}
\DeclareMathAlphabet{\mathfrak}{OT1}{eufm}{m}{n}

\DeclareFontFamily{OT1}{fraktura}{}
\DeclareFontShape{OT1}{fraktura}{m}{n} {<5> <6> <7> <8> <9> <10> <11> <12> <13> <14.4> [1.1] eufm10}{}
\DeclareMathAlphabet{\fraktura}{OT1}{fraktura}{m}{n}

\DeclareFontFamily{OT1}{cmfi}{} \DeclareFontShape{OT1}{cmfi}{m}{n} {<5> <6> <7> <8> <9> <10> <11> <12> <13> <14.4> [0.9] cmfi10}{}
\DeclareMathAlphabet{\cmfi}{OT1}{cmfi}{b}{n}

\DeclareFontFamily{OT1}{cmss}{} \DeclareFontShape{OT1}{cmss}{m}{n} {<5> <6> <7> <8> <9> <10> <11> <12> <13> <14.4> cmss10}{}
\DeclareMathAlphabet{\cmss}{OT1}{cmss}{m}{n}

\newtheoremstyle{thm}{1.5ex}{1.5ex}{\itshape\rmfamily}{} {\bfseries\rmfamily}{}{2ex}{}

\newtheoremstyle{def}{1.5ex}{1.5ex}{\rmfamily\sl}{} {\bfseries\rmfamily}{}{2ex}{}

\newtheoremstyle{rem}{1.3ex}{1.3ex}{\rmfamily}{} {\bfseries\rmfamily}{}{2ex}{}

\newtheoremstyle{ass}{1.5ex}{1.5ex}{\rmfamily\sl}{} {\bfseries\rmfamily}{}{2ex}{}

\newenvironment{proofsect}[1] {\vskip0.1cm\noindent{\rmfamily\itshape#1.}}{\qed\vspace{0.15cm}}

\theoremstyle{thm}
\newtheorem{theorem}{Theorem}[section]
\newtheorem{lemma}[theorem]{Lemma}
\newtheorem{proposition}[theorem]{Proposition}

\newtheorem*{Main Theorem}{Main Theorem.}
\newtheorem{corollary}[theorem]{Corollary}
\newtheorem{conjecture}[theorem]{Conjecture}
\newtheorem{question}[theorem]{Question}
\newtheorem{observe}[theorem]{Observation}

\newtheoremstyle{named}{}{}{\itshape}{}{\bfseries}{}{.5em}{\thmnote{#3}}
\theoremstyle{named}

\theoremstyle{def}
\newtheorem{definition}[theorem]{Definition}

\theoremstyle{rem}
\newtheorem{remark}[theorem]{{Remark}}

\numberwithin{equation}{section}
\renewcommand{\theequation}{\arabic{section}.\arabic{equation}}

\renewcommand{\section}{\secdef\sct\sect}
\newcommand{\sct}[2][default]{\refstepcounter{section}
\addcontentsline{toc}{section}
{{\tocsection {}{\thesection}{\!\!\!\!#1\dotfill}}{}}
\vspace{0.7cm}
\centerline{ 
\scshape\arabic{section}.\ #1} \nopagebreak \vspace{0.2cm}}
\newcommand{\sect}[1]{
\vspace{0.4cm} \centerline{\large\scshape\rmfamily #1}
\vspace{0.2cm}}

\renewcommand{\subsection}{\secdef\subsct\sbsect}
\newcommand{\subsct}[2][default]{\stepcounter{subsection}
\addcontentsline{toc}{subsection}
{{\tocsection{\!\!}{\hspace{1.2em}\thesubsection}{\!\!\!\!#1\dotfill}}{}}
\nopagebreak\vspace{0.45\baselineskip} {\flushleft\bf
\thesection.\arabic{subsection}~\bf #1.~}
\\*[3mm]\noindent
\nopagebreak}
\newcommand{\sbsect}[1]{
\vspace{0.1cm}\noindent
\textbf{#1.~}\vspace{0.1cm}}

\renewcommand{\subsubsection}{%
\secdef \subsubsect\sbsbsect}
\newcommand{\subsubsect}[2][default]{%
\refstepcounter{subsubsection} 
\addcontentsline{toc}{subsubsection}{{\tocsection{\!\!}
{\hspace{3.05em}\thesubsubsection}{\!\!\!\!#1\dotfill}}{}}
\nopagebreak
\vspace{0.15\baselineskip} \nopagebreak {\flushleft\rmfamily
\itshape\arabic{section}.\arabic{subsection}.\arabic{subsubsection}
\ \rmfamily #1\/.}\ }
\newcommand{\sbsbsect}[1]{\vspace{0.1cm}\noindent
\rmfamily \itshape
\arabic{section}.\arabic{subsection}.\arabic{subsubsection} \
\sffamily #1\/.\ }

\renewcommand{\caption}[1]{%
\vglue0.5cm
\refstepcounter{figure}
\begin{center}
\begin{minipage}[c]{0.8\textwidth}\small {\sc Fig.~\thefigure:\ }#1\end{minipage}
\end{center}
}

\newcommand{\dist}{\operatorname{dist}}

\newcommand{\textd}{\text{\rm d}\mkern0.5mu}

\newcommand{\texte}{\text{\rm  e}\mkern0.7mu}

\newcommand{\maxi}{{\text{\rm max}}}

\newcommand{\Var}{\text{\rm Var}}
\newcommand{\1}{{1\mkern-4.5mu\textrm{l}}}
\renewcommand{\1}{\text{\sf 1}}

\newcommand{\FF}{\mathcal F}

\newcommand{\NN}{\mathcal N}

\newcommand{\ZZ}{\mathcal Z}

\newcommand{\C}{\mathbb C}

\newcommand{\N}{\mathbb N}

\newcommand{\R}{\mathbb R}

\newcommand{\Z}{\mathbb Z}

\newcommand{\scrF}{\mathscr{F}}

\newcommand{\twoeqref}[2]{(\ref{#1}--\ref{#2})}

\newcommand{\cc}{{\text{\rm c}}}

\newcommand{\frakg}{\fraktura g}
\newcommand{\fraka}{\fraktura a}
\newcommand{\frakm}{\fraktura m}
\newcommand{\frakb}{\fraktura b}
\newcommand{\frakf}{\fraktura f}

\def\myffrac#1#2 in #3{\raise 2.6pt\hbox{$#3 #1$}\mkern-1.5mu\raise 0.8pt\hbox{$#3/$}\mkern-1.1mu\lower 1.5pt\hbox{$#3 #2$}}
\newcommand{\ffrac}[2]{\mathchoice%
	{\myffrac{#1}{#2} in \scriptstyle}
	{\myffrac{#1}{#2} in \scriptstyle}
	{\myffrac{#1}{#2} in \scriptscriptstyle}
	{\myffrac{#1}{#2} in \scriptscriptstyle}
}

\newcommand{\wt}{\widetilde}

\newcommand{\laweq}{\,\overset{\text{\rm law}}=\,}
\newcommand{\argmax}{\operatornamewithlimits{argmax}}

\newcommand{\lf}{\lfloor}
\newcommand{\rf}{\rfloor}

\newcommand{\Lawarrow}{{\,\overset{\text{\rm law}}\longrightarrow\,}}

\newcommand{\DGFF}{{\rm DGFF}}

\newcommand\Xiin{\Xi^{\text{\rm in}}}
\newcommand\Xiout{\Xi^{\text{\rm out}}}
\newcommand{\Cloc}{C^{\text{\rm loc}}_{\text{\rm b}}}
\newcommand{\Cb}{C_{\text{\rm b}}}
\newcommand{\Cc}{C_{\text{\rm c}}}

\begin{document}

\title[Full extremal process of 2D \DGFF{} \hfill \version\hfill]
{\large Full extremal process, cluster law and freezing \\ for the two-dimensional discrete Gaussian Free Field}

\author[\hfill  \version \hfill Biskup and Louidor]
{Marek~Biskup$^{1,2}$ and Oren~Louidor$^3$}
\thanks{\hglue-4.5mm\fontsize{9.6}{9.6}\selectfont\copyright\,\textrm{2018}\ \ \textrm{M.~Biskup, O.~Louidor.
Reproduction, by any means, of the entire
article for non-commercial purposes is permitted without charge.\vspace{2mm}}}
\maketitle

\vspace{-5mm}
\centerline{\textit{$^1$
Department of Mathematics, UCLA, Los Angeles, California, USA}}
\centerline{\textit{$^2$
Center for Theoretical Study, Charles University, Prague, Czech Republic}}
\centerline{\textit{$^3$
Faculty of Industrial Engineering and Management, Technion, Haifa, Israel}}

\vskip0.3cm
\begin{quote}
\footnotesize \textbf{Abstract:}
We study the  local structure of the  extremal process associated with the Discrete Gaussian Free Field (\DGFF{}) in scaled-up (square-)lattice versions of bounded open planar domains subject to mild regularity conditions on the boundary. We prove that, in the scaling limit, this process tends to a Cox process decorated by independent, correlated clusters whose distribution is completely characterized. As an application, we control the scaling limit of the discrete supercritical Liouville measure, extract a Poisson-Dirichlet statistics for the limit of the Gibbs measure associated with the \DGFF{} and establish the ``freezing phenomenon'' conjectured to occur in the ``glassy'' phase. In addition, we prove a local limit theorem for the position and value of the absolute maximum. The proofs are based on a concentric, finite-range decomposition of the \DGFF{} and entropic-repulsion arguments for an associated random walk. Although we naturally build on our earlier work on this problem, the methods developed here are largely independent.
\end{quote}


\tableofcontents

\section{Introduction}
\noindent
Recent years have witnessed remarkable advances in the understanding of extreme values of the two-dimensional Discrete Gaussian Free Field (\DGFF{}). This is a Gaussian process $\{h_x\colon x\in V\}$ in a proper subset~$V$ of the square lattice $\Z^2$ such that 
\begin{equation}
\label{E:DGFF-def}
E(h_x)=0\quad\text{and}\quad E(h_x\,h_y)=G^V(x,y),
\end{equation}
where~$G^V$ denotes the Green function of the simple symmetric random walk in~$V$ killed upon exit from~$V$. (We think of $h_x$ as fixed to zero outside~$V$.) 
Early efforts focused on the absolute maximum in square domains $V_N:=(0,N)^2\cap\Z^2$. Writing~$g:=2/\pi$ for the constant in the asymptotic $G^{V_N}(x,x)=g\log N+O(1)$ whenever~$N$ is large and~$x$ is deep inside~$V_N$, and denoting
\begin{equation}
\label{E:mN}
m_N:=2\sqrt g\,\log N-\frac34\sqrt g\,\log\log N,
\end{equation}
from the works of Bolthausen, Deuschel and Zeitouni~\cite{BDZ}, Bramson and Zei\-touni~\cite{BZ} and, particularly, Bramson, Ding and Zeitouni~\cite{BDingZ} we now know that the law of $\max_{x\in V_N}h_x-m_N$ converges to a non-degenerate limit as $N\to\infty$. 

In~\cite{BL1,BL2} the present authors turned to the extremal process associated with the \DGFF{} in a sequence $\{D_N\}$ of scaled-up versions of a bounded open set~$D\subset\C$ (see \twoeqref{E:1.1}{E:1.1a} for precise definitions).
Writing $\delta_a$ for the Dirac point-mass at~$a$, a standard way to describe extreme order statistics is to encode both the scaled positions and the centered values of the field~$\{h_x\colon x\in D_N\}$ into the random point measure
\begin{equation}
\label{E:1.3b}
\eta^D_N:=\sum_{x\in D_N}\delta_{\,x/N}\otimes\delta_{\,h_x-m_N}
\end{equation}
on $D\times\R$ and study its distributional limits as~$N\to\infty$. However, for conceptual reasons as well as the technical nature of the approach, the analysis~\cite{BL1,BL2} addressed only  the 
\emph{local} maxima (i.e., the tips of large ``peaks'') of the field. Using
\begin{equation}
\label{E:Lambda_r}
\Lambda_r(x):=\{z\in\Z^2\colon |z-x|\le r\},
\end{equation}
 to describe the meaning of the word ``local,'' these were captured by the random point measure~$\wt\eta^D_{N,r}$ on $D\times\R$ defined by
\begin{equation}
\label{E:1.4}
\wt\eta^D_{N,r}:=\sum_{x\in D_N}
\1_{\{h_x=\max_{z\in \Lambda_r(x)}h_z\}}
\delta_{\,x/N}\otimes\delta_{\,h_x-m_N}.
\end{equation}
For any sequence $r_N\to\infty$ with $N/r_N\to\infty$, it was then shown that
\begin{equation}
\label{E:1.5}
\wt \eta^D_{N,r_N}\,\,\underset{N\to\infty}\Lawarrow\,\,\text{\rm PPP}\bigl(Z^D(\textd x)\otimes\texte^{-\alpha h}\textd h\bigr),
\end{equation}
where $\text{\rm PPP}(\lambda)$ denotes the Poisson point process with intensity measure~$\lambda$,
\begin{equation}
\label{E:alpha}
\alpha:=\frac2{\sqrt g}=\sqrt{2\pi}
\end{equation}
and~$Z^D$ is a \emph{random} Borel measure on~$D$ with full support and $0<Z^D(D)<\infty$ a.s. This measure is independent of the  sequences~$r_N$ and $D_N$. 

The laws of the measures~$Z^D$ obey a host of specific properties that characterize them uniquely up to an overall multiplicative constant (see Theorem~2.8 of~\cite{BL2}). Despite its restriction to local maxima, \eqref{E:1.5} yields interesting conclusions for the full process~$\eta^D_N$ as well, e.g., the limit distribution of the scaled position and centered value of the absolute maximum,
\begin{equation}
\label{E:1.8ua}
P\Bigl(N^{-1}\argmax_{D_N}h\in A,\,\max_{x\in D_N}h(x)-m_N\le t\Bigr)
\,\underset{N\to\infty}\longrightarrow\,E\bigl(\,\widehat Z^D(A)\,\texte^{-\alpha^{-1}\texte^{-\alpha t}\,Z^D(D)}\bigr)
\end{equation}
for $A\subset D$ open and any $t\in\R$, as well as joint laws of  the  maxima in any finite number of disjoint  open  subsets of~$D_N$. Unfortunately, the methods of~\cite{BL1,BL2}, being tailored to the global structure of the extremal points, do not generalize to include local information.

The aim of the present article is to complete the description started in \cite{BL1,BL2} and derive the distributional limit of the \emph{full} extremal process \eqref{E:1.3b}. This requires development of techniques that capture the local structure of the extreme points and are, for reasons just mentioned, thus more or less unrelated to those of \cite{BL1,BL2}. As a reward, we are able to establish a number of additional results that have been conjectured in the so called ``glassy'' phase for the Gibbs measure naturally associated with the \DGFF{}. Our approach also yields a \emph{local} limit theorem for the location and the value of the absolute maximum. 

\section{Main results}
\noindent
We proceed to give precise statements of our results. The structure of the proofs, which constitute the remainder of this paper, is outlined along with some heuristics in Section~\ref{sec-2.4}.

\subsection{Full extremal process}
\label{sec-2.1}\noindent
The description of our results naturally starts with a limit theorem for the full extremal process. We will follow the setting of Biskup and Louidor~\cite{BL2} that considers the \DGFF{} over scaled-up versions of rather general domains in the complex plane.

Let~$\mathfrak D$ be the class of all bounded open sets $D\subset\C$ with a finite number of connected components and with boundary~$\partial D$ that has only a finite number of connected components  each of which has 
a positive (Euclidean) diameter. Given~$D\in\mathfrak D$, and letting $\text{dist}_\infty$ denote the $\ell^\infty$-distance on~$\Z^2$, let~$\{D_N\}$ be a sequence such that
\begin{equation}
\label{E:1.1}
	 D_N \subseteq \bigl\{x \in \Z^2 \colon \text{dist}_\infty(x/N,D^\cc)>\ffrac1N\bigr\}
\end{equation}
and, for each $\delta>0$ and all $N$ sufficiently large, also
\begin{equation}
\label{E:1.1a}
D_N\supseteq\bigl\{x \in \Z^2 \colon \text{dist}_\infty(x/N,D^\cc)>\delta\bigr\}.
\end{equation}
Note that $x\in D_N$ implies $x/N\in D$. A key point is that \eqref{E:1.5} holds for every~$D\in\mathfrak D$ (cf Theorem~2.1 of~Biskup and Louidor~\cite{BL2} for a formal statement). 

It is clear that the values of the field at nearby vertices are heavily correlated. Each high value of the field will thus come with a whole cluster of comparable values at basically the same  (scaled)  spatial location. For this reason, instead of \eqref{E:1.3b}, it is more natural to capture the extremal process associated with the \DGFF{} in~$D_N$ by way of \emph{structured} extremal point measures given by
\begin{equation}
\eta^D_{N,r}:=\sum_{x\in D_N}
\1_{\{h_x=\max_{z\in \Lambda_r(x)}h_z\}}
\delta_{\,x/N}\otimes\delta_{\,h_x-m_N}\otimes\delta_{\,\{h_{x}-h_{x+z}\colon z\in\Z^2\}}.
\end{equation}
These are formally Radon measures on~$D\times\R\times\R^{\Z^2}$ (with the product topology on~$\R^{\Z^2}$) that extend the point measures from \eqref{E:1.4} by including control of the ``shape'' of the field ``around'' the local maxima. 

The space of Radon measures on $D\times\R\times\R^{\Z^2}$ is naturally endowed with the topology of vague convergence which, in turn, permits discussion of distributional limits. (A sequence of random Radon measures thus converges in distribution if, and only if, integrals of compactly-supported continuous functions converge in distribution.) Our principal result~is~then:

\begin{theorem}[Full scaling limit]
\label{thm-main}
For each $D\in\mathfrak D$, let~$Z^D$ be the random Borel measure on~$D$ for which \eqref{E:1.5} holds.
There is a probability measure $\nu$ on $[0,\infty)^{\Z^2}$ such that for each $D\in\mathfrak D$ and each~$r_N$ with $r_N\to\infty$ and $r_N/N\to0$, 
\begin{equation}
\label{E:1.9}
\eta^D_{N,r_N}\,\,\,\underset{N\to\infty}\Lawarrow\,\,\,\text{\rm PPP}\bigl(Z^D(\textd x)\otimes\texte^{-\alpha h}\textd h\otimes\nu(\textd\phi)\bigr),
\end{equation}
where $\alpha$ is as in \eqref{E:alpha}. Moreover, $\phi_0=0$ and $\{x\in\Z^2\colon\phi_x\le c\}<\infty$ $\nu$-a.s.\ for each~$c>0$.
\end{theorem}

To interpret this result, one can say that although the spatial positions of the local maxima are correlated via the random measure $Z^D$, the configurations around each of the local maxima --- the shapes of the nearly-highest peaks --- are (in the limit) independent samples from~$\nu$. 

A consequence of the above theorem is a representation of the limit law of the ``unstructured'' extremal process $\eta^D_N$ from \eqref{E:1.3b} by means of a \emph{cluster process}:

\begin{corollary}[Cluster process]
\label{cor-cluster-process}
For the setting and notation of Theorem~\ref{thm-main}, let $\{(x_i,h_i)\colon i\in\N\}$ enumerate the points in a sample from $\text{\rm PPP}(Z^D(\textd x)\otimes\texte^{-\alpha h}\textd h)$. Let $\{\phi_z^{(i)}\colon z\in\Z^2\}$, $i\in\N$, be independent samples from the measure $\nu$, independent of $\{(x_i,h_i)\colon i\in\N\}$. Then 
\begin{equation}
\label{E:1.6}
\eta^D_N\,\,\,\underset{N\to\infty}\Lawarrow\,\,\,\sum_{i\in\N}\sum_{z\in\Z^2}\delta_{\,(x_i,\,h_i-\phi_z^{(i)})}.
\end{equation}
The measure on the right is locally finite on $D\times\R$ a.s. 
\end{corollary}

When we disregard the spatial positions, \eqref{E:1.6} becomes even simpler: 
\begin{equation}
\label{E:1.6a}
\sum_{x\in D_N}\delta_{\,h_x-m_N}
\,\,\,\,\underset{N\to\infty}\Lawarrow\,\,\,
\sum_{i\in\N}\sum_{z\in\Z^2}\,\delta_{\,t_i+\alpha^{-1}\log Z^D(D)-\phi_z^{(i)}}
\end{equation}
where $\{t_i\colon i\in\N\}$ is a sample from Gumbel PPP$(\texte^{-\alpha h}\textd h)$, $\{\phi_z^{(i)}\colon z\in\Z^2\}$ are i.i.d.\ samples from~$\nu$ and~$Z^D(D)$ is the total mass of~$Z^D(\textd x)$, with all three objects independent of one another. The limit process is thus a randomly-shifted Gumbel process decorated by independent and identically distributed clusters. 

\begin{figure}[t]
\centerline{\includegraphics[width=0.7\textwidth]{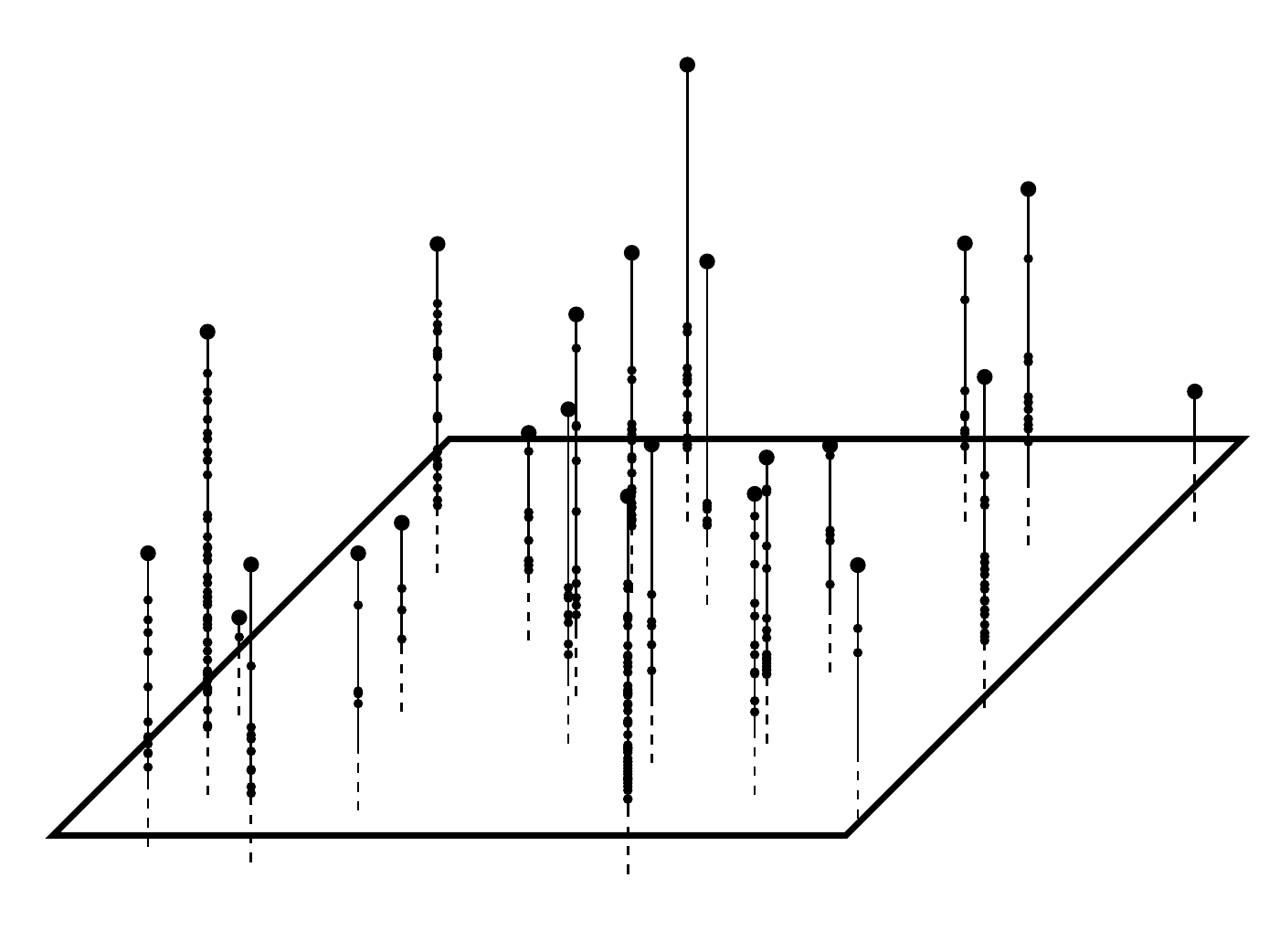}}
\begin{quote}
\small
\vglue-0.2cm
\caption{An illustration of the limit \eqref{E:1.6} of the unstructured point process~$\eta^D_N$ for~$D$ being a unit square. The ``notches'' on the vertical lines mark the locations of sample points of an individual cluster. Only the points above a certain (arbitrary) base value are~included. The top values in the clusters (marked by  larger  bullets) are distributed according to the Cox process in~\eqref{E:1.5}.
\label{fig0}
}
\normalsize
\end{quote}
\end{figure}

A randomly-shifted, i.i.d.-decorated Gumbel process is the limit of the extremal process associated with the Branching Brownian motion; see Arguin, Bovier and Kistler~\cite{ABK1,ABK2,ABK3}, A\"idekon, Berestycki, Brunet and Shi~\cite{ABBS} or Bovier and Hartung~\cite{Bovier-Hartung} who even track additional information analogous to our spatial positions. In these studies the cluster law is defined by taking the whole ensemble of branching Brownian motions conditioned to have an excessively large absolute maximum. (This essentially forces the whole process to be just one cluster.) It turns out that a relatively explicit description of the cluster law is possible in our case as well.

Let $\nu^0$ be the law of the mean-zero \DGFF{} in~$\Z^2$ pinned to zero at~$x=0$ or, equivalently, the \DGFF{} in~$\Z^2\smallsetminus\{0\}$. (Recall that all of our \DGFF{}s have zero boundary conditions.) Explicitly, $\nu^0$ is a Gaussian law on~$\R^{\Z^2}$ with mean zero and covariance
\begin{equation}
\label{E:2.8}
\text{Cov}_{\nu^0}(\phi_x,\phi_y)=\fraka(x)+\fraka(y)-\fraka(x-y),
\end{equation}
where $\fraka\colon\Z^2\to\R$ is the potential kernel of the simple symmetric random walk started from zero with the explicit representation
\begin{equation}
\label{E:2.8u}
\fraka(x):=\int_{[-\pi,\pi]^2}\frac{\textd k}{(2\pi)^2}\,\frac{1-\cos(k\cdot x)}{\sin(k_1/2)^2+\sin(k_2/2)^2}.
\end{equation}
Note that $\phi_0=0$ $\nu^0$-a.s. The following theorem characterizes the cluster law in terms of a limit of conditional laws involving $\nu^0$. Below we write ``$|x| \leq r$'' as a shorthand for ``$\forall x$ such that $|x| \leq r$''. This practice will be used  without further apology  throughout this manuscript.
\begin{theorem}[Cluster law]
\label{thm-lessmain}
The measure $\nu$ in Theorem~\ref{thm-main} is given by the weak limit
\begin{equation}
\label{E:1.14a}
\nu(\cdot)=\lim_{r\to\infty} \,\nu^0\biggl(\,\phi+\frac2{\sqrt g}\,\fraka\in\cdot\,\bigg|\,\phi_x+\frac2{\sqrt g}\,\fraka(x)\ge0\colon\, |x|\le r\biggr).
\end{equation}
\end{theorem}

The strong FKG property associated with the \DGFF{} (cf Lemma~\ref{lemma-FKG}) shows that, for increasing events, the probability on the right of \eqref{E:1.14a} is non-decreasing in~$r$. Unfortunately, this is not enough to infer the existence of the limit as stated: As the space of configurations is not compact, work is needed to prevent blow-ups to infinity (i.e., to prove tightness on $[0,\infty)^{\Z^2}$). To see that this is in fact a subtle issue, we note:

\begin{theorem}
\label{thm-2.5}
There is a constant~$\tilde c_\star\in(0,\infty)$ such that
\begin{equation}
\label{E:2.9y}
\nu^0\biggl(\,\phi_x+\frac2{\sqrt g}\,\fraka(x)\ge0\colon\, |x|\le r\biggr)\sim\frac {\tilde c_\star}{(\log r)^{1/2}},\qquad r\to\infty.
\end{equation}
\end{theorem}

Thus, staying non-negative in larger and larger volumes is increasingly costly for~$\phi+\frac2{\sqrt g}\fraka$ in spite of the ``logarithmic boost'' received from~$\fraka$ (recall that $\fraka(x)=g\log|x|+O(1)$ as~$|x|\to\infty$). Consequently, one cannot pass the limit $r\to\infty$ inside the conditioning event in \eqref{E:1.14a} and the measures $\nu$ and~$\nu^0$ are supported on disjoint sets. 

We remark that, in the proof of Theorem~\ref{thm-2.5}, the constant~$\tilde c^\star$ is obtained as the $\ell\to\infty$ limit of the quantity~$\Xiin_\ell(f)$ from \eqref{E:5.1ua} for~$f:=1$. However, we do not seem to have a way to express this constant without a limit procedure.

\subsection{Local limit theorem for absolute maximum}
The proofs of the above theorems hinge on control of the \DGFF{} conditioned on the maximum occurring at a given point. This is achieved by way of (what we call) a concentric decomposition of the \DGFF{}; see Section~\ref{sec3}. An augmented version of the same argument then yields also a \emph{local} limit theorem for both the value and the position of the absolute maximum:

\begin{theorem}[Local limit law for absolute maximum]
\label{thm-LLL}
For each $D\in\mathfrak D$ there exists a continuous function $\rho^D\colon D\times\R\to[0,\infty)$ such that for each $\{D_N\}$ that obeys \twoeqref{E:1.1}{E:1.1a}, for each $a<b$ and uniformly in~$x$ over compact subsets of~$D$,
\begin{equation}
\label{E:2.11c}
\lim_{N\to\infty}\, N^2\,\, P\Bigl(\operatornamewithlimits{argmax}_{D_N}h=\lfloor xN\rfloor,\,\max_{D_N}h-m_N\in (a,b)\Bigr)=\int_a^b \rho^D(x,t)\,\textd t.
\end{equation}
Moreover, $x\mapsto \rho^D(x,t)$ is, for each $t\in\R$, the Radon-Nikodym derivative of the measure
\begin{equation}
\label{E:2.12c}
A\mapsto \texte^{-\alpha t}E\bigl(\,Z^D(A)\,\texte^{-\alpha^{-1}\texte^{-\alpha t}\,Z^D(D)}\bigr)
\end{equation}
with respect to the Lebesgue measure on~$\R^2$. Here $Z^D$ is the random measure from \eqref{E:1.5}.
\end{theorem}

Comparing \twoeqref{E:2.11c}{E:2.12c} with~\eqref{E:1.8ua}, the local limit theorem is consistent with the limit law of the scaled and centered maximum.  The large-$t$ asymptotic of $\rho^D(x,t)$ is thus known; indeed, from Biskup and Louidor~\cite[Theorem~2.6]{BL2} we infer that
\begin{equation}
\label{E:2.13ua}
\rho^D(x,t)\sim t\texte^{-\alpha t}\,\psi^D(x),\qquad t\to\infty,
\end{equation}
where, for~$D$ simply connected,
\begin{equation}
\label{E:2.13}
\psi^D(x)=c_\star\text{rad}_D(x)^2
\end{equation}
 with $c_\star$ a positive constant and $\text{rad}_D(x)$ denoting the conformal radius of~$D$ from~$x$. Our proof gives a formula for~$\rho^D(x,t)$ (see \eqref{E:6.30a}) but this is still quite inexplicit as singular limits remain involved. (Notwithstanding, we do get a somewhat more explicit representation of the constant~$c_\star$ than what has been available so far; see Remark~\ref{remark-6.7}.)
 
\begin{figure}[t]
\vglue0.2cm
\centerline{\includegraphics[width=\textwidth]{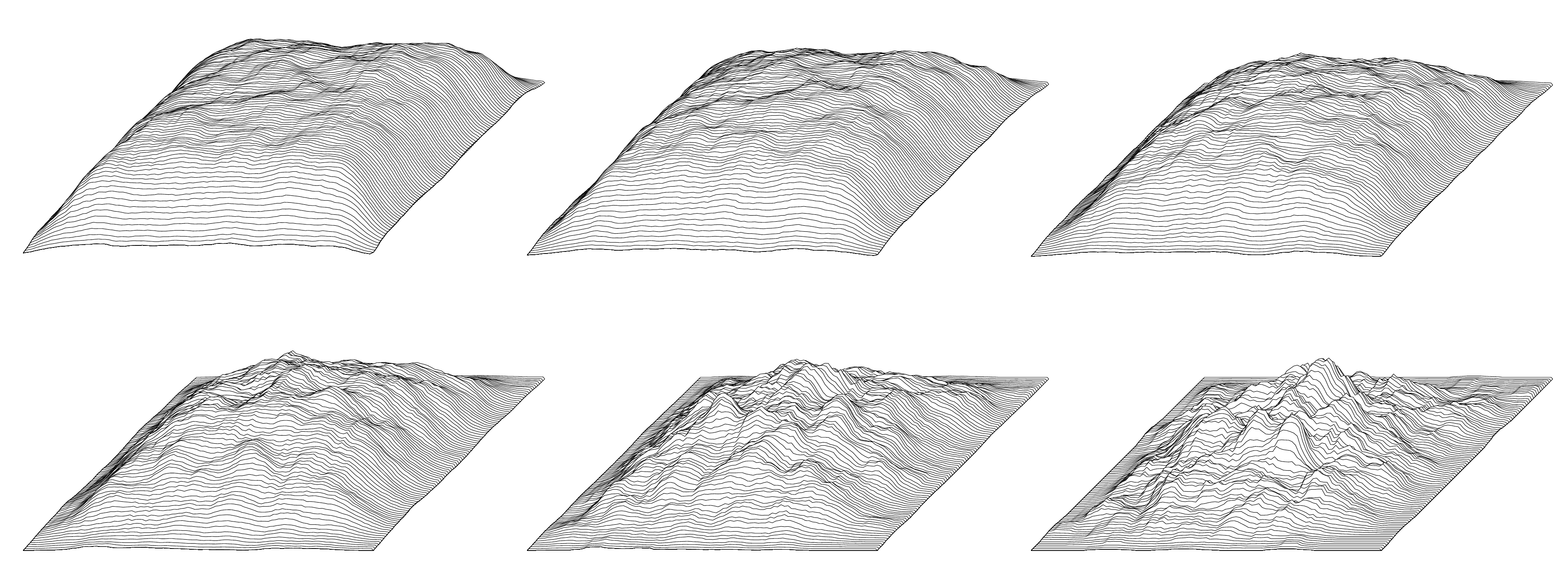}}
\begin{quote}
\small
\caption{Empirical plots of~$x\mapsto\rho^D(x,t)$ obtained from a set of about 100000 samples of the maximum of the DGFF on a 100$\times$100 square. The plots (labeled left to right starting with the top row) correspond to~$t$ increasing by uniform amounts over an interval of length 3 with~$t$ in the fourth figure set to the empirical mean. A certain amount of smoothing has been applied to eliminate discrete effects.
\label{fig6}
}
\normalsize
\end{quote}
\end{figure}

The absence of explicit expressions for the law of the maximum can presumably be blamed on strong correlations between the spatial positions of the large local maxima. To demonstrate the point, consider an analytic bijection $f\colon D\to D'$ and let $\{(x_i,h_i,\phi_i)\colon i\in\N\}$ enumerate the sample points of the (full) limit process~$\eta^D$. Theorem~\ref{thm-main} above and Theorem~2.5 of Biskup and Louidor~\cite{BL2} show that the point process with ``points''
\begin{equation}
\Bigl\{\bigl(f(x_i),h_i+2\sqrt g\log|f'(x_i)|,\phi_i\bigr)\colon i\in\N\Bigr\}
\end{equation}
is equidistributed to~$\eta^{D'}$. However, the shift $2\sqrt g\log|f'(x_i)|$ will generally permute the order of near-maximal points and so an explicit, autonomous expression for the law of the maximum alone is not reasonable to expect.

\subsection{Freezing and Liouville measure in the glassy phase}
\label{sec-2.3}\noindent
The control of the cluster distribution in Theorem~\ref{thm-lessmain} presents us with an opportunity to resolve  a couple of questions that have been debated in the spin-glass literature for some time. Before we formulate these precisely, let us give a bit of necessary motivation.

Given a sample~$h$ of the \DGFF{} on~$D_N$, it is natural to consider a continuous-time (variable speed) random walk on~$D_N$ that makes steps as the ordinary simple symmetric random walk but with exponential holding times whose parameter at vertex~$x$ is~$\texte^{\beta h_x}$. The stationary law of this walk is then given by the Gibbs (probability) measure on~$D_N$ defined by
\begin{equation}
\label{E:2.15tt}
\mu_{\beta,N}^D\bigl(\{x\}\bigr):=\frac1{\ZZ_N(\beta)}\,\texte^{\beta h_x}
\quad\text{where}\quad\ZZ_N(\beta):=\sum_{x\in D_N}\texte^{\beta h_x}.
\end{equation}
Disregarding the conventional minus sign in the exponent, the parameter $\beta\in[0,\infty)$ thus earns the meaning of the inverse temperature.

Obviously, $\mu_{\beta,N}^D$ puts the more weight on a vertex the larger the field is there. However, large field values are increasingly sparse and so a trade-off with entropy occurs. As observed by Carpentier and Le Doussal~\cite{Carpentier-LeDoussal}, this results in a phenomenon akin to that known from the Random Energy Model: The mass of~$\mu_{\beta,N}^D$ asymptotically concentrates on the level set
\begin{equation}
\label{E:level}
\Bigl\{x\in D_N\colon h_x\approx \frac{\beta\wedge\beta_\cc}{\beta_\cc}\, 2\sqrt g\log N\Bigr\},
\end{equation}
 where $\beta_\cc:=\alpha$. (A proof of this can be extracted directly from Daviaud's work~\cite{Daviaud}.) In particular, a phase transition occurs in this model as~$\beta$ varies through~$\beta_\cc$: Indeed, at~$\beta=\beta_\cc$ the support of~$\mu_{\beta,N}^D$ reaches the absolute maximum of $h$, and remains concentrated there for all~$\beta>\beta_\cc$.
 
Our focus here is the detailed structure of the scaled limiting measure in the supercritical ``glassy'' regime; i.e., when~$\beta>\beta_\cc$. Given a (Borel) probability measure~$Q$ on~$\C$ and a parameter~$s>0$, define the point measure $\Sigma_{s,Q}$ by
\begin{equation}
\label{E:2.17}
\Sigma_{s,Q}(\textd x):=\sum_{i\in\N}q_i\,\delta_{\,X_i},
\end{equation}
where $\{q_i\}$ enumerates the sample points of a Poisson process on $[0,\infty)$ with intensity $x^{-1-s}\textd x$ and~$\{X_i\}$ are independent samples from~$Q$, independent of the~$\{q_i\}$. We will in fact need to take~$Q$ random; in this case~$Q$ is sampled first and the construction of $\Sigma_{s,Q}$ is performed conditionally on the sample of~$Q$. Recall the notation 
\begin{equation}
\widehat Z^D(A):=\frac{Z^D(A)}{Z^D(D)}.
\end{equation}
Then we have:

\begin{theorem}[Liouville measure in the glassy phase]
\label{thm-2.7}
Given $D\in\mathfrak D$, let $D_N$ and $m_N$ be as above and let $Z^D$ denote the random measure from Theorem~\ref{thm-main}. For each $\beta>\beta_\cc:=\alpha$ there is a constant $c(\beta)\in(0,\infty)$ such that
\begin{equation}
\label{E:2.23}
\sum_{z\in D_N}\texte^{\beta (h_z-m_N)}\delta_{\,z/N}(\textd x)
\,\,\,\underset{N\to\infty}\Lawarrow\,\,\,
c(\beta) \,Z^D(D)^{\beta/\beta_\cc}\,\,\,\Sigma_{\beta_\cc/\beta,\,\widehat Z^D}(\textd x),
\end{equation}
where, we recall, $Z^D$ is sampled first and~$\Sigma_{\beta_\cc/\beta,\,\widehat Z^D}$ is defined conditionally  on~$Z^D$. Moreover, the constant $c(\beta)$ admits the explicit representation
\begin{equation}
\label{E:2.9b}
c(\beta):=\beta^{-\beta/\beta_\cc}\bigl[E_\nu(Y^\beta(\phi)^{\beta_\cc/\beta})\bigr]^{\beta/\beta_\cc}\quad\text{\rm with}\quad
Y^\beta(\phi):=\sum_{x\in\Z^2}\texte^{-\beta \phi_x}.
\end{equation}
In particular, $E_\nu(Y^\beta(\phi)^{\beta_\cc/\beta})<\infty$ for each~$\beta>\beta_\cc$. 
\end{theorem}

This result settles Conjecture~6.1 of Rhodes and Vargas~\cite{RV-review} for the \DGFF{} on the square lattice. A proof of this conjecture has previously been given in the context of continuum (the so-called star-scale invariant) fields and the associated multiplicative chaos; cf Theorem~2.8 in Madule, Rhodes and Vargas~\cite{MRV}. However (as stated in~\cite{MRV}) the cut-off procedures employed in \cite{MRV} would not permit extensions to the \DGFF{} discussed here. 

\begin{figure}[t]
\vglue0.2cm
\centerline{\includegraphics[width=0.5\textwidth]{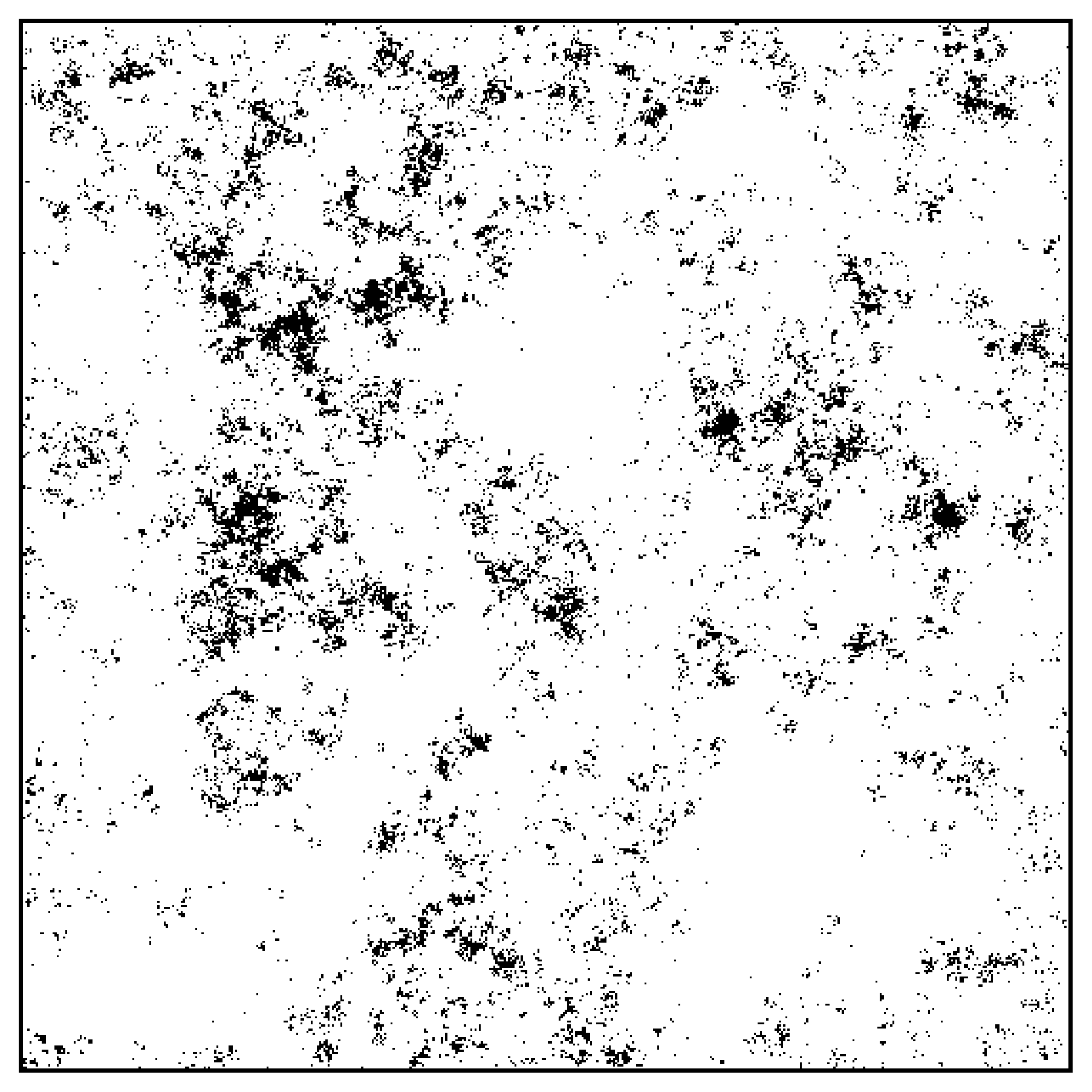}}
\begin{quote}
\small
\caption{A sample of the level set for the DGFF on a $500\times500$ square corresponding to values above 1/3 of the absolute maximum (which occurs at height~$8.17$ in this sample). The fractal, and highly-correlated, nature of this set is quite apparent. Level sets for higher cutoffs become increasingly sparse and thus difficult to visualize.
\label{fig2}
}
\normalsize
\end{quote}
\end{figure}

A direct consequence of Theorem~\ref{thm-2.7} is the following estimate on the size of the level sets for samples from~$\nu$:
\begin{equation}
\limsup_{t\to\infty}\frac1t\log\bigl|\{x\in\Z^2\colon\phi_x\le t\}\bigr| \le \alpha,\qquad\nu\text{\rm-a.s.},
\end{equation}
meaning, in short, that the set where $\phi_x\le t$ has asymptotically at most $\texte^{(\alpha+o(1))t}$ vertices. We in fact know that the limit exists with equality (to~$\alpha$) on the right-hand side, but this needs more than just the argument above. An estimate on the diameter of the level set can be gleaned from Proposition~\ref{cor-4.16}, although we do not believe that estimate to be even close to sharp.

Theorem~\ref{thm-2.7} directly yields a limit law for the normalization constant in \eqref{E:2.15tt},
\begin{equation}
\forall \beta>\beta_c\colon\qquad
\ZZ_N(\beta)\texte^{-\beta m_N}\,\,\underset{N\to\infty}\Lawarrow\,\,Z(D)^{\beta/\beta_\cc}X,
\end{equation}
where $X$ is independent of~$Z^D$ and has the law of a totally skewed (a.s.\ positive) $\beta_\cc/\beta$-stable random variable with an explicit overall normalization. More interestingly, Theorem~\ref{thm-2.7} also gives us the desired characterization of the above Gibbs measure~$\mu_{\beta,N}^D$ for $\beta>\beta_\cc$. Recall that the \emph{Poisson-Dirichlet law} with parameter $s\in(0,1)$, to be denoted PD$(s)$, is a probability measure on non-increasing non-negative normalized sequences,
\begin{equation}
\label{E:nu-level}
\Bigl\{\{p_i\}\colon p_1\ge p_2\ge\dots\ge0,\,\,\sum_{i\in\N}p_i=1\Bigr\}
\end{equation}
obtained by taking the sample points of the Poisson process on $[0,\infty)$ with intensity $x^{-1-s}\textd x$, normalizing them by their (a.s.-finite) sum and ordering the values decreasingly.  Then we have:

\begin{corollary}[Poisson-Dirichlet limit for the Gibbs measure]
\label{cor-2.7}
Let $D\in\mathfrak D$ and let $\mu_{\beta,N}^D$ be the Gibbs measure defined in \eqref{E:2.15tt}. Then for all $\beta>\beta_\cc:=\alpha$,
\begin{equation}
\label{E:2.21qq}
\sum_{z\in D_N}\mu_{\beta,N}^D\bigl(\{z\}\bigr)\delta_{\,z/N}(\textd x)
\,\,\,\underset{N\to\infty}\Lawarrow\,\,\,
\sum_{i\in \N}p_i\delta_{\,X_i},
\end{equation}
where $\{X_i\}$ are (conditionally on~$Z^D$) i.i.d.\ with common law~$\widehat Z^D$, while $\{p_i\} \laweq\text{\rm PD}(\beta_\cc/\beta)$ is independent of $Z^D$ and thus also $\{X_i\}$. 
\end{corollary}

A version of the Poisson-Dirichlet convergence \eqref{E:2.21qq} for overlap distributions has previously been established by Arguin and Zindy~\cite{Arguin-Zindy}.

Another consequence of Theorem~\ref{thm-2.7} is the proof of the so-called \emph{freezing phenomenon}. This is a term introduced in the context of the Branching Brownian Motion by Derrida and Spohn~\cite{Derrida-Spohn} and further expounded on by Fyodorov and Bouchard~\cite{Fyodorov-Bouchard}. Recently, Subag and Zeitouni~\cite{Subag-Zeitouni} offered a deeper insight into the connection between this concept and the type of cluster process we establish in Theorem~\ref{thm-main}. Our result is as follows:

\begin{corollary}[Freezing]
\label{cor-2.8}
Let $D\in\mathfrak D$ and, in accord with the above references, denote
\begin{equation}
\label{E:2.22ue}
 G_{N,\beta}(t):=E\biggl(\,\exp\Bigl\{-\texte^{-\beta t}\sum_{x\in D_N}\texte^{\beta h_x}\Bigr\}\biggr).
\end{equation}
Let $m_N$ be as above and let $Z^D$ be the measure from Theorem~\ref{thm-main}. Then for each $\beta>\beta_\cc:=\alpha$ there is a constant $\tilde c(\beta)\in\R$ such that
\begin{equation}
\label{E:2.26}
G_{N,\beta}\bigl(\,t+m_N+\tilde c(\beta)\bigr)\,\underset{N\to\infty}\longrightarrow\,E\bigl(\texte^{-Z^D(D)\,\texte^{-\alpha t}}\bigr).
\end{equation}
The constant $\tilde c(\beta)$, given explicitly in \eqref{E:6.54a}, depends only on the law of~$\nu$ and that only via the expectation in \eqref{E:2.9b}. 
\end{corollary}

A non-degenerate limit of $G_{N,\beta}(t+m_{N,\beta})$ --- with a suitable centering sequence $m_{N,\beta}$ --- exists for all values of $\beta\in[0,\infty)$.  (This follows from the existence of a distributional limit of suitably normalized $\sum_{x\in D_N}\texte^{\beta h_x}$. For $\beta<\beta_\cc$ this is proved in Rhodes and Vargas~\cite[Theorem~5.12 and Appendix~B]{RV-review}; the case $\beta=\beta_\cc$ is addressed, albeit in much less detail, in~\cite[Theorem~5.13]{RV-review}.)  The term  ``freezing'' then refers to the fact that the limit function ceases to depend on~$\beta$ (i.e., ``freezes'') once~$\beta$ passes through~$\beta_\cc$. 

\subsection{Heuristics and outline}
\label{sec-2.4}\noindent
The rest of this article is devoted to the proofs of the above results. In order to give an outline of what is to come, let us begin by an appealing heuristic argument why the above~$\nu$ should appear as the distribution of the clusters. 

A natural way to get to the clusters is by conditioning on the relevant local maxima. Assuming these occur at points $x_1,\dots,x_n$ with the field values at $m_N+t_1,\dots,m_N+t_n$, respectively, if we condition $h$ on just taking these values at these points, the field decomposes into the sum
\begin{equation}
\label{E:2.25}
h^{D_N\smallsetminus\{x_1,\dots,x_n\}}+\frakg_N,
\end{equation}
where $h^{D_N\smallsetminus\{x_1,\dots,x_n\}}$ is the \DGFF{} in~$D_N\smallsetminus\{x_1,\dots,x_n\}$ while $\frakg_N$ is discrete harmonic there, equal to $m_N+t_i$ at each~$x_i$ and vanishing outside~$D_N$. Since $x_1,\dots,x_n$ will be separated by distances of order~$N$ with high probability, as $N\to\infty$, we have $m_N+t_i-\frakg_N(x)\to\frac2{\sqrt g}\fraka(x-x_i)$ whenever $x$ is sufficiently near~$x_i$. Hence,
\begin{equation}
\label{E:2.29ua}
(m_N+t_i)-\Bigl[\,h^{D_N\smallsetminus\{x_1,\dots,x_n\}}(x_i+\cdot)+\frakg_N(x_i+\cdot)\Bigr] \,\,\,\,\underset{N\to\infty}\Lawarrow\,\,\,\,\phi^{(i)}(\cdot)+\frac2{\sqrt g}\fraka(\cdot),
\end{equation}
where $\phi^{(i)}\laweq-\phi^{(i)}$ is distributed according to~$\nu^0$. However, once we impose that each~$x_i$ is also a local maximum, the field in \eqref{E:2.25} must also not exceed the value at~$x_i$ in an $r$-neighborhood of~$x_i$. This forces the conditioning on~$\phi^{(i)}+\frac2{\sqrt g}\fraka\ge0$ in~$\Lambda_r(x_i)$. See Fig.~\ref{fig1} for an illustration.

Our proof of Theorem~\ref{thm-main} proceeds more or less along  these lines, albeit not without a significant amount of technical overhead caused, in hindsight, by the singular nature of the conditioning spelled out in Theorem~\ref{thm-2.5}. Indeed, in order to prove the Poisson law in \eqref{E:1.9}, we have to establish a version of \eqref{E:2.29ua} with $\phi^{(1)},\dots,\phi^{(n)}$ \emph{independent} of each other. This is easy for the unconditioned law but becomes a challenge once we condition on small probability events.

\begin{figure}[b]
\centerline{\includegraphics[width=0.99\textwidth]{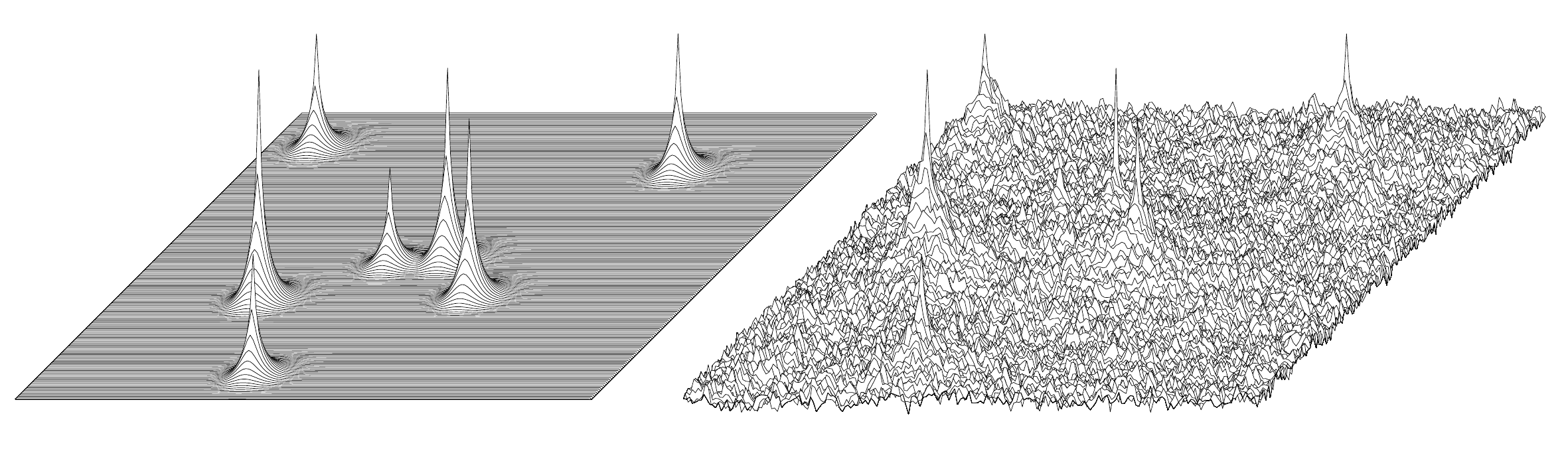}}
\begin{quote}
\small
\caption{Left: An illustration of  function $\mathfrak g_N$ from \eqref{E:2.25} for the underlying domain~$D$ being a unit square,~$N:=200$ and $n:=7$ local maxima being fixed. Right: A sample of the full field $h^{D_N\smallsetminus\{x_1,\dots,x_n\}}+\frakg_N$ conditioned on the values at $x_1,\dots,x_7$ to be local maxima in an $r$-neighborhood thereof for $r:=20$.
\label{fig1}
}
\normalsize
\end{quote}
\end{figure}

The approach we take is that we first focus on a single local maximum and analyze the situation around it in full detail. This is facilitated by a natural \emph{concentric decomposition} of the \DGFF{} pinned to a high value (of the form~$m_N+t$) into a sum of independent and, more or less, localized random fields with good control of the tails. The development and basic properties of the concentric decomposition are the subject of Section~\ref{sec3}. 

An attractive feature of the concentric decomposition is that the overall growth rate of the pinned field can be encoded into a \emph{``backbone'' random walk} (for the  pinned field  in all of~$\Z^2$) or an associated \emph{random-walk bridge} (for the  pinned field  in finite volume). The requirement that the field be maximized at the pinning point then more or less amounts to having this random walk/bridge stay above a polylogarithmic curve for a large interval of times; see Section~\ref{sec-4.1}. 

Estimating the probability that the random walk/bridge stays above such curves is easier to perform first for Brownian motion/bridge; the statements (Propositions~\ref{prop-uncond-BM-positive-cor}--\ref{prop-mine-too-corollary}) are to be found in Section~\ref{sec-4.2} with the proofs relegated to Appendix~A. Straightforward interpolation arguments, simplified considerably by the fact that the random walk/bridge has Gaussian (albeit not identically distributed) steps, then allow us to pull these to the discrete time setting; see Section~\ref{sec-4.3}.

The underlying mechanism that makes all this work is \emph{entropic repulsion}, which causes a random walk conditioned to stay above a slowly-varying negative curve to actually rise, with overwhelming probability, even above a (slowly-varying) positive curve. Explicit statements appear in Propositions~\ref{prop-entropy-BM}--\ref{prop-entropy-BB}. In Section~\ref{sec-4.4} we use this to derive a number of useful estimates for the various attributes of the concentric decomposition; these feed repeatedly into the proofs later. It is Section~\ref{sec4} where our paper makes close contact with the literature on the Branching Brownian Motion; specifically, the pioneering work by Bramson~\cite{Bramson-CPAM}. 

In Section~\ref{sec5a} we return to the problem of the pinned DGFF and start harvesting results. First (in Section~\ref{sec-5.2}) we establish the existence and non-dege\-neracy of the cluster law and thus prove Theorem~\ref{thm-lessmain}. With some additional work (spelled out in Sections~\ref{sec-5.3}--\ref{sec-5.4}), this yields  also the proof of the full scaling limit in Theorem~\ref{thm-main}. The main technical input here is the conditional version of the convergence in \eqref{E:2.29ua}, still for a single point, carried out in Propositions~\ref{prop-4.14}--\ref{prop-4.13}. The contributions of individual local maxima are separated with the help of Proposition~\ref{lemma-4.18}. These propositions are the core technical steps of the proofs of our main results.

Section~\ref{sec6} addresses the proofs of the remaining theorems; first the local limit theorem for the position and value of the maximum (Theorem~\ref{thm-LLL}) and then the proof of Theorem~\ref{thm-2.7} and Corollaries~\ref{cor-2.7}--\ref{cor-2.8} dealing with the Liouville measure, Poisson-Dirichlet statistics and freezing. 
A key point here is the fact that the contribution from the clusters to the measure on the left-hand side of \eqref{E:2.23} can be completely absorbed, via the expectation of the quantity $Y^\beta(\phi)$ in \eqref{E:2.9b}, into the overall normalizing constant~$c(\beta)$. This can (roughly) be attributed to:

\begin{observe}
Let $\{z_i\colon i\in\N\}$ enumerate the sample points from a Gumbel Poisson point process with intensity $\texte^{-\lambda x}\textd x$ for some $\lambda>0$ and let $\{X_i\colon i\in\N\}$ be i.i.d.\ random variables with $\theta:=E(\texte^{\lambda X_1})<\infty$, independent of $\{z_i\colon i\in\N\}$. Then also $\{z_i+X_i\colon i\in\N\}$ is a sample from a Gumbel process but this time with intensity $\theta\texte^{-\lambda x}\textd x$.
\end{observe}

\noindent
In fact, absorbing the contribution of the clusters into a single random variable is the main step of the above proofs; the rest follows fairly directly (by exponentiating) from Theorem~\ref{thm-main}. 

Our proofs naturally use a number of facts about Gaussian processes, and in particular the DGFF, that have been proved earlier. To systematize referencing, we list these results as separate lemmas in Appendix~B and then quote only these lemmas in the proofs.  Many of these facts, as well as other aspects of the extremal values of the DGFF, have  been reviewed in Biskup~\cite{Biskup-PIMS}. 

\subsection{Connections and open questions}
We conclude this section by listing some questions of further interest. Our first question concerns analytic properties of the density~$\rho^D$ from Theorem~\ref{thm-2.7}. The asymptotic expression~\eqref{E:2.13ua} and our simulations in Fig.~\ref{fig6} suggest the following:

\begin{conjecture}
For each~$t\in\R$, the function $x\mapsto\rho^D(t,x)$ is bounded and tends to zero as~$x$ approaches~$\partial D$. In particular, the function admits a continuous extension to all of~$\overline D$.
\end{conjecture}

Notice that, although the considerations of the \DGFF{} (cf Lemma~\ref{lem:6}) give the absolute continuity of the measure in \eqref{E:2.12c} with respect to the Lebesgue measure, they only imply a certain integrability condition for $x\mapsto\rho^D(t,x)$. Unfortunately, our proofs do not seem to be able to determine the boundary regularity of this function either.
 
As our next question, we wish to point out that the limits in Theorem~\ref{thm-2.7} and Corollary~\ref{cor-2.7} should have versions even for~$\beta=\beta_\cc$:

\begin{conjecture}
There is~$c\in(0,\infty)$ such that for each~$D\in\mathfrak D$, 
\begin{equation}
\label{E:2.23b}
\sqrt{\log N}\sum_{z\in D_N}\texte^{\beta_\cc (h_z-m_N)}\delta_{\,z/N}(\textd x)
\,\,\,\underset{N\to\infty}\Lawarrow\,\,\,
c \,Z^D(\textd x).
\end{equation}
In particular,
\begin{equation}
\label{E:2.21qqb}
\sum_{z\in D_N}\mu_{\beta_\cc,N}^D\bigl(\{z\}\bigr)\delta_{\,z/N}(\textd x)
\,\,\,\underset{N\to\infty}\Lawarrow\,\,\,
\widehat Z^D(\textd x).
\end{equation}
\end{conjecture}
 
\noindent
A version of the convergence in \eqref{E:2.23b} is claimed in Rhodes and Vargas~\cite[Theorem~5.13]{RV-review} in the framework of isoradial graphs, although the details given there (for the relevant critical case) are scarce. A full proof of \eqref{E:2.21qqb} still requires identification of the limit measure there with our~$Z^D$ which has not been accomplished so far.  (\textit{Update in revision}: The identification now appears in the revised version of~\cite{BL2}.)  
We note that, unlike for $\beta>\beta_\cc$ where the limit measures in \eqref{E:2.23} and \eqref{E:2.21qq} are purely atomic, the limit  measures  in \twoeqref{E:2.23b}{E:2.21qqb} have no atoms at all (see Biskup and Louidor~\cite[Theorem~2.1]{BL2}).

A corresponding question arises also for $\beta<\beta_\cc$, although there the limit of the Liouville measure is already reasonably well understood (see Rhodes and Vargas~\cite[Theorem~5.12 and Appendix~B]{RV-review}). However, it is not clear how this translates into the control of the level sets in \eqref{E:level} when ``$\approx$'' is replaced by ``$\ge$'' and $m_N$ is replaced by a function that grows in the leading order as $\lambda (2\sqrt g\log N)$ for some~$\lambda\in(0,1)$. Here is an attempt to formulate this more precisely:

\begin{question}
\label{Q-intermediate}
Let~$\lambda\in(0,1)$ and let $D_N$ arise from a~$D\in\mathfrak D$ as above. Is there~$K_N$ depending only on~$D$ and~$\lambda$ such that
\begin{equation}
\frac1{K_N}\sum_{x\in D_N}\delta_{\,x/N}\otimes \delta_{\,h_x-2\sqrt g\,\lambda\log N} 
\end{equation}
converges in law as~$N\to\infty$ to a non-trivial random measure?
\end{question}

\noindent
The point of this question is to find out if one can capture the scaling limit of the level set of the form \eqref{E:level}, including local information, using a meaningful continuum object.  (\textit{Update in revision:} Question~\ref{Q-intermediate} has now been fully resolved in Biskup and Louidor~\cite{BL3}.)  Continuum counterparts of this question exist; e.g., based on the notion of a \emph{thick point} of the (continuum) Gaussian Free Field analyzed by Hu, Miller and Peres~\cite{HMP10}.

 The next  natural question is whether and how our results extend to the class of logarithmically correlated Gaussian fields in general dimensions $d\ge1$. Here the convergence of the law of the centered maximum has already been proved (Ding, Roy and Zeitouni~\cite{DRZ}) but, for the lack the Gibbs-Markov property in~$d\ne2$, our techniques do not apply.  (An interesting exception is the four-dimensional membrane model, see e.g.~Kurt~\cite{Kurt}, which does have both a Gibbs structure and logarithmic correlations.)  We remark that the scaling limit of the extremal process of the \DGFF{} in $d\ge3$, which is \emph{not} logarithmically correlated, has been shown to coincide, modulo an overall shift, with that of i.i.d.\ Gaussian random variables (Chiarini, Cipriani, Hazra~\cite{CCH1,CCH2,CCH3}).

Another interesting direction concerns the corresponding problem for various non-Gaussian models on~$\Z^2$ with fluctuation structure described, at large scales, by the Gaussian Free Field. This includes the gradient models with uniformly-strictly convex potentials or the local time of the simple random walk run for multiples of the cover time. Unfortunately, here even the tightness of the maximum remains open. Notwithstanding, for the local time associated with the simple random walk on a homogeneous tree, the fluctuations are, at large scales, those of a Gaussian Branching Random Walk. In this case, Abe~\cite{Abe} was able to show that the convergence of the type~\eqref{E:1.5} holds for a suitably defined process of local maxima of the local time.

\section{Field pinned to a high value}
\label{sec3}\noindent
We are ready to begin the exposition of the proofs.  In this paper we will  show that a pinned \DGFF{} naturally decomposes into the sum of independent random fields indexed by a sequence of nested domains and use it to represent the growth rate of the field by way of a random walk.

Throughout this whole section, $D$ (or similar letters) will denote a generic finite set $D\subset\Z^2$ while~$D_N$ will keep denoting the set as in \twoeqref{E:1.1}{E:1.1a} for some (notationally implicit) underlying continuum domain. We will write~$h^D$ (instead of just~$h$) to denote the \DGFF{} in~$D$ and will write~$h^D_x$ or $h^D(x)$ to denote its value at~$x$. The arguments use various standard facts about the \DGFF{} and harmonic analysis on~$\Z^2$; these are for reader's convenience collected in Appendix~\ref{appendB}.

\subsection{Simplifying the conditioning}
\label{sec-3.1}\noindent
We begin by a reduction argument. Recall that our ultimate goal is to control the position and field value at, and the ``shape'' of the configuration around, the nearly-maximal local maxima. Thanks to the Gibbs-Markov property (Lemma~\ref{lemma-GM}) and estimates on separation of near-maximal values (Lemma~\ref{lemma-separation}), it will suffice to do this just for one local maximum. The main task is thus the $N\to \infty$ asymptotic of the probability
\begin{equation}
\label{E:3.1}
P\Bigl(h^{D_N}(0)-h^{D_N}\in A\colon h^{D_N}\le m_N+t+s,\,h^{D_N}(0)\ge m_N+t\Bigr)
\end{equation}
for $t\in\R$, $s\ge0$ and events~$A$ that depend only on a finite number of coordinates near~$0$. (Here and henceforth, $h\le f$ means that $h(x)\le f(x)$ on the natural domain of~$h$.) We will instead study the conditional probability
\begin{equation}
\label{E:3.2}
P\Bigl(h^{D_N}(0)-h^{D_N}\in A,\,h^{D_N}\le m_N+t+s\,\Big|\, h^{D_N}(0)=m_N+t\Bigr).
\end{equation}
This is sufficient for \eqref{E:3.1} because the probability (density) of the conditional event is explicitly available and \eqref{E:3.1} can thus be obtained from \eqref{E:3.2} by integrating over~$t$.

\begin{remark}
The type of ``singular'' conditioning as in \eqref{E:3.2} will be used frequently throughout the rest of this paper. In all such cases it will be clear that the conditional random variable has a well-defined, continuous probability density with respect to the Lebesgue measure and so the conditioning boils down to substituting the conditional value for that variable.
\end{remark}

\smallskip
In order to control \eqref{E:3.2} note that, due to the Gaussian nature of the field, conditioning on a large value at a given point can be reduced to a shift in the mean. More precisely, given a finite~$D\subset\Z^2$ with $0\in D$ let $\frakg^D\colon\Z^2\to[0,1]$ denote the (unique) function such that
\begin{enumerate}
\item[(1)] $\frakg^D$ is discrete harmonic on $D\smallsetminus\{0\}$,
\item[(2)] $\frakg^D(0)=1$ and $\frakg^D(x)=0$ for $x\not\in D$.
\end{enumerate}
By the maximum principle,~$\frakg^D$ takes values in~$[0,1]$. The conditioning then simplifies as:

\begin{lemma}
\label{lemma-3.1}
Let $D\subset\Z^d$ be finite with $0\in D$. Then for all~$t,s\in\R$ and any event~$A$,
\begin{equation}
\label{E:3.3}
P\Bigl(h^D\in A,\,h^D\le s\,\Big|\, h^D(0)=t\Bigr)
=P\Bigl(h^D+t\frakg^D\in A,\,h^D\le s-t\frakg^D\,\Big|\, h^D(0)=0\Bigr).
\end{equation}
\end{lemma}

\begin{proofsect}{Proof}
By the Gibbs-Markov property (cf Lemma~\ref{lemma-GM}), we have 
\begin{equation}
\label{E:3.4}
h^D \laweq h^D(0)\,\frakg^D\,+h^{D\smallsetminus\{0\}},
\end{equation}
where the two fields on the right are regarded as independent. This formula shows that $h^{D\smallsetminus\{0\}}$, the \DGFF{} in $D\smallsetminus \{0\}$, is also the \DGFF{} in~$D$ conditioned on $h^D(0)=0$. Plugging these facts into the left-hand side \eqref{E:3.3}, the claim follows.
\end{proofsect}

\subsection{Concentric decomposition}
Having reduced the problem to a field pinned to zero, the next technical step is a representation of this field as a sum of independent random fields. The pinning at a single point naturally leads us to consider a decomposition along a nested sequence of domains. Since we will ultimately work with $\ell^\infty$-balls centered at the origin, we will refer to this as a \emph{concentric decomposition}.

Given a set~$B\subset\Z^2$, let $\partial B$ denote the set of vertices on its external boundary. Consider an increasing sequence of {connected} sets $\Delta^0,\Delta^1\dots, \Delta^n\subset\Z^2$ satisfying
\begin{equation}
\label{E:3.5}
\Delta^0:=\{0\}\quad \text{and}\quad \overline{\Delta^k}:=\Delta^k\cup\partial \Delta^k\subseteq \Delta^{k+1},\quad k=0,\dots,n-1.
\end{equation}
Define
\begin{equation}
\label{E:3.6}
\varphi_k(x):=
\begin{cases}
\begin{aligned}
E\bigl(h^{\Delta^n}(x)\big|\sigma(&h^{\Delta^n}(z)\colon z\in\partial \Delta^{k-1}\cup\partial \Delta^k)\bigr)
\\&-E\bigl(h^{\Delta^n}(x)\big|\sigma(h^{\Delta^n}(z)\colon z\in\partial \Delta^k)\bigr)
\end{aligned}
\,\,,\qquad&\text{}k=1,\dots,n,
\\*[4mm]
h^{\Delta^n}(x)-E\bigl(h^{\Delta^n}(x)\big|\sigma(h^{\Delta^n}(z)\colon z\ne0)\bigr),\qquad&\text{}k=0,
\end{cases}
\end{equation}
and let
\begin{equation}
\label{E:3.7}
\chi_k(x):=\varphi_k(x)-E\bigl(\varphi_k(x)\bigl|\sigma(\varphi_k(0))\bigr),\qquad k=0,\dots,n.
\end{equation}
Then we set
\begin{equation}
\label{E:3.8}
h'(x):=h^{\Delta^n}(x)-\sum_{k=0}^n\varphi_k(x)
\end{equation}
and let
\begin{equation}
\label{E:3.9}
h_k'(x):=h'(x)\1_{\Delta^k\smallsetminus \Delta^{k-1}}(x),\qquad k=0,\dots,n.
\end{equation}
All the fields above are defined on the same probability space as~$h^{\Delta^n}$ and all can be regarded as fields on all of~$\Z^2$. Their distributional properties are summarized in:

\begin{proposition}[Concentric decomposition]
\label{prop-concentric}
Suppose the domains $\{\Delta^k\colon k=0,\dots,n\}$ obey the restrictions in \eqref{E:3.5}. Then the random objects in the union
\begin{equation}
\label{E:3.10a}
\bigl\{\varphi_k(0)\colon k=0,\dots,n\bigr\}\cup\bigl\{\chi_k\colon k=0,\dots,n\bigr\}\cup\bigl\{h_k'\colon k=0,\dots,n\bigr\}
\end{equation}
are all independent of one another. Their individual laws are (multivariate) normal and they are determined by the following properties:
\settowidth{\leftmargini}{(11)}
\begin{enumerate}
\item[(1)] For $k=1,\dots,n$, the field in \eqref{E:3.6} obeys
\begin{equation}
\label{E:3.10}
\varphi_k\laweq E\bigl(h^{\Delta^k}\big|\sigma(h^{\Delta^k}\colon z\in \partial \Delta^{k-1})\bigr),\qquad k=1,\dots,n,
\end{equation}
while, for $k=0$,
\begin{equation}
\label{E:3.11}
\varphi_0=\varphi_0(0)\1_{\{0\}},\quad\text{where}\quad\varphi_0(0)\laweq \NN(0,1).
\end{equation}
A.e.\ sample path of $\varphi_k$ is discrete harmonic on $\Delta^k\smallsetminus\partial \Delta^{k-1}$ and zero on $\Z^2\smallsetminus \Delta^k$. The law of~$\chi_k$ is then determined from \eqref{E:3.7}.
\item[(2)] For $k=1,\dots,n$, for the fields in \eqref{E:3.9} we have
\begin{equation}
\label{E:3.12}
h'_k\laweq h^{\Delta^k\smallsetminus \overline{\Delta^{k-1}}}
\end{equation}
while $h'_0=0$. In particular, $h'$ in \eqref{E:3.8} has the law of the \DGFF{} in $\bigcup_{k=1}^{n} \Delta^k\smallsetminus\overline{\Delta^{k-1}}$ with, per our convention, zero boundary conditions outside of this set.
\end{enumerate}
In addition, we have $\Var(\varphi_k(0))>0$ for all $k=0,\dots,n$ and, letting 
\begin{equation}
\label{E:3.14q}
\frakb_k(x):=\frac1{\Var(\varphi_k(0))}\,E\Bigl(\varphi_k(0)\bigl(\varphi_k(x)-\varphi_k(0)\bigr)\Bigr),
\end{equation}
the following representation holds
\begin{equation}
\label{E:3.14}
h^{\Delta^n}(x)=\sum_{k=0}^n\bigl(1+\frakb_k(x)\bigr)\varphi_k(0)+\sum_{k=0}^n\chi_k(x)+\sum_{k=0}^nh_k'(x).
\end{equation}
\end{proposition}

\begin{figure}
\vglue0.0cm
\centerline{\includegraphics[width=0.8\textwidth]{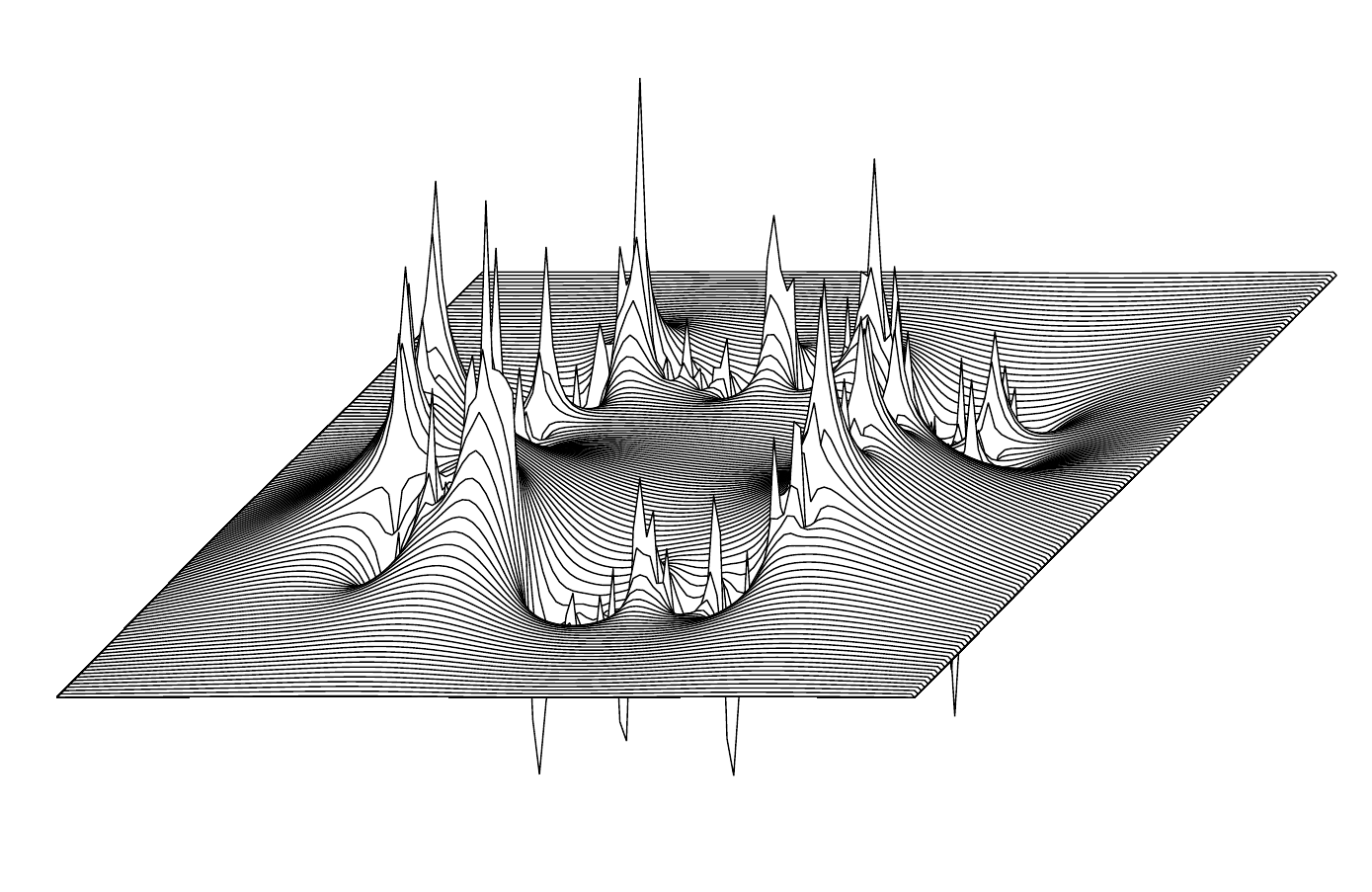}}
\begin{quote}
\small
\vglue-4mm
\caption{The graph of a random sample of~$\chi_k$ from \eqref{E:3.7} for~$\Delta^k$ being a square in~$\Z^2$ of side-length~$128$. The function $\varphi_k$ from which~$\chi_k$ is derived takes values of the DGFF~$h^{\Delta^k}$ on~$\partial\Delta^{k-1}$ and is discrete harmonic elsewhere.
\label{fig3}
}
\normalsize
\end{quote}
\end{figure}

\begin{proofsect}{Proof} Fix~$n\ge0$ and consider the $\sigma$-algebras
\begin{equation}
\FF_k:=\sigma\Bigl(h^{\Delta^n}(z)\colon z\in\bigcup_{\ell={n-k}}^n\partial \Delta^\ell\Bigr),\qquad k=0,\dots,n+1,
\end{equation}
where $\partial \Delta^{-1}:=\Delta^0$.
Obviously, $k\mapsto\FF_k$ is non-decreasing. Since $h^{\Delta^n}$ is fixed to zero on~$\partial \Delta^n$, we have $E(h^{\Delta^n}(x)\big|\FF_0)=E(h^{\Delta^n}(x))=0$ and so
\begin{equation}
\label{E:3.17}
E\bigl(h^{\Delta^n}(x)\big|\FF_{n+1}\bigr)=
\sum_{k=0}^{n}\Bigl(E\bigl(h^{\Delta^n}(x)\big|\FF_{k+1}\bigr)-E\bigl(h^{\Delta^n}(x)\big|\FF_k\bigr)\Bigr).
\end{equation}
We claim that
\begin{equation}
\label{E:3.18ww}
E\bigl(h^{\Delta^n}(x)\big|\FF_{k+1}\bigr)-E\bigl(h^{\Delta^n}(x)\big|\FF_k\bigr)=\varphi_{n-k}(x).
\end{equation}
Indeed, the Gibbs-Markov property (Lemma~\ref{lemma-GM}) tells us that $x\mapsto E(h^{\Delta^n}(x)|\FF_k)$ is the discrete-harmonic extension of the restriction of~$h^{\Delta^n}$ to the set in the definition of~$\FF_k$. The two terms on the left of \eqref{E:3.18ww} have equal extensions outside~$\Delta^{n-k}$ and so their difference vanishes there, while inside~$\Delta^{n-k}$ it is equal to the difference between the harmonic extension of~$h^{\Delta^n}$ restricted to $\partial \Delta^{n-k}\cup\partial \Delta^{n-k-1}$ and the harmonic extension of~$h^{\Delta^n}$ restricted to~$\partial \Delta^{n-k}$. Comparing with \eqref{E:3.6}, this yields \eqref{E:3.18ww}. 

Next we observe that the terms in the sum \eqref{E:3.17} are independent of one another because they are uncorrelated (being martingale increments) Gaussians.  The field $h'$ in \eqref{E:3.8} can in turn be written as $h^{\Delta^n}-E(h^{\Delta^n}(\cdot)\big|\FF_{n+1})$ and so, by the Gibbs-Markov property,  it has the law of the \DGFF{} in $\bigcup_{k=1}^{n} \Delta^k\smallsetminus\overline{\Delta^{k-1}}$. Since the \DGFF{}s in disconnected sets are independent, we get
\begin{equation}
h^{\Delta^n}(x)=\sum_{k=0}^n\varphi_k(x)+\sum_{k=0}^n h'_k
\end{equation}
with all the fields on the right independent of one another. The Gibbs-Markov property then also helps us check \twoeqref{E:3.10}{E:3.11}. 

Now define~$\chi_k$ by \eqref{E:3.7}. Then~$\varphi_k(0)$ and~$\chi_k$ are uncorrelated and thus independent, yielding the claimed independence of the family \eqref{E:3.10a}. Interpreting conditioning as a projection gives
\begin{equation}
E\bigl(\varphi_k(x)\bigl|\sigma(\varphi_k(0))\bigr)=\frakf_k(x)\varphi_k(0)
\end{equation}
for some deterministic function~$\frakf_k$. Writing $\frakf_k$ as $1+\frakb_k$, a covariance calculation shows that $\frakb_k$ must  be given by \eqref{E:3.14q} provided we can verify $\Var(\varphi_k(0))>0$. This follows from 
\begin{equation}
\label{E:3.21}
\Var\bigl(\varphi_k(0)\bigr) = \Var\bigl(h^{\Delta^k}(0)\bigr)-\Var\bigl(h^{\Delta^{k-1}}(0)\bigr)
\end{equation}
and the connectedness of~$\Delta^k$ along with $\Delta^{k-1}\subsetneq \Delta^{k}$, as implied by~\eqref{E:3.5}. (Indeed, by \eqref{E:DGFF-def} the variances are the expected numbers of returns of the simple random walk from 0 to 0 before the walk leaves the given set and, under the stated conditions, there is a finite path of the walk from~0 to~0 that stays in~$\Delta^k$ but leaves~$\Delta^{k-1}$ along the way.) 
\end{proofsect}

The decomposition \eqref{E:3.14} yields:

\begin{corollary}
Assume $\{\varphi_k(0)\colon k\ge0\}$, $\{\chi_k\colon k\ge0\}$ and $\{h_k'\colon k\ge0\}$ are independent objects  with the laws as specified in Proposition~\ref{prop-concentric} and let $\{\frakb_k\colon k\ge0\}$ be as in \eqref{E:3.14q}. Then, for each~$n\ge0$, the sum on the right-hand side of \eqref{E:3.14} has the law of the \DGFF{} in~$\Delta^n$. 
In particular, the whole family $\{h^{\Delta^n}\colon n\ge0\}$ can be constructed on a single probability space by imposing~\eqref{E:3.14} for all~$n\ge0$. Moreover, the object defined in  \eqref{E:3.6} is then given by
\begin{equation}
\label{E:3.22r}
\varphi_k(x)=\bigl(1+\frakb_k(x)\bigr)\varphi_k(0)+\chi_k(x)
\end{equation}
for all $k\ge0$ and all~$x\in\Z^2$.
\end{corollary}

\begin{proofsect}{Proof}
As is easily checked, the laws of the random variables $\varphi_k(0)$ and random fields~$\chi_k$ and~$h_k'$ as well as the function~$\frakb_k$ depend only on $\Delta^{k-1}$ and~$\Delta^k$ and, in particular, have no explicit dependence on~$n$. Hence, \eqref{E:3.14} can be imposed simultaneously for all~$n\ge0$. The final argument in the proof of Proposition~\ref{prop-concentric} then yields \eqref{E:3.22r}.
\end{proofsect}

We will henceforth assume that all random variables in \eqref{E:3.10a} are realized on the same probability space (as independent with laws as above) and the fields $\{h^{\Delta^n}\colon n\ge0\}$, resp., $\{\varphi_n\colon n\ge0\}$ are derived from these via \eqref{E:3.14}, resp.,~\eqref{E:3.22r}.

The sequence $\{\varphi_k(0)\colon k\ge0\}$ will be central to our arguments and that particularly so via the sequence of its partial sums
\begin{equation}
\label{E:RW}
S_k:=\sum_{\ell=0}^{k-1}\varphi_\ell(0),\qquad k\ge0.
\end{equation} 
Here it is important to note that, in light of
\begin{equation}
\frakb_k(0)=0 \quad\text{as well as}\quad \chi_k(0)=0\quad\text{and}\quad h_k'(0)=0\quad\text{ a.s.},\qquad k\ge0,
\end{equation}
we have~$S_k=h^{\Delta^{k-1}}(0)$ for all $k\ge1$.  
The independence of $\{\varphi_k(0)\colon k\ge0\}$ suggests to regard $\{S_k\colon k\ge1\}$ as a path of a \emph{random walk}, albeit with time-inhomogeneous steps. For the conditional event in \eqref{E:3.3} we then get
\begin{equation}
\label{E:3.16qq}
h^{\Delta^n}(0)=0\quad\Leftrightarrow\quad S_{n+1}=0.
\end{equation} 
The conditioning on $h^{\Delta^n}(0)=0$ thus amounts to the requirement that the random walk be back to its starting point after~$n+1$ steps.

\begin{remark}
The above decomposition of the \DGFF{} can be regarded as a concentric analogue of the finite-range decomposition of the (homogeneous) Gaussian Free Field (discrete or continuum) used in renormalization-based treatments of interacting random fields (Brydges~\cite{Brydges-Park-City}). There is also some vague analogy between~$\{S_n\}$ and the random walk that arises in the spine decomposition of a Branching Random Walk (e.g., A{\"{\i}}d\'ekon~\cite{Aidekon}).  The review~\cite{Biskup-PIMS} contains a thorough discussion of this connenction. 
\end{remark}

\subsection{Useful estimates}
Our next task is to represent the overall growth of $h^{\Delta^n}$ in terms of the random walk \eqref{E:RW}. This will require rather tight control of the various terms on the right-hand side of \eqref{E:3.14} and, for convenience, a specific choice of the domains in \eqref{E:3.5}. Therefore, we will set
\begin{equation}
\label{E:3.25}
\Delta^k:=\bigl\{x\in\Z^2\colon |x|_\infty\le 2^k\bigr\},\qquad k\ge1,
\end{equation}
for the remainder of this paper. (Generalizations when~$\Delta^n$ is replaced by a more general domain of the same  spatial  scale will be discussed in Section~\ref{sec-4.5}.)
Our discussion of the behavior of the objects entering \eqref{E:3.14} starts with the random variables~$\varphi_k(0)$:

\begin{lemma}
\label{lemma-3.4}
We have
\begin{equation}
\inf_{k\ge1}\,\Var\bigl(\varphi_k(0)\bigr)>0.
\end{equation}
Moreover,
\begin{equation}
\label{E:3.20}
\lim_{k\to\infty}\Var\bigl(\varphi_k(0)\bigr)=g\log 2.
\end{equation}
\end{lemma}

\begin{proofsect}{Proof}
The strict positivity of $\Var\bigl(\varphi_k(0)\bigr)$ for each~$k\ge0$ was proved in Proposition~\ref{prop-concentric} so it suffices to show \eqref{E:3.20}. This follows from \eqref{E:3.21}, the representation  of  $\Var(h^{\Delta^k}(0))$ as the Green function $G^{\Delta^k}(0,0)$, the definition of~$\Delta^k$ and the asymptotic $G^{\Delta^k}(0,0)=g\log(2^k)+c_0+o(1)$ as $k\to\infty$ for some constant~$c_0$. (This form is derived using Lemmas~\ref{lemma-G-potential}--\ref{lemma-potential}.)
\end{proofsect}

Next we will address the behavior of function $\frakb_k$:

\begin{lemma}
\label{lemma-3.5rt}
Let $k\in\{0,1,\dots\}$. Then $\frakb_k$ is discrete harmonic in~$\Z^2\smallsetminus(\partial \Delta^k\cup\partial \Delta^{k-1})$ and bounded uniformly in~$k$. In addition, we have $\frakb_k(x)\ge-1$ for all~$x\in\Z^2$,
\begin{equation}
\label{E:3.29}
\frakb_k(x)=-1,\qquad x\not\in \Delta^k,
\end{equation}
and, for some~$c\in(0,\infty)$ independent of~$k$ and ``$\dist$'' denoting the $\ell^\infty$-distance on~$\Z^2$,
\begin{equation}
\label{E:3.22q}
\bigl|\frakb_k(x)\bigr|\le c\frac{\dist(0,x)}{\dist(0,\partial \Delta^{k})},\qquad x\in \Delta^k.
\end{equation}
\end{lemma}

\begin{proofsect}{Proof}
The stated discrete harmonicity of~$\frakb_k$, resp., \eqref{E:3.29} are directly checked from the definition in \eqref{E:3.14q}, resp., \eqref{E:3.10}. To get that $1+\frakb_k\ge0$ it suffices to check that~$\varphi_k$ has positive correlations. This follows either directly from the strong-FKG property (cf Lemma~\ref{lemma-FKG}) or by 
\begin{equation}
\label{E:phi-diff}
E\bigl(\varphi_k(x)\varphi_k(0)\bigr)=E\bigl(h^{\Delta^k}(x)h^{\Delta^k}(0)\bigr)-E\bigl(h^{\Delta^{k-1}}(x)h^{\Delta^{k-1}}(0)\bigr)
\end{equation}
combined with the argument at the end of the proof of Proposition~\ref{prop-concentric}. The uniform boundedness of~$\frakb_k$ follows from \eqref{E:phi-diff} and Lemmas~\ref{lemma-G-potential}--\ref{lemma-potential}.

It remains to prove \eqref{E:3.22q}. Clearly, by our choice of~$\Delta^k$ and the boundedness of~$\frakb_k$, we only need to do this for~$x\in \Delta^{k-2}$. Writing $H^D(x,y)$ --- where $D\subset\Z^2$, $x\in D$ and $y\in\partial D$ --- for the harmonic measure for the simple random walk, the harmonicity of~$\frakb_k$ in~$\Delta^{k-1}$ and $\frakb_k(0)=0$ yield
\begin{equation}
\frakb_k(x)=\sum_{y\in\partial \Delta^{k-1}}\bigl[H^{\Delta^{k-1}}(x,y)-H^{\Delta^{k-1}}(0,y)\bigr]\frakb_k(y).
\end{equation}
By Lemma~\ref{lemma-HM}, we have
\begin{equation}
\label{E:3.32w}
\bigl|H^{\Delta^{k-1}}(x,y)-H^{\Delta^{k-1}}(x',y)\bigr|\le c\frac{\dist(x,x')}{\dist(0,\partial \Delta^{k-1})}\frac1{|\partial \Delta^{k-1}|},\qquad x,x'\in \Delta^{k-2}.
\end{equation}
for some constant $c\in(0,\infty)$.
Using that $\frakb_k$ is bounded on~$\partial \Delta^{k-1}$, and noting that the diameter of~$\Delta^k$ is proportional to that of~$\Delta^{k-2}$, we readily infer \eqref{E:3.22q}.
\end{proofsect}

\begin{figure}[t]
\vglue0.2cm
\centerline{\includegraphics[width=0.8\textwidth]{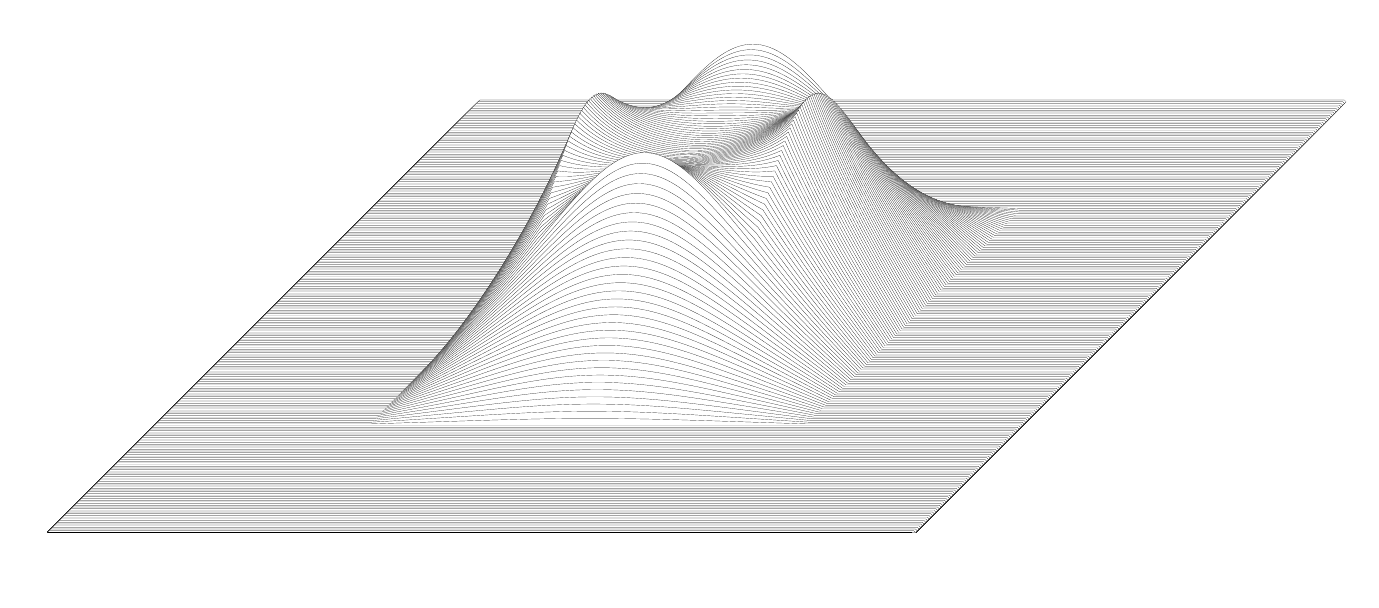}}
\begin{quote}
\small
\vglue0mm
\caption{The graph of function~$\frakb_k$ for~$k:=7$ (and~$\Delta^k$ as in \eqref{E:3.25}). Notice that $\frakb_k$ does attain strictly positive values, but only in the vicinity of~$\partial\Delta^{k-1}$. (The value outside $\Delta^k$ is negative one.)
\label{fig5}
}
\normalsize
\end{quote}
\end{figure}

Our next item of concern is the law of~$\chi_k$. Since~$\chi_k$ is a.s.\ discrete harmonic on~$\Delta^{k-1}$, its control on~$\Delta^\ell$ for $\ell\le k-2$ is fairly straightforward:

\begin{lemma}
\label{lemma-3.6a}
There is a constant~$c\in(0,\infty)$ such that for all~$k\in\N$ and all~$\ell$ with $0\le\ell\le k-2$,
\begin{equation}
\label{E:3.33}
E\Bigl(\,\max_{x\in \Delta^\ell}\chi_k(x)\Bigr)\le c2^{\ell-k}
\end{equation}
and, for some constant $c'\in(0,\infty)$ and all~$\lambda>0$,
\begin{equation}
\label{E:3.34}
P\biggl(\,\Bigl|\,\max_{x\in \Delta^\ell}\chi_k(x)-E\Bigl(\,\max_{x\in \Delta^\ell}\chi_k(x)\Bigr)\Bigr|>\lambda\biggr)
\le 2\texte^{-c'4^{k-\ell}\lambda^2}.
\end{equation}
We have $\chi_k(x)=0$ a.s.\ whenever $x\not\in \Delta^k$.
\end{lemma}

\begin{proofsect}{Proof}
We start with \eqref{E:3.33}. The sample paths of~$\chi_k$ are discrete harmonic on~$\Delta^{k-1}$ and obey $\chi_k(0)=0$. Using the notation from the previous proof, for each~$x\in \Delta^{k-2}$,
\begin{equation}
\chi_k(x)=\sum_{z\in\partial \Delta^{k-2}}\bigl[H^{\Delta^{k-2}}(x,z)-H^{\Delta^{k-2}}(0,z)\bigr]\chi_k(z).
\end{equation}
In light of \eqref{E:3.32w}, we get
\begin{equation}
\label{E:3.36}
\max_{x\in \Delta^\ell}\chi_k(x)\le c2^{\ell-k}\max_{x\in \partial \Delta^{k-2}}\chi_k(x),\qquad \ell=0,\dots,k-2.
\end{equation}
It thus suffices to show that the expected maximum on the right is uniformly bounded in~$k$. 

By independence of~$\varphi_k(0)$ and~$\chi_k$ and the stated harmonicity, for all $x,y\in \Delta^{k-1}$,
\begin{equation}
\begin{aligned}
E\bigl(|\chi_k(x)&-\chi_k(y)|^2\bigr)\le E\bigl(|\varphi_k(x)-\varphi_k(y)|^2\bigr)
\\
&=\sum_{z,z'\in\partial \Delta^{k-1}}
\bigl[H^{\Delta^{k-1}}(x,z)-H^{\Delta^{k-1}}(y,z)\bigr]
\bigl[H^{\Delta^{k-1}}(x,z')-H^{\Delta^{k-1}}(y,z')\bigr]
G^{\Delta^k}(z,z').
\end{aligned}
\end{equation}
Invoking the standard asymptotic form of the Green function, the average of $G^{\Delta^k}(z,z')$ over the points $z,z'\in\partial \Delta^{k-1}$ is bounded uniformly in~$k$. From  \eqref{E:3.32w} we thus conclude
\begin{equation}
\label{E:3.38}
E\bigl(|\chi_k(x)-\chi_k(y)|^2\bigr)\le c\Bigl(\frac{|x-y|}{2^k}\Bigr)^2,\qquad x,y\in \overline{\Delta^{k-1}}
\end{equation}
with~$c$ independent of~$k$. Using the Fernique criterion (cf~Lemma~\ref{lemma-Fernique}) with the counting measure on~$\overline{\Delta^{k-2}}$ as the majorization measure, the expected maximum of $\chi_k$ on $\overline{\Delta^{k-2}}$ is found to be bounded uniformly in~$k$. By \eqref{E:3.36}, we get \eqref{E:3.33}.

To get \eqref{E:3.34}, we note that (since $\chi_k(0)=0$) the bound \eqref{E:3.38} shows that the variance of $\chi_k(x)$ is bounded uniformly in~$x\in\overline{\Delta^{k-1}}$. By the argument leading to \eqref{E:3.36}, we thus get
\begin{equation}
\max_{x\in \Delta^\ell}\Var\bigl(\chi_k(x)\bigr)\le c 2^{2(\ell-k)},\qquad \ell=0,\dots,k-2.
\end{equation}
The Borell-Tsirelson inequality (cf~Lemma~\ref{lemma-BT}) then yields \eqref{E:3.34}.
\end{proofsect}

Unfortunately, $\chi_k$ is not discrete harmonic on~$\partial \Delta^{k-1}\cup\partial \Delta^k$, and so its control on $\Delta^k\smallsetminus \Delta^{k-2}$ requires a different argument. We will rely on the fact that, by combining $\chi_k$ with~$h_k'$ we get (more or less) the \DGFF{} in~$\Delta^k\smallsetminus \Delta^{k-2}$. Let us therefore first address the tails of~$h_k'$.

\begin{lemma}
\label{lemma-3.6}
Recall the notation~$m_N$ for the quantity from \eqref{E:mN}. There are constants~$c_1,c_2\in(0,\infty)$ such that for all~$k\ge1$ and all~$\lambda>0$,
\begin{equation}
\label{E:3.40}
P\Bigl(\,\bigl|\,\max_{x\in \Delta^k\smallsetminus{\Delta^{k-1}}}h_k'(x)-m_{2^k}\bigr|>\lambda\Bigr)
\le c_1\texte^{-c_2\lambda}.
\end{equation}
(We have $h'_k(x)=0$ a.s.\ outside the annulus $\Delta^k\smallsetminus{\Delta^{k-1}}$.) As $k\to\infty$ the joint law of
\begin{equation}
\label{E:3.41}
\bigl(2^{-k}\argmax_{\,\Delta^k\smallsetminus{\Delta^{k-1}}} \,h_k', \max_{x\in \Delta^k\smallsetminus{\Delta^{k-1}}}h_k'(x)-m_{2^k}\bigr)
\end{equation}
tends to a non-degenerate distribution on~$(-1,1)^2\times\R$.
\end{lemma}

\begin{proofsect}{Proof}
A bound of the form \eqref{E:3.40} has been proved for square domains in Ding and Zeitouni~\cite{DZ}; cf~Lemma~\ref{lemma-max-tail}. Concerning the maximum in the annuli $\Delta^k\smallsetminus \Delta^{k-1}$ we apply the bounds in Lemma~\ref{lemma-stoch-order} along with the fact that $N\mapsto m_N$ is slowly varying and that $\Delta^k\smallsetminus \Delta^{k-1}$ contains and is contained in a square of side of order~$2^k$.

The convergence in \eqref{E:3.41} has been proved in Bramson, Ding and Zeitouni~\cite{BDingZ}. (Strictly speaking, \cite{BDingZ} proves only the convergence in law for the value of the maximum in square-like domains. The above joint convergence can be gleaned from their proofs and/or can be found in Biskup and Louidor~\cite{BL1,BL2}, with~\cite{BL2} addressing general domains.  Alternative proofs of both tightness and distributional convergence have been given in \cite{Biskup-PIMS}. These rely strongly on the techniques developed in the present paper.) 
\end{proofsect}

We are now ready to deal with $\chi_k$ on $\Delta^k\smallsetminus \Delta^{k-2}$ although, as mentioned before, not without some help from $h'_k$:

\begin{lemma}
\label{lemma-3.7}
There are constants $c_1,c_2\in(0,\infty)$ such that for all $k\ge1$ and all $\lambda>0$,
\begin{equation}
\label{E:3.43}
P\Bigl(\,\bigl|\,\max_{x\in \Delta^{k}\smallsetminus \Delta^{k-1}}\bigl(\chi_k(x)+\chi_{k-1}(x)+h_k'(x)\bigr)-m_{2^k}\bigr|>\lambda\Bigr)
\le c_1\texte^{-c_2\lambda}.
\end{equation}
Similarly, we also have
\begin{equation}
\label{E:3.44}
P\Bigl(\,\bigl|\,\max_{x\in \Delta^{k-1}\smallsetminus \Delta^{k-2}}\bigl(\chi_k(x)+\chi_{k-1}(x)+\chi_{k-2}(x)+h_{k-1}'(x)\bigr)-m_{2^k}\bigr|>\lambda\Bigr)
\le c_1\texte^{-c_2\lambda}.
\end{equation}
\end{lemma}

\begin{proofsect}{Proof}
Focusing first on the values in~$\Delta^k\smallsetminus \Delta^{k-1}$, the representation \eqref{E:3.14} shows that the difference between the field in the maximum and $h^{\Delta^k}$ equals
\begin{multline}
\qquad
h^{\Delta^k}(x)-\bigl[\chi_k(x)+\chi_{k-1}(x)+h_k'(x)\bigr]
\\
=\bigl(1+\frakb_k(x)\bigr)\varphi_k(0)+\bigl(1+\frakb_{k-1}(x)\bigr)\varphi_{k-1}(0).
\qquad
\end{multline}
By Lemmas~\ref{lemma-3.4}--\ref{lemma-3.5rt}, the field on the right-hand side has uniform Gaussian tails which converts \eqref{E:3.43} to a bound on the tails of $\max_{x\in \Delta^{k}\smallsetminus \Delta^{k-1}}h^{\Delta^k}(x)-m_{2^k}$. This bound is obtained by another reference to Lemma~\ref{lemma-max-tail}, combined also with Lemma~\ref{lemma-stoch-order} in the bound of the lower tail. The proof of \eqref{E:3.44} is completely analogous.
\end{proofsect}

The last matter to address concerns a rewrite of the event $h^{\Delta^n}\le (m_{2^n}+t)(1-\frakg^{\Delta^n})$ in terms of the random objects $\varphi_k(0)$, $\chi_k$ and~$h_k'$. For this we need to compare the quantity on the right to the natural growth of the maximum of~$h'_k$ which, as stated in Lemma~\ref{lemma-3.6}, is captured by the sequence~$m_{2^k}$. Here we note:

\begin{lemma}
\label{lemma-3.8}
There is a constant~$c\in(0,\infty)$ such that for all $n\ge1$ and all $k=0,\dots,n$,
\begin{equation}
\max_{x\in \Delta^k\smallsetminus \Delta^{k-1}}\Bigl| \,m_{2^n}\bigl(1-\frakg^{\Delta^n}(x)\bigr)-m_{2^k}\Bigr|\le c+\frac34\sqrt g\,\,\log\bigl(1+k\wedge(n-k)\bigr).
\end{equation}
\end{lemma}

\begin{proofsect}{Proof}
We have
\begin{equation}
1-\frakg^{\Delta^n}(x)=\frac{G^{\Delta^n}(0,0)-G^{\Delta^n}(0,x)}{G^{\Delta^n}(0,0)}.
\end{equation}
The Green function admits the representation \eqref{E:B.4} in Lemma~\ref{lemma-G-potential} which, using the asymptotic form of the potential from Lemma~\ref{lemma-potential}, some straightforward calculations imply
\begin{equation}
1-\frakg^{\Delta^n}(x)=\frac{k}n+O\Bigl(\frac1n\Bigr),\qquad x\in \Delta^k\smallsetminus \Delta^{k-1},
\end{equation}
with the implicit constant in the error term uniform in~$k$. Multiplying by $m_{2^n}$ and invoking the explicit form \eqref{E:mN} then yields the result.
\end{proofsect}

\subsection{Whole-space pinned field}
The estimates derived above permit us to make a connection to the pinned \DGFF{} on all of~$\Z^2$ and, in fact, realize this field on the same probability space as the objects $\varphi_k$, $\chi_k$ and~$h_k'$:

\begin{proposition}
\label{prop-3.11}
Suppose $\{\varphi_k(0)\colon k\ge0\}$, $\{\chi_k\colon k\ge0\}$ and $\{h_k'\colon k\ge0\}$ are independent objects with the laws as specified in Proposition~\ref{prop-concentric} and let $\{\frakb_k\colon k\ge0\}$ be as in \eqref{E:3.14q}. Then for each~$x\in\Z^2$, the infinite sum
\begin{equation}
\label{E:3.28}
\phi(x):=\sum_{k\ge0}\Bigl(\frakb_k(x)\varphi_k(0)+\chi_k(x)+h_k'(x)\Bigr)
\end{equation}
converges absolutely almost surely. Moreover, $\phi$ has the law~$\nu^0$; i.e., that of the (mean-zero) \DGFF{} in $\Z^2\smallsetminus\{0\}$ or, the \DGFF{} on~$\Z^2$ pinned to zero at~$x=0$.
\end{proposition}

\begin{proofsect}{Proof}
By standard Gaussian bounds and \eqref{E:3.20}, $k\mapsto\varphi_k(0)$ grows at most polylogarithmically while, by Lemma~\ref{lemma-3.6a}, $k\mapsto\chi_k(x)$ decays exponentially. The value $h_k'(x)$ is (for a given~$x$) non-zero only for one~$k$. In light of \eqref{E:3.22q}, the sum in \eqref{E:3.28} thus converges almost surely. 

In order to identify the law of the resulting field we observe (e.g., by comparison of covariances) that the \DGFF{} in~$\Z^2\smallsetminus\{0\}$ is the weak limit of $h^{\Delta^n}-h^{\Delta^n}(0)$ (as $n\to\infty$) which, by \eqref{E:3.4} and the fact that $\frakg^{\Delta^n}\to1$ pointwise, coincides with the weak limit of $h^{\Delta^n}$ conditioned on $h^{\Delta^n}(0)=0$. Now \eqref{E:3.16qq} yields
\begin{equation}
h^{\Delta^n}(x)=\sum_{k=0}^n\Bigl(\frakb_k(x)\varphi_k(0)+\chi_k(x)+h_k'(x)\Bigr)
\qquad\text{on }\bigl\{h^{\Delta^n}(0)=0\bigr\},
\end{equation}
which is close to \eqref{E:3.28} except for one fact: conditioning on $h^{\Delta^n}(0)=0$ changes the law of the variables $\{\varphi_k(0)\colon k=0,\dots,n\}$. 

To account for this change, note that, by \eqref{E:3.16qq} again and the Gaussian nature of all variables, conditional on $h^{\Delta^n}(0)=0$, the law of $\{\varphi_k(0)\colon k=0,\dots,n\}$ is that of
\begin{equation}
\biggl\{\varphi_k(0)-c_n(k)\sum_{\ell=0}^n\varphi_\ell(0)\colon k=0,\dots,n\biggr\}
\end{equation}
under the unconditioned (product) measure, where
\begin{equation}
\label{E:3.33qq}
c_n(k):=\Var\bigl(\varphi_k(0)\bigr)\biggl(\,\sum_{\ell=0}^n\Var\bigl(\varphi_k(0)\bigr)\biggr)^{-1}.
\end{equation}
Neither $\chi_k$ nor~$h_k'$ are affected by the conditioning, being independent of the $\varphi_k(0)$'s, and so\begin{equation}
\label{E:3.31}
\bigl(h^{\Delta^n}\bigl|h^{\Delta^n}(0)=0\bigr) \laweq  \widetilde\phi_n,
\end{equation}
where
\begin{equation}
\label{E:3.32}
\widetilde\phi_n(x):=\sum_{k=0}^n\Bigl(\frakb_k(x)\varphi_k(0)+\chi_k(x)+h_k'(x)\Bigr)\,-\,\biggl(\,\sum_{k=0}^n \frakb_k(x)c_n(k)\biggr)\biggl(\,\sum_{\ell=0}^n\varphi_\ell(0)\biggr)
\end{equation}
with all the variables distributed as under the unconditioned measure.

Taking~$n\to\infty$, the first term on the right of \eqref{E:3.32} tends to the infinite series \eqref{E:3.28} a.s.\ by our observations above.  For the second term, by \eqref{E:3.20} and the absolute summability of the family of numbers $\{\frakb_k(x)\colon k\ge0\}$, the first sum is of order $n^{-1}$. Since~$\{\varphi_k(0)\colon k\ge0\}$ are independent, Gaussian with mean zero and bounded variances, as $n\to\infty$ the whole second term in \eqref{E:3.32} tends to zero a.s.\ by the Law of Large Numbers. 
\end{proofsect}

\begin{remark}
Since the law $\nu^0$ of the pinned DGFF~$\phi$ is explicitly known (see \twoeqref{E:2.8}{E:2.8u}), we could have perhaps considered checking \eqref{E:3.28} directly by comparing covariances. Notwithstanding, we still find the above proof more illuminating; particularly, since we will re-use some of its arguments later.
\end{remark}

\section{Reduction to a random walk}
\label{sec4}\noindent
The goal of this section is to rewrite the event $\{h^{\Delta^n}\le(m_{2^n}+t)(1-\frakg^{\Delta^n})\}$ from \eqref{E:3.3}, as well as its $\Z^2$-counterpart $\{\phi\le\frac2{\sqrt g}\fraka\text{ in }\Delta^n\}$, using the above random walk and well-behaved correction terms. The control of the resulting events for the random walk will require calculations for Brownian motion that are relegated to Appendix~A. Although the random walk estimates are key for most subsequent derivations, as far as the main line of the proof is concerned, the main conclusions  come in Lemmas~\ref{lemma-4.7}--\ref{lemma-4.20}.

\subsection{Control variables}
\label{sec-4.1}\noindent
We begin with the harder of the two events, $\{h^{\Delta^n}\le(m_{2^n}+t)(1-\frakg^{\Delta^n})\}$. The starting point is a definition of a random variable that will help us control the growth of the quantities~$\varphi_k(0)$,~$\chi_k$ and~$h_k'$. To express the requisite errors, for $k,\ell=1,\dots,n$ denote
\begin{equation}
\label{E:4.1}
\Theta_{k}(\ell):=\Bigl[\log\bigl(1+[k\vee(\ell\wedge(n-\ell))]\bigr)\Bigr]^2.
\end{equation}
Note that $k\mapsto\Theta_{k}(\ell)$ is increasing for each~$\ell$. The quantity also depends on~$n$ but we will keep that notationally suppressed. We then pose:

\begin{definition}
\label{def4.1}
Let~$K$ be the minimal~$k\in\{2,\dots,\lfloor \ffrac n2\rfloor\}$ such that the following holds:
\begin{enumerate}
\item[(1)] For each~$\ell=0,\dots,n$,
\begin{equation}
\bigl|\varphi_\ell(0)\bigr|\le\Theta_{k}(\ell),
\end{equation}
\item[(2)] for each $\ell=2,\dots,n$ and each $r=0,\dots,\ell-2$,
\begin{equation}
\max_{x\in \Delta^r}\,\bigl|\chi_\ell(x)\bigr|\le 2^{(r-\ell)/2}\Theta_{k}(\ell),
\end{equation}
\item[(3)] for each $\ell=1,\dots,n$,
\begin{equation}
\label{E:4.4b}
\Bigl|\,\max_{x\in \Delta^\ell\smallsetminus \Delta^{\ell-1}}\bigl(\chi_\ell(x)+\chi_{\ell-1}(x)+h_\ell'(x)\bigr)-m_{2^\ell}\Bigr|\le\Theta_{k}(\ell)
\end{equation}
and
\begin{equation}
\label{E:4.4}
\Bigl|\,\max_{x\in \Delta^\ell\smallsetminus \Delta^{\ell-1}}\bigl(\chi_\ell(x)+\chi_{\ell-1}(x)+\chi_{\ell+1}(x)+h_\ell'(x)\bigr)-m_{2^\ell}\Bigr|\le\Theta_{k}(\ell).
\end{equation}
\end{enumerate}
If no such~$k$ exists, then we set $K:=\lfloor\ffrac n2\rfloor+1$.
\end{definition}

In the default case (i.e., $K:=\lfloor\ffrac n2\rfloor+1$) no explicit bound on the above quantities can be assumed. However, this comes at little loss since we have:

\begin{lemma}
\label{lemma-4.2u}
There are $c>0$ and $k_0\ge2$ such that
\begin{equation}
\label{E:4.5}
P(K=k|S_{n+1}=0)\le\texte^{-c(\log k)^2},\qquad k=k_0,\dots,\lfloor\ffrac n2\rfloor+1.
\end{equation}
In particular, for each $\delta\in(0,1)$ and all~$n$ sufficiently large,
\begin{equation}
\label{E:4.6}
P(K>n^\delta|S_{n+1}=0)\le n^{-2}.
\end{equation}
\end{lemma}

\begin{proofsect}{Proof}
Recall (from the proof of Proposition~\ref{prop-3.11}) that, conditional on~$S_{n+1}=0$, the law of $\{\varphi_k(0)\colon k=0,\dots,n\}$ is that of $\{\varphi_k(0)-c_n(k)S_{n+1}\colon k=0,\dots,n\}$ under the unconditional measure, while $\chi_k$ and~$h_k'$ are not affected by the conditioning. By Lemmas~\ref{lemma-3.4} and~\ref{lemma-3.6a}, the numbers $\{nc_n(k)\colon k=0,\dots,n\}$ are uniformly bounded in both~$k$ and~$n$ and the random variables
\begin{equation}
\bigl\{|\varphi_k(0)-c_n(k)S_{n+1}|\colon k=0,\dots,n\bigr\}\quad\text{and}\quad
\bigl\{2^{\ell-r}\max_{x\in \Delta^r}|\chi_\ell(x)|\colon 0\le r\le\ell-2,\ell\le n\bigr\}
\end{equation}
thus have uniform (in all indices involved) Gaussian tails. Lemma~\ref{lemma-3.6} in turn ensures that the random variables on the left of \eqref{E:4.4} have a uniform exponential tail. 

Using the union bound, the probability that, in the conditional ensemble, condition~(1) of Definition~\ref{def4.1} fails is bounded by $2\sum_{\ell\ge0}\texte^{-c[\log(\ell\vee k)]^4}$, while the probability that~(3) fails is bounded by a similar expression with~$4$ replaced by~$2$ in the exponent of the logarithm. The probability that condition~(2) fails is in turn bounded by
\begin{equation}
\sum_{\ell=2}^\infty\sum_{r=0}^{\ell-2}\texte^{-c2^{\ell-r}[\log(\ell\vee k)]^4}\le \sum_{i\ge2}\sum_{\ell\ge2}\texte^{-c2^i[\log(\ell\vee k)]^4}.
\end{equation}
All of these error bounds combined yield no more that $\texte^{-c(\log k)^2}$ for a suitable~$c>0$ as soon as~$k$ is sufficiently large. This proves \eqref{E:4.5}; the bound \eqref{E:4.6} is then immediate.
\end{proofsect}

When the control variable is not defaulted to the maximal value, the above definitions ensure a rather tight control between the maximum of the field in the annuli $\Delta^k\smallsetminus \Delta^{k-1}$ and a corresponding increment of the above random walk:

\begin{lemma}[Approximation by a random walk]
\label{lemma-4.3}
There is a constant~$C\in(0,\infty)$ such that if $K\le\lfloor\ffrac n2\rfloor$ holds for some~$n\ge1$, then for each $k=0,\dots,n$ we have
\begin{equation}
\label{E:4.5ww}
\Bigl|\,\max_{x\in \Delta^k\smallsetminus \Delta^{k-1}}\bigl[h^{\Delta^n}(x)-m_{2^n}(1-\frakg^{\Delta^n}(x))\bigr]-(S_{n+1}-S_k)\Bigr|
\le R_{K}(k),
\end{equation}
where $R_{k}(\ell):=C[1+\Theta_{k}(\ell)]$.
\end{lemma}

\begin{proofsect}{Proof}
 When~$k=0$, the claim follows directly from $h^{\Delta^n}(0)=S_{n+1}$ and $\frakg^{\Delta^n}(0)=1$ so let  $k\in\{1,\dots,n\}$  and pick~$x\in \Delta^k\smallsetminus \Delta^{k-1}$. Then \eqref{E:3.9}, \eqref{E:3.29} and the last clause in Lemma~\ref{lemma-3.6a} imply
\begin{multline}
\qquad
h^{\Delta^n}(x)=S_{n+1}-S_k+\Bigl(\,\sum_{\ell=k}^n\frakb_\ell(x)\varphi_\ell(0)\Bigr)+\Bigl(\,\sum_{\ell=k+2}^n\chi_\ell(x)\Bigr)
\\+\chi_{k-1}(x)+\chi_k(x)+\chi_{k+1}(x)+h_{k}'(x),
\qquad
\end{multline}
where $\chi_{k+1}(x)$ is to be dropped when~$k=n$.
It follows from the definition of~$K$ and some elementary calculations that each of the two sums are bounded (in absolute value, uniformly in above~$x$) by a quantity of the form $C[1+\Theta_K(k)]$. An analogous bound holds also for the maximum of
\begin{equation}
\chi_{k-1}(x)+\chi_k(x)+\chi_{k+1}(x)+h_{k}'(x)-m_{2^{k}},
\end{equation}
again with $\chi_{k+1}$ is dropped when~$k=n$.
The claim follows by replacing $m_{2^n}(1-\frakg^{\Delta^n}(x))$ by~$m_{2^{k}}$, which causes an error of the form $C[1+\Theta_K(k)]$ as shown in Lemma~\ref{lemma-3.8}.
\end{proofsect}

The upshot of Lemma~\ref{lemma-4.3} is that that the event $\{h^{\Delta^n}\le(m_{2^n}+t)(1-\frakg^{\Delta^n})\}$ can be approximated by events defined solely in terms of the random walk $\{S_k\colon k=1,\dots,n+1\}$ and the control variable~$K$. Explicitly, we have:
\begin{multline}
\qquad
\label{E:4.10}
\bigl\{h^{\Delta^n}\le(m_{2^n}+t)(1-\frakg^{\Delta^n})\bigr\}\cap\bigl\{K\le\lfloor\ffrac n2\rfloor\bigr\}
\cap\bigl\{h^{\Delta^n}(0)=0\bigr\}\\
\subseteq\,\,\{S_{n+1}=0\}\cap\bigcap_{k=1}^{n}\bigl\{S_k\ge -R_{K}(k)-|t|\bigr\}
\qquad
\end{multline}
and
\begin{multline}
\qquad
\label{E:4.11}
\bigl\{h^{\Delta^n}\le(m_{2^n}+t)(1-\frakg^{\Delta^n})\bigr\}\cap\bigl\{h^{\Delta^n}(0)=0\bigr\}\\
\supseteq\,\,\bigl\{K\le\lfloor\ffrac n2\rfloor\bigr\}\cap\{S_{n+1}=0\}
\cap\bigcap_{k=1}^{n}\bigl\{S_k\ge R_{K}(k)+|t|\bigr\},
\qquad
\end{multline}
where we also recalled that $0\le1-\frakg^{\Delta^n}\le1$.\

\smallskip
Our control of the event $\{\phi\le\frac2{\sqrt g}\fraka\text{ in }\Delta^n\}$ will be quite similar, and will use only slightly modified versions of the above concepts. For $k,\ell\ge1$, let
\begin{equation}
\widetilde\Theta_k(\ell):=\bigl[\log(1+k\vee\ell)\bigr]^2,
\end{equation}
which we can view as the $n\to\infty$ limit of~$\Theta_{k}(\ell)$ above. For the analogue of the control variable~$K$, we put forward the following definition:

\begin{definition}
\label{def4.2}
Let~$\widetilde K$ be the smallest~$k\ge2$ such that, for all~$n\ge2$, the bounds in Definition~\ref{def4.1} hold with~$\Theta_{k}$ replaced by~$\widetilde\Theta_k$. (If no such~$k$ exists, we set $\widetilde K:=\infty$.)
\end{definition}

The tails of the control variable~$\widetilde K$ are easy to bound:

\begin{lemma}
There is $c>0$ and $k_0\ge2$ such that
\begin{equation}
\label{E:4.16ueu}
P(\widetilde K\ge k)\le\texte^{-c(\log k)^2},\qquad k\ge k_0.
\end{equation}
In particular, $\widetilde K<\infty$ almost surely.
\end{lemma}

\begin{proofsect}{Proof}
The proof is nearly identical to that of Lemma~\ref{lemma-4.2u}; in fact, it is much easier due to the absence of conditioning on~$S_{n+1}=0$. 
\end{proofsect}

 Definition~\ref{def4.2} then readily implies: 

\begin{lemma}
\label{lemma-4.6a}
Let~$\phi$ be as in \eqref{E:3.28}. Then there exists a constant~$C>0$ such that, on $\{\wt K<\infty\}$,
\begin{equation}
\label{E:4.17ueu}
\biggl|\,\min_{x\in \Delta^\ell\smallsetminus \Delta^{\ell-1}}\Bigl(\frac2{\sqrt g}\fraka(x)-\phi(x)\Bigr)-S_{\ell}\,\biggr|\le \widetilde R_{\widetilde K}(\ell),\qquad\ell\ge1,
\end{equation}
where $\widetilde R_k(\ell):=C[1+\widetilde\Theta_k(\ell)]$.
\end{lemma}

\begin{proofsect}{Proof}
Let $\ell\ge1$ and pick $x\in \Delta^\ell\smallsetminus \Delta^{\ell-1}$. Then \eqref{E:3.29} and the various properties of the random objects $\varphi_j(0)$, $\chi_j$ and $h_j'$ imply
\begin{multline}
\qquad
\phi(x)=-S_{\ell}+\sum_{j\ge\ell}\frakb_j(x)\varphi_j(0)
\\+\sum_{j\ge\ell+2}\chi_j(x)+\bigl[\chi_{\ell-1}(x)+\chi_\ell(x)+\chi_{\ell+1}(x)+h_\ell'(x)\bigr].
\qquad
\end{multline}
As in the proof of Lemma~\ref{lemma-4.3}, the first two sums are bounded (in absolute value) by a quantity of the form $C[1+\Theta_k(\ell)]$ for some constant~$C>0$, uniformly in above~$x$. Similarly, we get
\begin{equation}
\Bigl|\,\min_{x\in \Delta^\ell\smallsetminus \Delta^{\ell-1}}\bigl[\chi_{\ell-1}(x)+\chi_\ell(x)+\chi_{\ell+1}(x)+h_\ell'(x)\bigr]-m_{2^\ell}\Bigr|\le C[1+\Theta_k(\ell)].
\end{equation}
The claim follows from
\begin{equation}
\max_{x\in \Delta^\ell}\,\Bigl|\,\frac2{\sqrt g}\fraka(x)-m_{2^\ell}\Bigr|\le C[1+\Theta_k(\ell)],
\end{equation}
as implied by \eqref{E:mN} and the asymptotic form of~$\fraka$ (see Lemma~\ref{lemma-potential}).
\end{proofsect}

As a consequence of the above observations, we again get a tight control between the event in Theorems~\ref{thm-lessmain} and~\ref{thm-2.5} and the above random walk:
\begin{equation}
\label{E:4.22uu}
\{\wt K<\infty\}\cap\bigcap_{\ell=1}^n\bigl\{S_{\ell}\ge \widetilde R_{\widetilde K}(\ell)\bigr\}
\subseteq\Bigl\{\phi(x)\le\frac2{\sqrt g}\fraka(x)\colon x\in \Delta^n\Bigr\}
\end{equation}
and
\begin{equation}
\label{E:4.22uuu}
\{\wt K<\infty\}\cap\Bigl\{\phi(x)\le\frac2{\sqrt g}\fraka(x)\colon x\in \Delta^n\Bigr\}
\subseteq\{\wt K<\infty\}\cap\bigcap_{\ell=1}^n\bigl\{S_{\ell}\ge -\widetilde R_{\widetilde K}(\ell)\bigr\}
\end{equation}
The upshot of \twoeqref{E:4.10}{E:4.11} and \twoeqref{E:4.22uu}{E:4.22uuu} is that the two events $\{h^{\Delta^n}\le(m_{2^n}+t)(1-\frakg^{\Delta^n})\}$ and $\Z^2$-counterpart $\{\phi\le\frac2{\sqrt g}\fraka\text{ in }\Delta^n\}$ of our prime interest can be represented, via bounds, as events that the random walk $\{S_k\colon k=1,\dots,n+1\}$ stays above a polylogarithmic curve. We thus need to find a way to efficiently control the probability of such random-walk events.

\subsection{Brownian motion above a curve}
\label{sec-4.2}\noindent
It is well-known that the simple random walk bridge (from zero to zero) of time-length~$n$ stays positive with probability that decays proportionally to~$n^{-1}$. An elegant proof exists, based on a symmetry argument, which applies to rather general walks with time-homogeneous steps. Unfortunately, our problem is harder for the following reasons:
\begin{enumerate}
\item[(1)] the steps of our random walk~$\{S_k\colon k=0,\dots,n\}$ are only approximately time-ho\-mo\-ge\-neous (see \eqref{E:3.20} for a precise statement),
\item[(2)] the events in \twoeqref{E:4.10}{E:4.11} compare the random walk to polylogarithmic curves and so an argument based solely on symmetry is not possible,
\item[(3)] the asymptotic probability of the giant intersections in \eqref{E:4.10} and \eqref{E:4.11} will differ by a mutliplicative constant, due to a difference in the restriction near the endpoints.
\end{enumerate}
We thus have to develop tools to address these differences. Based on experience gained in the context of Branching Brownian Motion by Bramson~\cite{Bramson-CPAM}, it is easier to first deal with corresponding claims for Brownian motion and Brownian bridge.

\smallskip
Let $\{B_t\colon t\ge0\}$ be the standard Brownian motion and let~$P^x$ be the law with $P^x(B_0=x)=1$. We will represent the ``curve'' by a function $\zeta\colon[0,\infty)\to[0,\infty)$.  We will separately deal with the cases of curves that are positive (i.e., the case when we require $B\ge\zeta$) or negative (i.e., the case of $B\ge-\zeta$). We begin with (unconditioned) Brownian motion above a positive curve:

\begin{proposition}
\label{prop-uncond-BM-positive-cor}
For~$\zeta\colon[0,\infty)\to[0,\infty)$ non-decreasing, continuous and such that $\zeta(s)=o(s^{1/2})$ as~$s\to\infty$, let
\begin{equation}
\label{E:rho}
\rho(x):=\zeta(x^2)+\frac x2\int_{x^2}^\infty\frac{\zeta(s)}{s^{3/2}}\,\textd s.
\end{equation}
Then for all $t>0$ and all $x>\zeta(0)$,
\begin{equation}
\label{E:4.24wu}
P^x\Bigl(B_s\ge\zeta(s)\colon s\in[0,t]\Bigr)
\ge (1-\delta)\sqrt{\frac2\pi}\frac{x}{\sqrt t}
\end{equation}
holds with
\begin{equation}
\delta:=\frac{x^2}{2t}+4\biggl(\frac{\rho(x)}x\biggr)^{2/3}.
\end{equation}
\end{proposition}

Clearly, the bound is not useful unless $x^2\ll t$ and $\rho(x)\ll x$. It turns out that roughly the same conditions guarantee good control for Brownian bridge above a positive curve as well:

\begin{proposition}
\label{prop-mine-corollary}
Let~$\zeta$ and~$\rho$ be as in Proposition~\ref{prop-uncond-BM-positive-cor}.
Then for each $t>0$ and all $x,y>\zeta(0)$,
\begin{equation}
\label{E:4.25wq}
P^x\Bigl(B_s\ge\zeta(s\wedge(t-s))\colon s\in[0,t]\,\Big|\,B_t=y\Bigr)
\ge (1-\delta)\frac{2xy}t
\end{equation}
holds with
\begin{equation}
\delta:=\frac{xy}t+4\biggl(\sqrt{\frac{\rho(x)}x}+\sqrt{\frac{\rho(y)}y}\biggr)\,\texte^{\frac{(x-y)^2}{2t}}.
\end{equation}
\end{proposition}

The proofs of Propositions~\ref{prop-uncond-BM-positive-cor} and~\ref{prop-mine-corollary} are long and technical and would detract from the main line of presentation. Hence we defer them to Appendix~\ref{appendA}. 

Next we move to the case of negative curves. We again start with the case of unconditioned Brownian motion:

\begin{proposition}
\label{prop-4.9ue}
For~$\zeta\colon[0,\infty)\to[0,\infty)$ non-decreasing and continuous and with $\zeta(s)=o(s^{1/4})$ as~$s\to\infty$, denote
\begin{equation}
\label{E:tilde-rho}
\tilde \rho(x):=\rho(x)+4\,\frac{\zeta(x^2)^2}x+2\int_{x^2}^\infty\frac{\zeta(s)^2}{s^{3/2}}\,\textd s
\end{equation}
where $\rho(x)$ is as in \eqref{E:rho}. Then for all $t>0$ and all $x>0$,
\begin{equation}
\label{E:4.29wu}
P^x\Bigl(B_s\ge-\zeta(s)\colon s\in[0,t]\Bigr)
\le (1+\delta)\sqrt{\frac2\pi}\,\frac{x}{\sqrt t}
\end{equation}
holds true with $\delta:=\kappa_1\bigl(\frac{\tilde\rho(x)}x\bigr)$ for $\kappa_1(u):=4(1+u^{2/3})u^{2/3}$.
\end{proposition}

For the bound to be useful, we need that $\delta\ll1$ which in turn requires that~$\tilde\rho(x)\ll x$. Not too surprisingly, the same criterion also applies to the case of the Brownian bridge:

\begin{proposition}
\label{prop-mine-too-corollary}
Suppose~$\zeta$ and $\tilde\rho$ are as in Proposition~\ref{prop-uncond-BM-positive-cor}.
Then for all $t>0$ and all $x,y>0$,
\begin{equation}
\label{E:4.25w}
P^x\Bigl(B_s\ge-\zeta(s\wedge(t-s))\colon s\in[0,t]\,\Big|\,B_t=y\Bigr)
\le (1+\delta)\frac{2xy}t
\end{equation}
holds true with $\delta:=\kappa_2\bigl(\frac{\tilde\rho(x)}x,\frac{\tilde\rho(y)}y\bigr)\texte^{\frac{(x-y)^2}{2t}}$ for $\kappa_2(u,v):=48(1+u)(1+v)(\sqrt{u}+\sqrt{v})$.
\end{proposition}

The proofs of Propositions~\ref{prop-4.9ue} and~\ref{prop-mine-too-corollary} are again relegated to Appendix~\ref{appendA}. The reader should note that, despite some similarities between the cases of positive and negative curves, there are also notable differences. Indeed, in order to have $\rho(x)\ll x$ it suffices to have $\zeta(s)=o(s^{1/2})$, while for~$\tilde\rho(x)\ll x$, we seem to need $\zeta(s)=o(s^{1/4})$. While the former criterion is basically best possible, the latter is likely not optimal. However, as the above statements are fully sufficient for our needs, we have not tried to bring them to a necessarily optimal form. 

\begin{remark}
\label{remark-Bramson}
We note that, in his groundbreaking study of the Branching Brownian Motion, Bramson~\cite[Propositions~1, 1', 2 and~2']{Bramson-CPAM} proved bounds of the form \eqref{E:4.24wu}, \eqref{E:4.25wq}, \eqref{E:4.29wu} and~\eqref{E:4.25w} for the specific choice $\zeta(s):=(3/\sqrt8)\log(s\vee1)$ and~$x=y$. Unfortunately, we are not able to use his conclusions for two reasons: First, Bramson's upper and lower bounds differ by an overall multiplicative constant which is something that our applications of these bounds cannot tolerate. Second, we need to control a whole class of~$\zeta$'s uniformly. 
\end{remark}

In order to use the above statements in various situations of interest, it will be convenient to have a quick tool for bounding the quantity~$\tilde\rho(x)$ (which automatically bounds also $\rho(x)$) for a reasonably large class of~$\zeta$. This is the subject of:

\begin{lemma}
\label{lemma-A.9}
Let $\zeta\colon[0,\infty)\to[0,\infty)$ be non-decreasing, continuously differentiable and such that $\zeta(0)>0$ and, for some~$a,\sigma>0$,
\begin{equation}
\label{E:4.23wu}
\zeta'(s)\le a\frac{\log(1+s/\sigma^2)}{1+s/\sigma^2},\qquad s>0.
\end{equation}
For~$u\ge0$ set $\zeta_u(s):=\zeta(u+s)$ and let $\tilde\rho_u(x)$ be the quantity in \eqref{E:tilde-rho} associated with~$\zeta_u$. There is a constant~$c=c(a,\sigma)$ such that for all~$x\ge1$ and all~$u\ge0$,
\begin{equation}
\label{E:4.31uea}
\tilde\rho_u(x)\le 2\zeta(u)+16\frac{\zeta(u)^2}x+c\Bigl(\log\bigl(\texte+\tfrac{x^2}{\sigma^2}\bigr)\Bigr)^4\,
\end{equation}
\end{lemma}

\begin{proofsect}{Proof}
Since $\frac{\log(1+s)}{1+s}\le\texte\frac{\log(\texte+s)}{\texte+s}$, with the expression on the right decreasing on~$\{s\ge0\}$, integrating the inequality \eqref{E:4.23wu} yields
\begin{equation}
\zeta_u(s)\le\zeta(u)+\frac12a\sigma^2\texte\Bigl[1+\log\bigl(\texte+\tfrac{s}{\sigma^2}\bigr)\Bigr]^2.
\end{equation}
To get the result, plug this in the expression for~$\tilde\rho$, use $(a+b)^2\le 2a^2+2b^2$ to deal with the quadratic occurrences of~$\zeta(s)$ and perform a sequence of integrations by parts.
\end{proofsect}

The above calculations permit us to make a statement concerning \emph{entropic repulsion}, which amounts to the fact that, conditioning a Brownian path to stay above a negative curve, it will stay above a positive curve except perhaps near the starting point.

\begin{proposition}
\label{prop-entropy-BM}
Let $\zeta\colon[0,\infty)\to[0,\infty)$  be related to  $a$ and~$\sigma$ as in Lemma~\ref{lemma-A.9}. There are constants $c=c(a,\sigma)>0$ and~$c'=c'(a,\sigma,\zeta(0))>0$ such that for all~$t>2c'$ and all~$u\in[c',t/2]$,
\begin{equation}
P^0\Bigl(\,\min_{0\le s\le t}\bigl[B_s+\zeta(s)\bigr]>0>\min_{u\le s\le t}\bigl[B_s-\zeta(s)\bigr]\Bigr)
\le c\,u^{-\frac1{16}}\frac1{\sqrt t}\,.
\end{equation}
\end{proposition}

This statement is far from optimal --- in fact, one expects that the Brownian motion will stay above the curve $s\mapsto s^\delta$ for each $\delta<\ffrac12$ --- but the above is easy to prove given what we already have and is completely sufficient for our needs. The proof is given in Section~\ref{sec-A5}. A similar statement holds for the Brownian bridge as well:

\begin{proposition}
\label{prop-entropy-BB}
Let~$\zeta$ be as in Proposition~\ref{prop-entropy-BM} and abbreviate $\tilde\zeta(s):=\zeta(s\wedge(t-s))$.
Then there are constants~$\tilde c=\tilde c(a,\sigma)>0$ and~$\tilde c'=\tilde c'(a,\sigma,\zeta(0))>0$ such that for all sufficiently large~$t>0$ and all~$u\in[\tilde c',t/4]$,
\begin{equation}
\label{E:4.34r} 
P^0\Bigl(\,\min_{0\le s\le t}\bigl[B_s+\tilde\zeta(s)\bigr]>0>\min_{u\le s\le t}\bigl[B_s-\tilde\zeta(s)\bigr]\,\Big|\, B_t=0\Bigr)
\le \tilde c\,u^{-\frac1{16}}\frac1{ t}\,.
\end{equation}
\end{proposition}

The proof proceeds by a reduction to Proposition~\ref{prop-entropy-BM} and is therefore also deferred to Section~\ref{sec-A5}. As before, we believe that the path gets repelled above a curve $s\mapsto (s\wedge (t-s))^\delta$ for every~$\delta<\ffrac12$, except perhaps near the endpoints.

\subsection{Random walk above a curve}
\label{sec-4.3}\noindent
Our next task is to convert the above statements for Brownian motion to statements about random walks. We will use the convenient fact that our random walks have mean-zero Gaussian steps and so we may as well realize them as values of a standard Brownian motion observed at a deterministic sequence of times. Let $\{\sigma_k^2\colon k\ge0\}$ denote a sequence of numbers obeying
\begin{equation}
\label{E:4.24y}
0<\inf_{k\ge0}\sigma_k^2\le\sup_{k\ge0}\sigma_k^2<\infty.
\end{equation}
Define
\begin{equation}
t_k:=\sum_{\ell=0}^{k-1}\sigma_\ell^2,\qquad k\ge1,
\end{equation}
with $t_0:=0$. Note that, for the specific choice $\sigma_k^2:=\text{Var}(\varphi_k(0))$ --- not necessarily assumed in the discussion later --- we get
\begin{equation}
\label{E:4.35ua}
\{S_k\colon k=1,\dots,n+1\}\laweq\{B_{t_k}\colon k=1,\dots,n+1\}.
\end{equation}
The reduction of the key statements from Brownian motion to the random walk will be considerably simplified using the following claim:

\begin{lemma}
\label{lemma-reduction}
For~$\{\sigma_k^2\colon k\ge0\}$ and $\{t_k\colon k\ge0\}$ as above, let $\sigma_{\min}^2$, resp., $\sigma_{\max}^2$ denote the infimum, resp., the supremum in \eqref{E:4.24y}. Given integers $n\ge1$ and $k\le\lfloor\ffrac n2\rfloor$ and a non-decreasing concave function $\gamma\colon[0,\infty)\to[0,\infty)$, let
\begin{equation}
\label{E:4.25qq}
\zeta(s):=\gamma\Bigl(\frac{t_k+s}{\sigma_{\min}^2}\Bigr).
\end{equation}
Then for all $x,y\in\R$,
\begin{multline}
\label{E:4.28qq}
\qquad
P^0\Bigl(B_{t_\ell}\ge-\gamma\bigl(\ell\wedge(n-\ell)\bigr)\colon \ell=k,\dots,n-k\,\Big|\, B_{t_k}=x, B_{t_{n-k}}=y\Bigr)
\\
\le P^x\Bigl(B_s\ge-2\zeta\bigl(s\wedge(t_{n-k}-t_k-s)\bigr)\colon s\in[0,t_{n-k}-t_k]\,\Big|\, B_{t_{n-k}-t_k}=y\Bigr)
\\
\times\prod_{j=k}^{n-k}\Bigl(1-\texte^{-2\sigma_{\max}^{-2}\gamma(j)^2}\Bigr)^{-2}.
\qquad
\end{multline}
Similarly, if~$\tilde\zeta$ is defined by
\begin{equation}
\label{E:4.28qq1}
\tilde\zeta(s):=\gamma\Bigl(\frac{t_k+s}{\sigma_{\max}^2}\Bigr),
\end{equation}
then for all $x,y\in\R$,
\begin{multline}
\label{E:4.29qq}
\qquad
P^0\Bigl(B_{t_\ell}\ge-\gamma\bigl(\ell\wedge(n-\ell)\bigr)\colon \ell=k,\dots,n-k\,\Big|\, B_{t_k}=x, B_{t_{n-k}}=y\Bigr)
\\
\ge P^x\Bigl(B_s\ge-\tilde\zeta\bigl(s\wedge(t_{n-k}-t_k-s)\bigr)\colon s\in[0,t_{n-k}-t_k]\,\Big|\, B_{t_{n-k}-t_k}=y\Bigr).
\qquad
\end{multline}
\end{lemma}

Notice that the function $s\mapsto \zeta\bigl(s\wedge(t_{n-k}-t_k-s)\bigr)$ is symmetric about the midpoint of the interval $[0,t_{n-k}-t_k]$. This will be important as soon as we try to apply the earlier conclusions to the probabilities on the right-hand sides of \eqref{E:4.28qq} and \eqref{E:4.29qq}. 

\smallskip
The proof of Lemma~\ref{lemma-reduction} will rely on the following bounds:

\begin{lemma}
\label{lemma-4.7q}
Given sequences $\{\sigma_k^2\colon k\ge1\}$ and $\{t_k\colon k\ge0\}$ as above, let~$\zeta\colon[0,t_n]\to[0,\infty)$ be a concave function. Then for each $x,y\ge0$,
\begin{multline}
\label{E:4.20long}
\qquad
P^x\bigl(B_s\ge-2\zeta(s)\colon 0\le s\le t_n\,|\,B_{t_n}=y\bigr)\prod_{k=1}^{n}\biggl(1-\texte^{-2\sigma_k^{-2}[\zeta(t_{k-1})\wedge\zeta(t_k)]^2}\biggr)^{-1}
\\
\ge
P^x\bigl(B_{t_k}\ge-\zeta(t_k)\colon k=0,\dots,n\,|\,B_{t_n}=y\bigr)
\\
\ge
P^x\bigl(B_s\ge-\zeta(s)\colon 0\le s\le t_n\,|\,B_{t_n}=y\bigr).
\qquad
\end{multline}
\end{lemma}

\begin{proofsect}{Proof}
The inequality on the right is trivial so let us focus on that on the left. Let
\begin{equation}
W^{(k)}(s):=\frac{t_{k+1}-s}{t_{k+1}-t_k}B_{t_k}+\frac{s-t_k}{t_{k+1}-t_k}B_{t_{k+1}}-B_s,\qquad t_k\le s\le t_{k+1}.
\end{equation}
Under the conditional measure $P^x(-|B_{t_n}=y)$, the processes $\{W^{(k)}\colon k\ge0\}$ have the law of a family of independent Brownian bridges --- from zero to zero, with $W^{(k)}$ indexed by times in the interval~$[t_k,t_{k+1}]$ --- and this family is independent of the values $\{B_{t_k}\colon k=0,\dots,n\}$. Since $\zeta$ is concave on the intervals $[t_k,t_{k+1}]$, we have
\begin{multline}
\label{E:4.42qq}
\bigl\{B_s\ge-2\zeta(s)\colon 0\le s\le t_n\bigr\}
\\\supseteq 
\bigl\{B_{t_k}\ge-\zeta(t_k)\colon k=0,\dots,n\bigr\}\cap\bigcap_{k=0}^{n-1}\Bigl\{\,\max_{t_k\le s\le t_{k+1}} W^{(k)}(s)\le\zeta(t_{k})\wedge\zeta(t_{k+1})\Bigr\}.
\quad
\end{multline}
By the Reflection Principle (cf~\eqref{E:4.20}), the probability of the event in the giant intersection corresponding to index~$k$ is equal to $1-\texte^{-2\sigma_{k+1}^{-2}[\zeta(t_{k})\wedge\zeta(t_{k+1})]^2}$. Using the stated independence, we readily get the left inequality in \eqref{E:4.20long} as well.
\end{proofsect}


\begin{proofsect}{Proof of Lemma~\ref{lemma-reduction}}
Abbreviate $\widehat\zeta(s):=\zeta(s\wedge(t_{n-k}-t_k-s))$. Noting that
\begin{equation}
t_\ell\ge\sigma_{\min}^2\ell\quad\text{and}\quad
t_{n-k}+t_k-t_\ell\ge\sigma_{\min}^2(n-\ell),\qquad\ell=k,\dots,n-k,
\end{equation}
we have
\begin{equation}
\label{E:4.35qq}
\widehat\zeta(t_\ell-t_k)\ge\gamma\bigl(\tfrac1{\sigma_{\min}^2}(t_\ell\wedge(t_{n-k}+t_k-t_\ell))\bigr)
\ge\gamma\bigl(\ell\wedge(n-\ell)\bigr)
\end{equation}
for all $\ell=k,\dots,n-k$. This yields
\begin{multline}
\qquad
P^0\Bigl(B_{t_\ell}\ge-\gamma\bigl(\ell\wedge(n-\ell)\bigr)\colon \ell=k,\dots,n-k\,\Big|\, B_{t_k}=x, B_{t_{n-k}}=y\Bigr)
\\
\le P^x\Bigl(B_{t_\ell-t_k}\ge-\widehat\zeta(t_\ell-t-t_k)\colon \ell=k,\dots,n-k\,\Big|\, B_{t_{n-k}-t_k}=y\Bigr)
\qquad
\end{multline}
Next we use \eqref{E:4.20long} to estimate this by 
\begin{equation}
P^x\Bigl(B_s\ge-2\widehat\zeta(s)\colon s\in[0,t_{n-k}-t_k]\,\Big|\, B_{t_{n-k}-t_k}=y\Bigr)
\prod_{\ell=k+1}^{n-k}\Bigl(1-\texte^{-2\sigma_\ell^{-2}[\widehat\zeta(t_{\ell-1})\wedge \widehat\zeta(t_{\ell})]^2}\Bigr)^{-1}.
\end{equation}
From \eqref{E:4.35qq} it is easy to check that the product is bounded by that in \eqref{E:4.28qq}. This proves the upper bound. For the lower bound \eqref{E:4.29qq} we replace~$\zeta$ by~$\tilde\zeta$ in the definition of~$\widehat\zeta$ and notice that the opposite inequality then holds in~\eqref{E:4.35qq}. Lemma~\ref{lemma-4.7q} then readily yields the claim.
\end{proofsect}

We will also need a similar statement for the unconditioned random walk. (Unfortunately, due to the symmetrization of the function~$\zeta$ in \eqref{E:4.28qq} and \eqref{E:4.29qq}, this does not follow by mere integration over~$y$.) Naturally, the expressions are considerably simpler in this case:

\begin{lemma}
\label{lemma-reduction1}
Let $\gamma$ be a function as in Lemma~\ref{lemma-reduction}. For each~$x\in\R$, each $n\ge k\ge1$ and $\zeta$ defined from~$\gamma$ and~$k$ as in \eqref{E:4.25qq},
\begin{multline}
\label{E:4.28new}
\qquad
P^0\Bigl(B_{t_\ell}\ge-\gamma\bigl(\ell\bigr)\colon \ell=k,\dots,n\,\Big|\, B_{t_k}=x\Bigr)
\\
\le P^x\Bigl(B_s\ge-2\zeta(s)\colon s\in[0,t_{n}-t_k]\Bigr)
\prod_{j=k}^{n}\Bigl(1-\texte^{-2\sigma_{\max}^{-2}\gamma(j)^2}\Bigr)^{-2}.
\qquad
\end{multline}
Similarly, for all~$x\in\R$ and all $n\ge k\ge1$ and~$\tilde\zeta$ as in \eqref{E:4.28qq1},
\begin{equation}
\label{E:4.29new}
P^0\Bigl(B_{t_\ell}\ge-\gamma(\ell)\colon \ell=k,\dots,n\,\Big|\, B_{t_k}=x\Bigr)
\ge P^x\Bigl(B_s\ge-\tilde\zeta(s)\colon s\in[0,t_{n}-t_k]\Bigr).
\end{equation}
\end{lemma}

\begin{proofsect}{Proof}
The proof follows exactly the same lines as that of Lemma~\ref{lemma-reduction} and is in fact easier do to the absence of symmetrization. Hence, it is omitted.
\end{proofsect}

Next we will move to the statement of entropic repulsion for the above random walk. We will henceforth work with
\begin{equation}
\label{E:gamma}
\gamma(s):=a\bigl[1+\log(a+s)\bigr]^2\quad\text{where}\quad a>\frac{1+\sqrt5}2.
\end{equation}
As a calculation shows, the restriction on~$a$ ensures that $\gamma$ is non-decreasing and concave on~$[0,\infty)$.
For a change, we start with the claim for the unconditioned walk:

\begin{lemma}
\label{lemma-entropy-RW}
Let~$\gamma$ be as in \eqref{E:gamma}. Then there is~$c=c(a)\in(0,\infty)$ such that for all sufficiently large~$n\ge1$ and all~$k\ge1$ with $k\le\frac n2$,
\begin{equation}
\label{E:4.47wq}
\frac{\displaystyle P^0\Bigl(\,\min_{0\le \ell\le n}\bigl[B_{t_\ell}+\gamma(\ell)\bigr]>0>\min_{k\le\ell\le n}\bigl[B_{t_\ell}-\gamma(\ell)\bigr]\Bigr)}
{\displaystyle P^0\Bigl(\,\min_{0\le \ell\le n}\bigl[B_{t_\ell}+\gamma(\ell)\bigr]>0\Bigr)}
\le c\,k^{-\frac1{16}}\,.
\end{equation}
\end{lemma}

\begin{proofsect}{Proof}
Let $\zeta^{\min}(s):=\gamma(s/\sigma_{\min}^2)$. Then $\zeta^{\min}(t_\ell)\ge\gamma(\ell)$ and \eqref{E:4.42qq} thus implies
\begin{multline}
\qquad
P^0\Bigl(\,\min_{0\le \ell\le n}\bigl[B_{t_\ell}+\gamma(\ell)\bigr]>0>\min_{k\le\ell\le n}\bigl[B_{t_\ell}-\gamma(\ell)\bigr]\Bigr)\prod_{j=1}^{n}\Bigl(1-\texte^{-2\sigma_{\max}^{-2}\gamma(j)^2}\Bigr)
\\\le
P^0\Bigl(\,\min_{0\le s\le t_n}\bigl[B_s+\zeta^{\min}(s)\bigr]>0>\min_{t_k\le s\le t_n}\bigl[B_s-\zeta^{\min}(s)\bigr]\Bigr)
\qquad
\end{multline}
The product is bounded away from zero uniformly in~$n$ and so, Proposition~\ref{prop-entropy-BM} and the fact that $t_k\ge\sigma_{\min}^2k$, the right-hand side is at most $c k^{-\frac1{16}}/\sqrt{t_n}$ as soon as~$k$ is large enough.

For a suitable lower bound on the denominator, here we set $\zeta^{\max}(s):=\gamma(s/\sigma_{\max}^2)$ and note that $\gamma(\ell)\ge \zeta^{\max}(\ell)$. By Proposition~\ref{prop-entropy-BM},
\begin{equation}
\begin{aligned}
\label{E:4.53ueu}
P^0\Bigl(\,\min_{0\le \ell\le n}\bigl[B_{t_\ell}+\gamma(\ell)\bigr]>0\Bigr)
&\ge P^0\Bigl(\,\min_{0\le s\le t_n}\bigl[B_{s}+\zeta^{\max}(s)\bigr]>0\Bigr)
\\
&\ge P^0\Bigl(\,\min_{u\le s\le t_n}\bigl[B_{s}-\zeta^{\max}(s)\bigr]>0\Bigr)
-c u^{-\frac1{16}}\frac1{\sqrt{t_n}},
\end{aligned}
\end{equation}
where $c=c(a,\sigma)$ and $u\ge c'=c'(a,\sigma,\zeta(0))$.
By Lemma~\ref{lemma-A.9} for~$\zeta^{\max}$ in place of~$\zeta$, the quantity $\sup_{x\ge u}\tilde\rho_u(x)/x$ is bounded and tending to zero with~$u\to\infty$. Hence, for~$u$ large enough, 
Proposition~\ref{prop-uncond-BM-positive-cor} bounds the probability on the right of \eqref{E:4.53ueu} by~$\tilde c/\sqrt{t_n}$ with some~$\tilde c$ independent of $(a,\sigma,\zeta(0))$. The claim follows.
\end{proofsect}

\begin{lemma}
\label{lemma-entropy-RWB}
Let $\gamma$ be as in \eqref{E:gamma} and set $\tilde\gamma(k):=\gamma(k\wedge(n-k))$. There is~$c=c(a)\in(0,\infty)$ such that for all sufficiently large~$n\ge1$ and all~$k\ge1$ with $k\le\frac n4$,
\begin{equation}
\frac{\displaystyle P^0\Bigl(\,\min_{0\le \ell\le n}\bigl[B_{t_\ell}+\tilde\gamma(\ell)\bigr]>0>\inf_{k\le\ell\le n-k}\bigl[B_{t_\ell}-\tilde\gamma(\ell)\bigr]\,\Big|\, B_{t_n}=0\Bigr)}
{\displaystyle P^0\Bigl(\,\min_{0\le \ell\le n}\bigl[B_{t_\ell}+\tilde\gamma(\ell)\bigr]>0\,\Big|\, B_{t_n}=0\Bigr)}
\le c\,k^{-\frac1{16}}\,.
\end{equation}
\end{lemma}

\begin{proofsect}{Proof}
The proof is completely analogous to that of Lemma~\ref{lemma-entropy-RW};  we just use Lemma~\ref{lemma-reduction} instead of Lemma~\ref{lemma-reduction1} and Proposition~\ref{prop-entropy-BB} instead of Proposition~\ref{prop-entropy-BM}.
\end{proofsect}

\subsection{Growth and gap control}
\label{sec-4.4}\noindent
Returning back to our consideration of the pinned \DGFF{}, we are now in a position to derive quantitative estimates on various undesirable events of interest. We start by bounding the probability that the control variables~$K$, resp.,~$\widetilde K$ introduced earlier take a given value: 

\begin{lemma}
\label{lemma-4.7}
There are constants~$c_1,c_2\in(0,\infty)$ such that for all~$n\ge1$ sufficiently large, all $k\ge1$ with~$k<n/4$ and all $t\ge0$,
\begin{equation}
\label{E:4.19qq}
P\biggl(\,\{K=k\}\cap \bigcap_{\ell=k}^{n-k}\bigl\{S_\ell\ge -R_{k}(\ell)-t\bigr\}\,\bigg|\,S_{n+1}=0\biggr)\le c_1(1+t)^2\frac{\texte^{-c_2(\log k)^2}}{n}.
\end{equation}
Similarly, there are constants~$\tilde c_1,\tilde c_2\in(0,\infty)$ such that for all~$r\ge k\ge1$,
\begin{equation}
\label{E:4.19qq2}
\nu^0\biggl(\,\{\widetilde K=k\}\cap \bigcap_{\ell=k}^{r}\bigl\{S_\ell\ge -\widetilde R_{k}(\ell)\bigr\}\biggr)\le \tilde c_1\frac{\texte^{-\tilde c_2(\log k)^2}}{\sqrt r}.
\end{equation}
\end{lemma}

\begin{proofsect}{Proof}
Given~$k\ge1$, let~$A_k$ denote the event that the conditions in Definition~\ref{def4.1} (with~$k$ as stated) hold for all~$\ell$ satisfying $\ell\le k$ or~$\ell\ge n-k$, but at least one of these conditions fails when~$k$ is replaced by~$k-1$. Note that~$A_k$ belongs to the $\sigma$-algebra
\begin{equation}
\FF_k:=\sigma\Bigl(\varphi_\ell(0),\chi_\ell,h'_\ell\colon \ell=0,\dots,k,n-k,\dots,n\Bigr)
\end{equation}
and that $\{K=k\}\subseteq A_k$. The desired probability is thus bounded by
\begin{equation}
\label{E:4.53ue}
E\biggl(\1_{A_k}\,P\Bigl(\,\,\bigcap_{\ell=k}^{n-k}\bigl\{S_\ell\ge -R_{k}(\ell)-t\bigr\}\,\Big|\,\FF_k\Bigr)\,\bigg|\,S_{n+1}=0\biggr).
\end{equation}
Our strategy is to derive a pointwise estimate on the conditional probability. 

First we note that, setting $\gamma$ to \eqref{E:gamma} with~$a>\max\{C,\frac{1+\sqrt5}2\}$ where~$C$ is the constant in Lemma~\ref{lemma-4.3}, on the event $\{S_k=x,\,S_{n-k}=y\}$ the above conditional probability can be bounded above by
\begin{equation}
P^0\Bigl(B_{t_\ell}\ge-\gamma(\ell\wedge(n-\ell))-t\colon \ell=k,\dots,n-k\,\Big|\,B_{t_k}=x,\,B_{t_{n-k}}=y\Bigr),
\end{equation}
where $B$ is the standard Brownian motion and~$\{t_k\colon k\ge1\}$ --- not to be confused with~$t$ and~$t_0$ in the statement --- are now defined using $\sigma_k^2:=\Var(\varphi_k(0))$. Setting~$\zeta(s):=\gamma(s/\sigma_{\min}^2)$ and comparing this with \eqref{E:4.25qq}, we bound this probability using \eqref{E:4.28qq} with~$\zeta$ replaced by $\zeta_u(s):=\zeta(u+s)$ where $u:=t_k$. Here we observe that, for our choice of~$\gamma$, the product on the right of \eqref{E:4.28qq} is bounded by a constant uniformly in~$k$,~$n$ and $t\ge0$. Using a trivial shift of coordinates, we thus need to derive a uniform bound on
\begin{equation}
\label{E:4.45ww}
P^{x+t}\Bigl(B_s\ge-2\zeta_{t_k}\bigl(s\wedge(t_{n-k}-t_k-s)\bigr)\colon s\in[0,t_{n-k}-t_k]\,\Big|\, B_{t_{n-k}-t_k}=y+t\Bigr)
\end{equation}
for all relevant~$x$ and~$y$.
For this we note that, on~$A_k$, we necessarily have $|\varphi_\ell(0)|\le\Theta_k(k)$ for all $\ell\le k$ and $\ell\ge n-k$ and so, since we condition on~$S_{n+1}=0$, we may assume that
\begin{equation}
-\gamma(k\wedge(n-k))\le S_k,S_{n-k}\le (k+1)\Theta_k(k)=(k+1)[\log(k+1)]^2.
\end{equation}
Since \eqref{E:4.45ww} increases when both~$x$ and~$y$ are increased by the same amount, by adding to~$x$ and~$y$ two times $a_k:=\gamma(k)\vee((k+1)[\log(k+1)]^2)$ it thus suffices to estimate the maximal value the probability in \eqref{E:4.45ww} takes for $a_k\le x,y \le 3a_k$. 

We will use Proposition~\ref{prop-mine-too-corollary} but for that we first need to bound the error term  denoted by~$\delta$.  Since we work with~$\zeta_u$, we first check that \eqref{E:4.23wu} in~Lemma~\ref{lemma-A.9} applies with~$a$ as above and $\sigma^2:=\sigma_{\min}^2$, and so the requisite~$\tilde\rho_u$ for $u:=t_k$ is can be estimated as in \eqref{E:4.31uea}. Hence we get
\begin{equation}
\sup_{t\ge0}\,\,\max_{k=1,\dots,n}\,\,\sup_{a_k\le x \le 3a_k}\frac{\tilde\rho_{t_k}(x+t)}{x+t}<\infty.
\end{equation}
As $t_{n-k}-t_k\ge c'n$ for some~$c'>0$ due to our assumption that $k<n/4$, \eqref{E:4.25w} in conjunction with the previous steps yield
\begin{equation}
P\Bigl(\,\,\bigcap_{\ell=k}^{n-k}\bigl\{S_\ell\ge -R_{k}(\ell)-t\bigr\}\,\Big|\,\FF_k\Bigr)
\le c_1\frac{(3a_k+t)^2}n\qquad\text{on }A_k\cap\{S_{n+1}=0\},
\end{equation}
for some constant~$c_1>0$.
The expectation in \eqref{E:4.53ue} is then at most $c_1\frac{(3a_k+t)^2}n P(A_k\,|\,S_{n+1}=0)$ and the last probability is estimated as in Lemma~\ref{lemma-4.2u} by $\texte^{-c_2(\log k)^2}$. Sacrificing part of the exponent, we can now rewrite this as \eqref{E:4.19qq}.

The proof of the corresponding statement \eqref{E:4.19qq2} follows exactly the same argument as that of \eqref{E:4.19qq}; in fact, the derivation is simpler due to the absence of~$n$ and~$t$ dependence of all terms. We leave the details to the reader.
\end{proofsect}

As a consequence of these bounds, we are able to control the leading order of the probability of the principal events of interest:

\begin{lemma}
\label{lemma-4.9new}
There are a constant~$c_2\in(0,\infty)$ and for each $t_0>0$ also a constant $c_1\in(0,\infty)$ such that for all $n\ge1$ and all $s\in[0,t_0]$ and all $t\le s$, we have
\begin{equation}
\label{E:4.24z}
\frac{c_1}n\le P\Bigl(\,h^{\Delta^n}\le m_{2^n}+s-(m_{2^n}+t)\frakg^{\Delta^n}
\,\Big|\,h^{\Delta^n}(0)=0\Bigr)\le\frac{c_2}n(1+s-t)^2.
\end{equation}
Similarly, there are $c_1',c_2'\in(0,\infty)$ such that for all $r\ge1$,
\begin{equation}
\label{E:4.25z}
\frac{c_1'}{\sqrt r}\le \nu^0\Bigl(\phi(x)\le\frac2{\sqrt g}\fraka(x)\colon x\in \Delta^r\Bigr)\le\frac{c_2'}{\sqrt r}.
\end{equation}
\end{lemma}

\begin{proofsect}{Proof}
Let us start with the upper bound in \eqref{E:4.24z}. For the event under consideration we get
\begin{multline}
\quad
\bigl\{h^{\Delta^n}-m_{2^n}(1-\frakg^{\Delta^n})\le (s-t)\frakg^{\Delta^n}\bigr\}
\\
\subseteq\{K=\lfloor\ffrac n2\rfloor+1\}\cup \bigcap_{k=1}^n\Bigl(\bigl\{S_k-S_{n+1}\ge -R_K(k)-(s-t)\bigr\}\cap\bigl\{K\le\lfloor\ffrac n2\rfloor\bigr\}\Bigr).
\end{multline}
Here we bounded $(s-t)\frakg^{\Delta^n}\le(s-t)$ and then invoked Lemma~\ref{lemma-4.3}. In light of \eqref{E:3.16qq}, the upper bound in \eqref{E:4.24z} follows  by summing \eqref{E:4.19qq} over $1\le k\le\ffrac n4$ and applying \eqref{E:4.6} to deal with the complementary values of the control variable. (The contribution of the latter part is then absorbed into that of the former.)

The upper bound in \eqref{E:4.25z} is completely analogous; we invoke Lemma~\ref{lemma-4.6a} to rewrite the event using the random walk and the control variable and then bound the resulting probability by \eqref{E:4.19qq2} for $k\le r$ and \eqref{E:4.16ueu} for $k\ge r$ (including the default value $\widetilde K=\infty$). 

For the lower bounds, we will conveniently use the fact that the law of both $h^{\Delta^n}$ conditioned on~$h^{\Delta^n}(0)=0$ and $\phi$ are positively correlated (see Lemma~\ref{lemma-FKG}). Focussing, for simplicity, on \eqref{E:4.25z} and abbreviating $\phi'(x):=\frac2{\sqrt g}\fraka(x)-\phi(x)$, for each $1\le k\le r$ we thus have
\begin{equation}
\nu^0\bigl(\phi'(x)\ge0\colon x\in \Delta^r\bigr)
\ge \nu^0\bigl(\phi'(x)\ge0\colon x\in \Delta^k\bigr)\nu^0\bigl(\phi'(x)\ge0\colon x\in \Delta^r\smallsetminus \Delta^k\bigr)
\end{equation}
By inserting the event $\{\widetilde K\le k\}$, we then get
\begin{multline}
\label{E:4.65aa}
\qquad
\nu^0\bigl(\phi'(x)\ge0\colon x\in \Delta^r\smallsetminus \Delta^k\bigr)
\\\ge P\Bigl(\{\widetilde K\le k\}\cap\bigcap_{\ell=1}^r\bigl\{S_k\ge -\widetilde R_k(\ell)\bigr\}\cap\bigcap_{\ell=k}^r\bigl\{S_k\ge \widetilde R_k(\ell)\bigr\}\Bigr),
\qquad
\end{multline}
where we also noted that, by \twoeqref{E:4.22uu}{E:4.22uuu}, on $\{\widetilde K\le k\}$ the event on the left is a subset of the first giant intersection on the right. Then we applied \eqref{E:4.17ueu}. 

Now we use \eqref{E:4.19qq2} conclude that the right-hand side of \eqref{E:4.65aa} is at least
\begin{equation}
\label{E:4.66ui}
P\Bigl(\,\bigcap_{\ell=1}^r\bigl\{S_k\ge -\widetilde R_k(\ell)\bigr\}\Bigr)
P\Bigl(\,\bigcap_{\ell=k}^r\bigl\{S_k\ge \widetilde R_k(\ell)\bigr\}\,\Big|\,\bigcap_{\ell=1}^r\bigl\{S_k\ge -\widetilde R_k(\ell)\bigr\}\Bigr)-\tilde c_1\frac{\texte^{-\tilde c_2(\log k)^2}}{\sqrt r}.
\end{equation}
Since $\widetilde R_k(\ell)\ge0$, our embedding of the random walk into Brownian motion as detailed in \eqref{E:4.35ua} shows that the first probability can be bounded as
\begin{equation}
P\Bigl(\,\bigcap_{\ell=1}^r\bigl\{S_k\ge -\widetilde R_k(\ell)\bigr\}\Bigr)
\ge P\bigl(\,B_s\ge 0\colon s\in[1,t_r]\bigr)\,,
\end{equation}
where $t_r:=\Var(S_r)$. By the Reflection Principle, this is at least a constant times $t_r^{-1/2}$, which by \eqref{E:4.24y} is at least a constant times $r^{-1/2}$. The  conditional  probability in \eqref{E:4.66ui} is bounded  from  below using Lemma~\ref{lemma-entropy-RWB} by a quantity of the form $1-ck^{-\frac1{16}}$. The left inequality in \eqref{E:4.25z} follows once~$k$ is taken sufficiently large.

Concerning the lower bound in \eqref{E:4.24z}, here we simply set~$s:=t$ and then proceed by an argument quite analogous to the one above. (The only change is that we need to write $R_k(\ell)+|t|$ instead of~$\widetilde R_k(\ell)$.) We omit further details for brevity.
\end{proofsect}

Our final task in this subsection is to show that configurations conditioned to stay above (or below) a function will leave a uniform gap between the conditional and typical value. This gap will be crucial in various approximation arguments in the next section. Here, for ease of expression, we write ``$f\not\ge g$ in~$\Lambda$'' to designate that there is $x\in\Lambda$ such that $f(x)<g(x)$.

\begin{lemma}
\label{lemma-4.20}
For all $k\ge1$ and~$\epsilon>0$ there is $\delta>0$ such that for all $r\ge k$ and $s\in\{0,\delta\}$,
\begin{equation}
\label{E:4.64}
\nu^0\Bigl(\frac2{\sqrt g}\fraka-\phi\not\ge \delta-s\text{\rm\,\ in } \Delta^k\smallsetminus\{0\}\,\Big|\,\frac2{\sqrt g}\fraka-\phi\ge-s \text{\rm\,\ in } \Delta^r\Bigr)\le\epsilon.
\end{equation}
Similarly, for all $k\ge1$, all $\epsilon>0$ and all~$t_0>0$ there is $\delta>0$ such that for all $t\in\R$ with $|t|<t_0$ and all $n\ge k$,
\begin{equation}
\label{E:4.65a}
P\Bigl(\,h^{\Delta^n}\not\le\frakm_n -\delta\text{\rm\,\ in }\Delta^k\smallsetminus\{0\}
\,\Big|\,h^{\Delta^n}\le \frakm_n,\,h^{\Delta^n}(0)=0\Bigr)\le \epsilon, 
\end{equation}
where $\frakm_n(x):=(m_{2^n}+t)(1-\frakg^{\Delta^n})$.
\end{lemma}

\begin{proofsect}{Proof}
For convenience, both parts will again use the fact that that laws of the random fields~$\phi$ and~$h^{\Delta^n}$ are strong-FKG (see Lemma~\ref{lemma-FKG}). Our argument will be facilitated by the following general observations: If $\chi$ is a field on a set~$\Lambda$ with the strong-FKG property and $f\colon \Lambda\to\R$ is a function, then
\begin{equation}
P\bigl(\chi\not\ge f+\delta\,\big|\,\chi\ge f\bigr)\le\sum_{x\in\Lambda}P\bigl(\chi(x)< f(x)+\delta\,\big|\,\chi(x)\ge f(x)\bigr).
\end{equation}
Moreover, if $\chi(x)$ is normal with mean zero and $\sigma^2:=\Var(\chi(x))>0$, and $f(x)\le0$, then
\begin{equation}
P\bigl(\chi(x)\not\ge f(x)+\delta\,\big|\,\chi(x)\ge f(x)\bigr)\le\frac2{\sqrt{2\pi}}\frac\delta\sigma\le\frac\delta\sigma.
\end{equation}
This follows by a straightforward estimate of the probability density of~$\chi(x)$.

The bound \eqref{E:4.64} follows with the choice $f(x):=-\frac2{\sqrt g}\fraka(x)-s$ and $\epsilon:=2c\delta|\Delta^k|$, where~$c$ is a number such that~$c^{-2}$ is the minimum of $\Var(\phi(x))$ over $x\in \Delta^k\smallsetminus\{0\}$. The bound \eqref{E:4.65a} is obtained analogously; we just need to also observe that, since $(h^{\Delta^n}(x)|h^{\Delta^n}(0)=0)$ converges in law to~$\phi$ as $n\to\infty$, we have $\inf_{n\ge k}\Var(h^{\Delta^n}(x)|h^{\Delta^n}(0)=0)>0$ for each $x\in \Delta^k\smallsetminus\{0\}$.
\end{proofsect}

\subsection{More general outer domains}
\label{sec-4.5}\noindent
In order to simplify exposition, the derivations in the previous sections have been based on the definition of~$\{\Delta^k\colon k\ge1\}$ via \eqref{E:3.25}. However, applications will occasionally require estimates also in the situations when the last domain in the sequence~$\Delta^0,\Delta^1,\dots,\Delta^n$ is replaced by a slightly more general domain, albeit of the same spatial scale;  see Fig.~\ref{fig4}.  Here we observe: 

\begin{lemma}
\label{lemma-gen-domain}
Let~$q\ge1$. Then there are constants $c_1,c_2\in(0,\infty)$ such that if we replace~$\Delta^n$ for $n\ge1$ by a set~$D\subset\Z^2$ satisfying
\begin{equation}
\label{E:D-incl}
\Delta^n\subseteq D\subseteq \Delta^{n+q},
\end{equation}
then
\begin{equation}
c_1\le\Var\bigl(\varphi_n(0)\bigr)\le c_2.
\end{equation}
Similarly, when we replace $\Delta^n$ by any~$D$ obeying \eqref{E:D-incl}, the conclusions of Lemmas~\ref{lemma-3.5rt},~\ref{lemma-3.6a} and \eqref{E:3.40} in Lemma~\ref{lemma-3.6} hold as stated for all $k\le n$ with the constants that may depend on~$q$ but not on~$n$. The same holds for \eqref{E:4.19qq} in Lemma~\ref{lemma-4.7}, \eqref{E:4.24z} in Lemma~\ref{lemma-4.9new} and \eqref{E:4.65a} in Lemma~\ref{lemma-4.20}.
\end{lemma}

\begin{proofsect}{Proof}
The proofs of the statements under consideration depend sensitively on the underlying domains only via \eqref{E:3.32w}, which we use only for $\Delta^k$ with~$k\le n-1$, and bounds on the variance of~$\varphi_k(0)$. The only variance that changes when~$\Delta^n$ is replace by above~$D$ is that of~$\varphi_n(0)$, which is bounded by \eqref{E:3.21} and the fact that $D\mapsto\Var(h^{D}(0))$ is non-decreasing with respect to set inclusion. The bounds in Lemmas~\ref{lemma-3.6} and~\ref{lemma-3.7} use only the various coarse facts about the underlying domain spelled out in Lemmas~\ref{lemma-stoch-order} and \ref{lem:6}. The statements of Lemmas~\ref{lemma-4.7},~\ref{lemma-4.9new} and~\ref{lemma-4.20} then follow as well.
\end{proofsect}

\begin{figure}
\vtop to 0.5\textwidth{
\centerline{\includegraphics[width=0.5\textwidth]{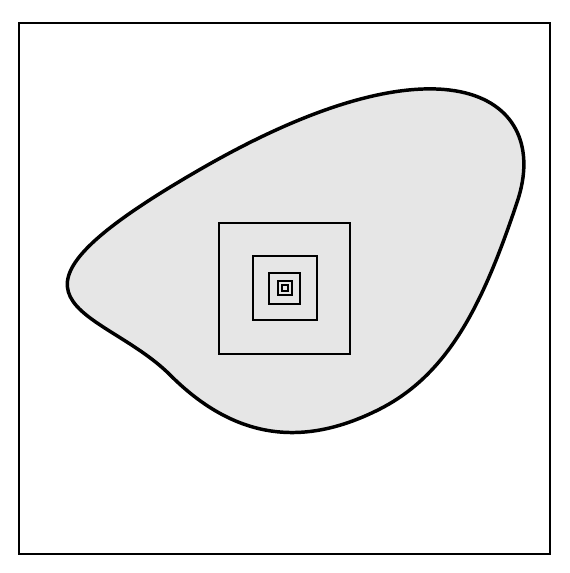}}
}
\begin{quote}
\small
\caption{An illustration of the setting for more general outer domains with parametrization as in~\eqref{E:D-incl}. The set~$D$ is lightly shaded, the parameter~$q$ equals~$2$.
\label{fig4}
}
\normalsize
\end{quote}
\end{figure}

\section{Limit extremal process}
\label{sec5a}\noindent
We are ready to start proving our main results. We begin by addressing the existence of the limiting distribution of the clusters. The key technical input for these will be the asymptotic formulas stated in Propositions~\ref{prop-4.14}--\ref{prop-4.13} below.

\subsection{Key asymptotic formulas}
Abusing the standard notations slightly, let $\Cb(\R^{\Delta^j})$ denote the class of bounded and continuous functions $f\colon\R^{\Z^d}\to\R$ that depend only on the coordinates in~$\Delta^j$. We will also write $\Cloc(\R^{\Z^2}):=\bigcup_{j\ge1}\Cb(\R^{\Delta^j})$.
Given an integer~$\ell\ge1$ and an $f\in\Cloc(\R^{\Z^2})$, define
\begin{equation}
\label{E:5.1ua}
\Xiin_\ell(f):= 
E\biggl( f\bigl(\tfrac2{\sqrt g}\fraka-\phi_\ell\bigr)\,S_\ell\,\1_{\{S_{\ell}\in[\ell^{1/6},\ell^2]\}}\prod_{x\in \Delta^\ell}\1_{\{\phi_\ell(x)\le\frac2{\sqrt g}\fraka(x)\}}\biggr),
\end{equation}
where
\begin{equation}
\phi_\ell(x):=h^{\Delta^{\ell}}(x)-h^{\Delta^{\ell}}(0).
\end{equation}
Given also an integer~$n\ge\ell$ and a number~$t\in\R$, we also set
\begin{equation}
\label{E:5.2ua}
\Xiout_{n,\ell}(t):=
E\biggl(S_{n-\ell}\,\1_{\{S_{n-\ell}\in[\ell^{1/6},\ell^2]\}}\prod_{x\in \Delta^n\smallsetminus \Delta^{n-\ell}}\1_{\{h^{\Delta^n}(x)\le(m_{2^n}+t)(1-\frakg^{\Delta^n}(x))\}}\,\bigg|\, S_{n+1}=0\biggr).
\end{equation}
The key tool for much of our forthcoming derivations are the following two propositions:

\begin{proposition}
\label{prop-4.14}
Let $f\in\Cloc(\R^{\Z^2})$. For each $\epsilon>0$ there is $\ell_0\ge1$ such that for all $\ell\ge \ell_0$ and all $r\ge \ell$ sufficiently large, we have
\begin{equation}
\label{E:4.31z}
\biggl|\,E_{\nu^0}\Bigl(f\bigl(\tfrac2{\sqrt g}\fraka-\phi)\prod_{x\in \Delta^r}\1_{\{\phi(x)\le\frac2{\sqrt g}\fraka(x)\}}\Bigr)-\frac{1}{\sqrt{\log 2}}\,\frac{\Xiin_\ell(f)}{\sqrt r}\biggr|\le\frac{\epsilon}{\sqrt r}.
\end{equation}
\end{proposition}

\begin{proposition}
\label{prop-4.13}
Let $f\in\Cloc(\R^{\Z^2})$. For each $\epsilon>0$ and each $t_0>0$ there is $\ell_0\ge1$ such that for all $t\in\R$ with $|t|<t_0$, all $\ell\ge\ell_0$ and all $n$ with $\ell\le n^{1/8}$,
\begin{equation}
\label{E:4.30z}
\biggl|\,E\Bigl(\,f\bigl(\frakm_n-h^{\Delta^n}\bigr)\,\1_{\{h^{\Delta^n}\le\frakm_n\}}\,\Big|\,h^{\Delta^n}(0)=0\Bigr)
-\frac 1n\,\frac{2}{g\log2}\Xiin_\ell(f)\,\Xiout_{n,\ell}(t)\biggr|\le \frac \epsilon n,
\end{equation}
where, as before, $\frakm_n(x):=(m_{2^n}+t)(1-\frakg^{\Delta^n}(x))$.
\end{proposition}

We will proceed by giving the proofs of these propositions, starting with Proposition~\ref{prop-4.14} first. We need two lemmas:

\begin{lemma}
\label{lemma-5.3}
We have:
\begin{equation}
\lim_{\ell\to\infty}\,\limsup_{r\to\infty}\sqrt r\,\,E_{\nu^0}\Bigl(\1_{\{S_\ell\le\ell^{1/6}\}}\prod_{x\in \Delta^r}\1_{\{\phi(x)\le\frac2{\sqrt g}\fraka(x)\}}\Bigr)
=0.
\end{equation}
\end{lemma}

\begin{proofsect}{Proof}
Thanks to \twoeqref{E:4.22uu}{E:4.22uuu} and Lemma~\ref{lemma-4.7}, the expectation is at most
\begin{equation}
\frac{c_1\texte^{-c_2(\log\ell)^2}}{\sqrt r}
+P\Bigl(\bigl\{S_\ell\le\ell^{1/6}\bigr\}\cap\bigcap_{j=1}^r\bigl\{S_j\ge-\widetilde R_\ell(j)\bigr\}\Bigr).
\end{equation}
Invoking the argument from the proof of Lemma~\ref{lemma-4.7q}, the second term is bounded by 
\begin{equation}
\bigl(1-\texte^{-c(\log\ell)^2}\bigr)^{-1}P^0\Bigl(\bigl\{B_{t_\ell}\le\ell^{1/6}\bigr\}\cap\bigl\{B_s\ge-2\zeta(s)\colon s\in[0,t_r]\bigr\}\Bigr),
\end{equation}
where $c\in(0,\infty)$ and $\zeta(s):=C[1+\log(\ell\vee (s/\sigma_{\min}^2))]$.
Lemma~\ref{lemma-7.9} (with $x:=\ell^{1/6}$, $u:=t_\ell$ and $t:=t_r$) dominates the probability by $c\ell^{-1/6}/\sqrt r$.
\end{proofsect}

\begin{lemma}
\label{lemma-5.4}
There is a constant~$c\in(0,\infty)$ such that for $\ell\ge1$ sufficiently large and $r\ge\ell$ sufficiently large, 
\begin{equation}
\biggl|\,P\Bigl(\,\bigcap_{j=\ell+1}^r\{S_j\ge 0\}\,\Big|\,\sigma(S_\ell)\Bigr)-\frac1{\sqrt{\log2}}\,\,\frac{S_\ell}{\sqrt r}\biggr|\le c\,\frac{\ell^4}r\,\frac{S_\ell}{\sqrt r}
\end{equation}
holds on the event $\{S_\ell\in[\ell^{1/6},\ell^2]\}$.
\end{lemma}

\begin{proofsect}{Proof}
 We will argue separately for the lower and upper bound on the conditional probability in the statement. First,   
the conditional probability is bounded from below by the probability that the standard Brownian motion started from~$S_\ell$ stays positive for time $t_r-t_\ell$. Since $S_\ell\le\ell^2$ and $t_r$ is at least a constant times~$r$, Lemma~\ref{lem:2.8} shows that this probability is at least
\begin{equation}
\sqrt{\frac 2\pi}\, \frac{S_\ell}{\sqrt{t_r}}\,(1-c\ell^4/r)
\end{equation}
for some constant~$c\in(0,\infty)$. The lower bound follows from the fact that, thanks to Lemma~\ref{lemma-3.4}, for each~$\epsilon>0$ we have $t_r\le(1+\epsilon)(g\log 2)\,r$ as soon as~$r$ is sufficiently large.

For the upper bound we use $\{S_j\ge0\}\subseteq\{S_j\ge-(\log j)^2\}$ for $j=\ell,\dots,r$ and then bound the resulting probability by
\begin{equation}
P\bigl(B_{t_j}\ge-\zeta(t_j)\colon j=\ell,\dots,r\,\big|\, B_{t_\ell}=x\bigr)\quad\text{evaluated at\ \ } x:=S_\ell,
\end{equation}
where $\zeta(s):=[\log(s/\sigma_{\max}^2)]^2$ with~$\sigma_{\max}^2$ denoting the supremum in \eqref{E:4.24y}.
Using the left-hand side of the bound in Lemma~\ref{lemma-4.7q} (integrated with respect to the probability density of the end point over~$y$) this is now  at most 
\begin{equation}
\prod_{j=\ell}^r\biggl(1-\texte^{-2\sigma_{\maxi}^{-2}\,[\log(cj)]^2}\biggr)^{-1}\,
P\Bigl(B_s\ge-2\zeta(s)\colon s\in[t_\ell,t_r]\Big|\, B_{t_\ell}=x\Bigr)\Bigl|_{x:=S_\ell},
\end{equation}
where~$c:=\sigma_{\min}^2/\sigma_{\max}^2$ with $\sigma_{\min}^2$ denoting the infimum in \eqref{E:4.24y}. Invoking Proposition~\ref{prop-4.9ue} with the help of Lemma~\ref{lemma-A.9}, the probability on the right is bounded by 
\begin{equation}
(1+\delta_\ell)\sqrt{\frac2\pi}\,\frac{S_\ell}{\sqrt{t_r-t_\ell}}, 
\end{equation}
uniformly on $\{S_\ell\in[\ell^{1/6},\ell^2]\}$, where~$\delta_\ell$ is a number such that $\delta_\ell\to0$ as~$\ell\to\infty$. The claim again follows by $t_t\ge(1-\epsilon)(g\log 2) r$.
\end{proofsect}

\begin{proofsect}{Proof of Proposition~\ref{prop-4.14} }
Let~$\epsilon>0$ and pick a function $f\in\Cloc(\R^{\Z^2})$. Assume without loss of generality that~$0\le f\le1$ and let $k_0$ be such that~$f$ only depends on the coordinates in~$\Delta^{k_0}$ and let $\ell\ge k_0$. Our task to to replace~$\phi$ by $\phi_\ell$ in the argument of~$f$ as well as in the part of the product corresponding to~$x\in \Delta^\ell$. For this we first note that
\begin{equation}
\phi_\ell(x)=\sum_{j=0}^\ell\Bigl(\frakb_j(x)\varphi_j(0)+\chi_j(x)+h_j'(x)\Bigr).
\end{equation}
For each~$j\in\{1,\dots,\ell\}$, we thus get
\begin{equation}
\label{E:5.7ui}
\max_{x\in \Delta^j}\bigl|\phi(x)-\phi_\ell(x)\bigr|\le c_1 (\log\ell)^2 2^{(j-\ell)/2}\qquad\text{on }\{\widetilde K\le \ell\}.
\end{equation}
Moreover, on $\{\widetilde K\le k\}$, both~$\phi$ and~$\phi_\ell$ are bounded on~$\Delta^{k_0}$ by a $k$-dependent quantity only. Thanks to uniform continuity of~$f$ on compact sets, we thus get
\begin{equation}
\lim_{\ell\to\infty}\limsup_{r\to\infty}\,\sqrt r\,\,E_{\nu^0}\biggl(\,\Bigl|f\bigl(\tfrac2{\sqrt g}\fraka-\phi)-f\bigl(\tfrac2{\sqrt g}\fraka-\phi_\ell)\Bigr|\prod_{x\in \Delta^r}\1_{\{\phi(x)\le\frac2{\sqrt g}\fraka(x)\}}\biggr)=0
\end{equation}
by inserting the indicator of $\{\widetilde K\le k\}$ for $k\ge k_0$, applying \eqref{E:5.7ui} for $\ell\ge k$, invoking \twoeqref{E:4.22uu}{E:4.22uuu} in conjunction with \eqref{E:4.19qq2} and taking the stated limits followed by~$k\to\infty$.

Having brought the argument of~$f$ to what it is supposed to be, for each~$k\ge1$ let $\delta_k>0$ be such that \eqref{E:4.64} holds with $\delta:=\delta_k$ and the above~$\epsilon$. Abbreviate $\phi'(x):=\frac2{\sqrt g}\fraka(x)-\phi(x)$ and define
\begin{equation}
A_{r,\ell,k}:=\{\widetilde K\le k\}\cap\bigl\{\phi'\ge\delta_k\text{ on }\Delta^k\bigr\}\cap\bigl\{S_\ell\in[\ell^{1/6},\ell^2]\bigr\}\cap\bigcap_{j=k+1}^r\bigl\{S_j\ge2\widetilde R_k(j)\bigr\}.
\end{equation}
Lemmas~\ref{lemma-4.7},~\ref{lemma-4.20},~\ref{lemma-entropy-RW} and~\ref{lemma-5.3} show
\begin{equation}
\label{E:5.18ww}
\limsup_{r\to\infty}\sqrt r\,\,E_{\nu^0}\Bigl(\1_{A_{r,\ell,k}^\cc}\prod_{x\in \Delta^r}\1_{\{\phi'(x)\ge0\}}\Bigr)
\le\epsilon+ck^{-\frac1{16}}<2\epsilon
\end{equation}
once $k$ and $\ell\ge k$ are sufficiently large. We fix such a~$k$ for the rest of the argument.

Using the shorthand $\phi_\ell'(x):=\frac2{\sqrt g}\fraka(x)-\phi_\ell(x)$, the bound in \eqref{E:5.18ww} yields
\begin{multline}
\label{E:5.11ui}
\qquad
E_{\nu^0}\Bigl(f(\phi_\ell')\prod_{x\in \Delta^r}\1_{\{\phi'(x)\ge0\}}\Bigr)
\\\le\frac{2\epsilon}{\sqrt r}+E_{\nu^0}\Bigl(f(\phi_\ell')\1_{\{\widetilde K\le k\}}\1_{\{S_\ell\in[\ell^{1/6},\ell^2]\}}\prod_{x\in \Delta^k}\1_{\{\phi'(x)\ge\delta_k\}}\prod_{j=k+1}^r\1_{\{S_j\ge2\widetilde R_k(j)\}}\Bigr).
\quad
\end{multline}
Now pick $\ell\ge k$ so large that the right-hand side of \eqref{E:5.7ui} is less than~$\delta_k$ for $j=1,\dots,k$ and less than~$\widetilde R_k(j)$ for $j=k+1,\dots,\ell$. (This assumes that the constant~$C$ defining~$\widetilde R_k(j)$ was taken large enough.) Then 
\begin{equation}
\bigl\{\phi'\ge\delta_k\text{ on }\Delta^k\bigr\}\subseteq\bigl\{\phi_\ell'\ge0\text{ on }\Delta^k\bigr\}
\end{equation}
while, for $j=k+1,\dots,\ell$,
\begin{equation}
\{\widetilde K\le k\}\cap\bigl\{S_j\ge2\widetilde R_k(j)\bigr\}\subseteq\bigl\{\phi'\ge\widetilde R_k(j)\text{ on }\Delta^j\smallsetminus \Delta^{j-1}\bigr\}\subseteq\{\phi'_\ell\ge0\text{ on }\Delta^j\smallsetminus \Delta^{j-1}\bigr\}.
\end{equation}
It follows that the expectation on the right of \eqref{E:5.11ui} is bounded by
\begin{equation}
\label{E:5.12ui}
E_{\nu^0}\Bigl(f(\phi_\ell')\1_{\{S_\ell\in[\ell^{1/6},\ell^2]\}}\prod_{x\in \Delta^\ell}\1_{\{\phi_\ell'(x)\ge0\}}\prod_{j=\ell+1}^r\1_{\{S_j\ge0\}}\Bigr).
\end{equation}
Conditional on~$S_\ell$, the field $\phi'_\ell$ is independent of $\sigma(S_{\ell+1},\dots,S_r)$. 
Lemma~\ref{lemma-5.4} and the Bounded Convergence Theorem then
yield
\begin{equation}
\label{E:5.15ui}
\lim_{\ell\to\infty}\limsup_{r\to\infty}\,\Bigl|\,\sqrt r\,\,E_{\nu^0}\Bigl(f(\phi_\ell')\1_{\{S_\ell\in[\ell^{1/6},\ell^2]\}}\prod_{x\in \Delta^\ell}\1_{\{\phi_\ell'(x)\ge0\}}\prod_{j=\ell+1}^r\1_{\{S_j\ge0\}}\Bigr) - \Xiin_\ell(f)\Bigr|=0.
\end{equation}
In conjunction with \eqref{E:5.11ui} and the derivations above, this proves one ``half'' of \eqref{E:4.31z}.

To get the other ``half'' of \eqref{E:4.31z}, we in turn define
\begin{equation}
\widetilde A_{r,\ell,k}:=\{\widetilde K\le k\}\cap\bigl\{\phi'\ge0\text{ on }\Delta^r\bigr\}\cap\bigl\{S_\ell\in[\ell^{1/6},\ell^2]\bigr\}
\end{equation}
and use Lemmas~\ref{lemma-4.7},~\ref{lemma-4.20},~\ref{lemma-entropy-RW} and~\ref{lemma-5.3} to get, for $k\le\ell$ large enough, 
\begin{equation}
\label{E:5.17ui}
\limsup_{r\to\infty}\sqrt r\,\,E\Bigl(\1_{\widetilde A_{r,\ell,k}^\cc}\prod_{x\in \Delta^k}\1_{\{\phi'(x)\ge-\delta_k\}}\prod_{j=k+1}^r\1_{\{S_j\ge-2\widetilde R_k(j)\}}\Bigr)<2\epsilon.
\end{equation}
Now, clearly,
\begin{equation}
\begin{aligned}
E_{\nu^0}\Bigl(f(\phi_\ell')\prod_{x\in \Delta^r}\1_{\{\phi'(x)\ge0\}}\Bigr)
&\ge E_{\nu^0}\bigl(f(\phi_\ell')\1_{\widetilde A_{r,\ell,k}}\bigr)
\\
&\ge E_{\nu^0}\Bigl(f(\phi_\ell')\1_{\widetilde A_{r,\ell,k}}\prod_{x\in \Delta^k}\1_{\{\phi'(x)\ge-\delta_k\}}\prod_{j=k+1}^r\1_{\{S_j\ge-2\widetilde R_k(j)\}}\Bigr).
\end{aligned}
\end{equation}
Assuming~$\ell\gg k$, on $\{\widetilde K\le k\}$, which is a subset of $\widetilde A_{r,\ell,k}$, we have $\{\phi'\ge-\delta_k\text{ on }\Delta^k\}\supseteq\{\phi_\ell'\ge0\text{ on }\Delta^k\}$ and $\{S_j\ge-2\widetilde R_k(j)\}\supseteq\{\phi_\ell'\ge0\text{ on }\Delta^j\smallsetminus \Delta^{j-1}\}$ for $j=k+1,\dots,\ell$. By \eqref{E:5.17ui},
\begin{multline}
\qquad
E_{\nu^0}\Bigl(f(\phi_\ell')\prod_{x\in \Delta^r}\1_{\{\phi'(x)\ge0\}}\Bigr)
\ge E_{\nu^0}\Bigl(f(\phi_\ell')\1_{\widetilde A_{r,\ell}}\prod_{x\in \Delta^\ell}\1_{\{\phi_\ell'(x)\ge0\}}\prod_{j=\ell+1}^r\1_{\{S_j\ge0\}}\Bigr)
\\
\ge -\frac{2\epsilon}{\sqrt r}
+E_{\nu^0}\Bigl(f(\phi_\ell')\1_{\{S_\ell\in[\ell^{1/6},\ell^2]\}}\prod_{x\in \Delta^\ell}\1_{\{\phi_\ell'(x)\ge0\}}\prod_{j=\ell+1}^r\1_{\{S_j\ge0\}}\Bigr).
\qquad
\end{multline}
The proof of \eqref{E:4.31z} is finished using \eqref{E:5.15ui}.
\end{proofsect}

In order to prove Proposition~\ref{prop-4.13}, we need substitutes for Lemmas~\ref{lemma-5.3}--\ref{lemma-5.4}:
 
\begin{lemma}
\label{lemma-5.5}
Recall the shorthand $\frakm_n(x):=(m_{2^n}+t)(1-\frakg^{\Delta^n}(x))$. Then
\begin{equation}
\lim_{\ell\to\infty}\,\limsup_{n\to\infty} \,\,n\,E\Bigl(\,\1_{\{S_\ell\wedge S_{n-\ell}\le\ell^{1/6}\}}\prod_{x\in \Delta^n}\1_{\{h^{\Delta^n}(x)\le\frakm_n(x)\}}\,\Big|\,h^{\Delta^n}(0)=0\Bigr)=0.
\end{equation}
\end{lemma}

\begin{lemma}
\label{lemma-5.6}
There is a constant~$c\in(0,\infty)$ such that
\begin{equation}
\biggl|\,P\Bigl(\,\bigcap_{j=\ell+1}^{n-\ell}\{S_j\ge0\}\,\Big|\,\sigma(S_\ell,S_{n-\ell})\Bigr)-\frac2{g{\log2}}\,\,\frac{S_\ell S_{n-\ell}}n\biggr|\le c\frac{\ell^4}n\,\frac{S_\ell S_{n-\ell}}n
\end{equation}
holds on $\{S_\ell,S_{n-\ell}\in[\ell^{1/6},\ell^2]\}$. 
\end{lemma}

\begin{proofsect}{Proofs of Lemmas~\ref{lemma-5.5}--\ref{lemma-5.6}}
These statements are proved by arguments nearly identical to those of Lemmas~\ref{lemma-5.3}--\ref{lemma-5.4}. The following two points are worthy of a note: First, a bound in Lemma~\ref{lemma-7.9} can be used for Brownian motion conditioned on $B_t=0$ (and a slightly worse numerical constant). This follows from the decoupling argument in Lemma~\ref{lemma-decouple}. Second, the constant multiplying~$S_\ell S_{n-\ell}$ is different from that multiplying~$S_\ell$ in Lemma~\ref{lemma-5.3}. This stems from the difference of the prefactors in \eqref{E:4.17q} and \eqref{E:4.16q}. Further details are left to the reader.
\end{proofsect}

\begin{proofsect}{Proof of Proposition~\ref{prop-4.13}}
Pick $f\in\Cloc(\R^{\Z^2})$ and assume without loss of generality that $1\le f\le2$. By routine approximation arguments we can suppose that~$f$ is in fact Lipschitz in the variables that it depends on. Similarly as in the proof above, on $\{K\le k\}\cap\{h^{\Delta^n}(0)=0\}$ we have
\begin{equation}
\max_{x\in \Delta^j}\bigl|h^{\Delta^n}(x)-\phi_\ell(x)\bigr|\le c_1(\log\ell)^2\,2^{-(\ell-k)/2},\qquad j=k,\dots,\ell
\end{equation}
for each $k\le\ell\le n/2$. This and the fact that $\frakm_n(x)\to\frac2{\sqrt g}\fraka(x)$ as $n\to\infty$ (uniformly on compact intervals of~$t$) imply that it suffices to prove \eqref{E:4.30z} for $f(\frakm_n-h^{\Delta^n})$ replaced by $f(\phi_\ell')$, where (as before)
\begin{equation}
\phi_\ell'(x):=\frac2{\sqrt g}\fraka(x)-\phi_\ell(x),
\end{equation}
for any~$\ell$ sufficiently large. Here we again used the inclusions in \twoeqref{E:4.22uu}{E:4.22uuu} in conjunction with \eqref{E:4.19qq2} to bound the contribution of the event when the control variable~$K$ is large.

Let $P^0$, resp., $E^0$ abbreviate the probability, resp., expectation conditional on $\{h^{\Delta^n}(0)=0\}$. Given $\epsilon>0$ and~$k\ge1$, let $\delta_k$ be such that \eqref{E:4.65a} holds with~$\delta:=2\delta_k$ for this~$\epsilon$. For $\ell$ with $k\le\ell\le n/2$, define
\begin{multline}
\qquad
A_{n,\ell,k}:=\{K\le k\}\cap\bigl\{ \frakm_n- h^{\Delta^n}\ge2\delta_k\text{ on }\Delta^k\bigr\}
\\
\cap\bigl\{S_\ell,S_{n-\ell}\in[2\ell^{1/6},\tfrac12\ell^2]\bigr\}\cap\bigcap_{j=k+1}^{n-\ell}\bigl\{S_j\ge 2R_k(j)\bigr\}.
\qquad
\end{multline}
By Lemmas~\ref{lemma-4.7} and~\ref{lemma-4.20}, Lemma~\ref{lemma-entropy-RWB} and Lemma~\ref{lemma-5.5} we have
\begin{equation}
\limsup_{n\to\infty} \,\,n\,E^0\Bigl(\,\1_{A_{n,\ell,k}^\cc}\prod_{x\in \Delta^n}\1_{\{h^{\Delta^n}(x)\le\frakm_n(x)\}}\Bigr)<2\epsilon
\end{equation}
once~$k$ is large enough. We again fix this~$k$ throughout the rest of the argument.

For~$n$ sufficiently large, we get
\begin{multline}
\label{E:5.24}
\quad
E^0\Bigl(\,f(\phi_\ell')\prod_{x\in \Delta^n}\1_{\{h^{\Delta^n}(x)\le \frakm_n(x)\}}\Bigr)
\\
\le\frac{2\epsilon}n+E^0\biggl(\,f(\phi_\ell)\1_{\{S_\ell,S_{n-\ell}\in[2\ell^{1/6},\frac12\ell^2]\}}
\prod_{x\in \Delta^k}\1_{\{\phi_\ell'\ge2\delta_k\}}
\\
\times\prod_{j=k+1}^{n-\ell}\1_{\{S_j\ge2R_k(j)\}}\prod_{x\in \Delta^n\smallsetminus \Delta^{n-\ell}}\1_{\{h^{\Delta^n}(x)\le\frakm_n(x)\}}\biggr).
\quad
\end{multline}
The next important observation is that, conditional on
\begin{equation}
\FF_\ell:=\sigma\bigl(S_1,\dots,S_\ell,S_{n-\ell},\dots,S_{n+1}\bigr)
\end{equation}
the field $\phi_\ell$, the random variables $\{S_j\colon j=\ell,\dots,n-\ell\}$ and the field $\{h^{\Delta^n}(x)\colon x\in \Delta^n\smallsetminus \Delta^{n-\ell}\}$ are independent under~$P^0$. Using Lemma~\ref{lemma-5.6},
the expectation on the right of \eqref{E:5.24} is bounded by
\begin{multline}
\label{E:5.27}
\qquad
\frac1n\Bigl(\frac2{g\log 2}+c\frac{\ell^4}n\Bigr)E^0\biggl(\,f(\phi_\ell')\1_{\{S_\ell,S_{n-\ell}\in[2\ell^{1/6},\frac12\ell^2]\}}S_\ell S_{n-\ell}
\prod_{x\in \Delta^k}\1_{\{\phi_\ell'(x)\ge2\delta_k\}}
\\\times\prod_{j=k+1}^{\ell}\1_{\{S_j\ge2R_k(j)\}}\prod_{x\in \Delta^n\smallsetminus \Delta^{n-\ell}}\1_{\{h^{\Delta^n}(x)\le \frakm_n(x)\}}\biggr).
\quad
\end{multline}
Our task is to dominate the expectation by $(1+o(1))\Xiin_\ell(f)\Xiout_{n,\ell}(t)$. The main issue is to decouple the scales $j\le \ell$ from the scales $j\ge n-\ell$.

First we recall that, by \twoeqref{E:3.31}{E:3.32}, for each $r=0,\dots,n$ the random variables
\begin{equation}
\widetilde S^{(r)}_j:=S_j-\Bigl(\,\sum_{i=0}^{j-1} c_{r-1}(i)\Bigr)S_r,\qquad j=1,\dots,r,
\end{equation}
where $c_n(k)$ are as in \eqref{E:3.33qq}, as well as the field
\begin{equation}
\widetilde\phi_r(x):=\phi_r(x)-\Bigl(\,\sum_{j=0}^{r-1}\frakb_j(x)c_{r-1}(j)\Bigr)S_r,
\end{equation}
are independent of $\sigma(S_r,S_{r+1},\dots,S_{n+1})$ under~$P^0$. Lemma~\ref{lemma-3.4} shows that $0\le c_n(k)\le \tilde c/n$ for some $\tilde c\in(0,\infty)$ and Lemma~\ref{lemma-3.5rt} then ensures that, for some $c\in(0,\infty)$ and all $k=0,\dots,n$,
\begin{equation}
\Bigl|\,\sum_{j=0}^n\frakb_j(x)c_n(j)\Bigr|\le c\frac rn,\qquad x\in \Delta^r.
\end{equation}
It follows that, on $\{S_{n-\ell}\in[\ell^{1/6},\ell^2]\}$ and for $\ell<n^{1/4}$,
\begin{equation}
\max_{x\in \Delta^\ell}\bigl|\phi_\ell'(x)-\widetilde\phi_{n-\ell}'(x)\bigr|\le c{\ell^3}(n-\ell)^{-1}\le 2c \ell^3 n^{-1}
\end{equation}
and
\begin{equation}
0\le S_j-\widetilde S^{(n-\ell)}_j\le \tilde c\ell^{3}n^{-1},\qquad j=0,\dots,\ell.
\end{equation}
Thanks to the Lipschitz property of~$f$, this permits us to bound the expectation in \eqref{E:5.27} by a quantity of the form  $(1+O(n^{-1/4}))$ --- where the $\ell^4$ term comes from bounding $S_\ell S_{n-\ell}$ by their worst case value --- times
\begin{multline}
\label{E:5.33}
\quad 
E^0\biggl(\,f(\widetilde\phi_{n-\ell}')\1_{\{\widetilde S^{(n-\ell)}_\ell\in[\ell^{1/6},\frac12\ell^2]\}}\widetilde S^{(n-\ell)}_\ell \prod_{x\in \Delta^k}\1_{\{\widetilde\phi_{n-\ell}'\ge\delta_k\}}\prod_{j=k+1}^{\ell}\1_{\{\widetilde S^{(n-\ell)}_j\ge R_k(j)\}}
\\
\times
\1_{\{S_{n-\ell}\in[\ell^{1/6},\ell^2]\}} \,S_{n-\ell}\prod_{x\in \Delta^n\smallsetminus \Delta^{n-\ell}}\1_{\{h^{\Delta^n}(x) \le\frakm_n(x)\}}\biggr).
\quad
\end{multline}
Conditional on~$S_{n-\ell}$, the ``small scale'' part of the function under expectation is now independent of the rest and so this \emph{equals} $\Xiout_{n,\ell}(t)$ times
\begin{equation}
E^0\biggl(\,f(\widetilde\phi_{n-\ell}')\1_{\{\widetilde S^{(n-\ell)}_\ell\in[\ell^{1/6},\frac12\ell^2]\}}\widetilde S^{(n-\ell)}_\ell \prod_{x\in \Delta^k}\1_{\{\widetilde\phi_{n-\ell}'\ge\delta_k\}}\prod_{j=k+1}^{\ell}\1_{\{\widetilde S^{(n-\ell)}_j\ge R_k(j)\}}
\1_{\{S_{n-\ell}\in[\ell^{1/6},\ell^2]\}}
\biggr).
\end{equation}
The last indicator permits us to return all $\widetilde\phi_{n-\ell}'$ back to $\phi_\ell'$ and all~$\widetilde S^{(n-\ell)}_j$ back to $S_j$ at the cost of another multiplicative factor $(1+O(\ell^5 n^{-1}))$. Dropping the indicator and applying the argument \twoeqref{E:5.11ui}{E:5.12ui} then yields the expectation in the definition of $\Xiin_\ell(f)$.

The complementary (lower) bound is proved by modifications similar to those used in the proof of Proposition~\ref{prop-4.14}. We omit the details.
\end{proofsect}

\subsection{Extracting the cluster law}
\label{sec-5.2}\noindent
We are now ready to harvest the first fruits of our hard work in the previous sections. In particular, we will establish the existence of the cluster law and prove the asymptotic in Theorem~\ref{thm-2.5}. We begin by noting that the estimates in Lemma~\ref{lemma-4.9new} imply uniform bounds on the quantities $\Xiin_\ell(1)$ and $\Xiout_{n,\ell}(t)$ from \twoeqref{E:5.1ua}{E:5.2ua}.

\begin{lemma}
There are $c_1,c_2\in(0,\infty)$ such that for all sufficiently large $\ell$,
\begin{equation}
\label{E:4.32z}
c_1<\Xiin_\ell(1)<c_2.
\end{equation}
Moreover, for each~$t_0>0$ there are $c_1',c_2'\in(0,\infty)$ such that for all sufficiently large $\ell$, all $n\ge\ell^8$ and all $t\in\R$ with~$|t|<t_0$, we also have
\begin{equation}
\label{E:4.33z}
c_1'<\Xiout_{n,\ell}(t)<c_2'.
\end{equation}
\end{lemma}

\begin{proofsect}{Proof}
Combining \eqref{E:4.31z} for $f:=1$ with \eqref{E:4.25z} directly yields \eqref{E:4.32z}. A similar reasoning then shows that \eqref{E:4.30z} and \eqref{E:4.24z} imply \eqref{E:4.33z}.
\end{proofsect}

With these bounds in hand, Proposition~\ref{prop-4.14} readily yields the asymptotic of the probability of the conditional event that leads to the definition of the cluster law:

\begin{proofsect}{Proof of Theorem~\ref{thm-2.5}}
 Using $\phi\leftrightarrow-\phi$ symmetry of~$\nu^0$, the  bound in \eqref{E:4.31z} reads
\begin{equation}
\sqrt r\,\nu^0\biggl(\,\phi_x+\frac2{\sqrt g}\,\fraka(x)\ge0\colon\, |x|\le 2^r\biggr)= 
\frac{\Xiin_\ell(1)}{\sqrt{\log 2}}+\epsilon_r(\ell),
\end{equation}
where $\lim_{\ell\to\infty}\limsup_{r\to\infty}|\epsilon_r(\ell)|=0$. Since the left-hand side is independent of~$\ell$, taking~$r\to\infty$ followed by~$\ell\to\infty$ along suitable subsequences shows that the \emph{limes superior} of the left-hand side is less than the \emph{limes inferior} of the right-hand side and, similarly, the \emph{limes inferior} of the left-hand side is at least the \emph{limes superior} of the right-hand side. It follows that both limits in
\begin{equation}
\lim_{r\to\infty}\sqrt r\,\nu^0\biggl(\,\phi_x+\frac2{\sqrt g}\,\fraka(x)\ge0\colon\, |x|\le 2^r\biggr)=\lim_{\ell\to\infty}\frac{\Xiin_\ell(1)}{\sqrt{\log 2}}
\end{equation}
exist and are related as stated. By \eqref{E:4.32z} the right-hand side is also positive and finite. The claim then follows with $\tilde c^\star:=\lim_{\ell\to\infty}\Xiin_\ell(1)$ (which we denote by $\Xiin_\infty(1)$ later) by noting that, thanks to monotonicity in~$r$ of the probability in \eqref{E:2.9y} and the slowly-varying nature of $r\mapsto\log r$, it suffices to prove the desired asymptotic only for radii varying along powers of~$2$.
\end{proofsect}

Proposition~\ref{prop-4.14} permits us to work with a large class of test functions~$f$. This in particular permits us to prove the existence of the limit in \eqref{E:1.14a}. (We state this as a separate proposition because we will only identify the limit measure with that in Theorem~\ref{thm-main} later.)

\begin{proposition}[Existence of cluster law]
\label{cor-4.16}
For every $f\in\Cloc(\R^{\Z^2})$, the limit
\begin{equation}
\label{E:5.40a}
\Xiin_\infty(f):=\lim_{\ell\to\infty}\Xiin_\ell(f)
\end{equation}
exists and is finite. Moreover, also the following limit exists and obeys
\begin{equation}
\label{E:5.40}
\lim_{r\to\infty}\frac{E_{\nu^0}(f(\phi+\frac2{\sqrt g}\fraka)\,\prod_{x\colon|x|\le r}\,\1_{\{\phi(x)+\frac2{\sqrt g}\fraka(x)\ge0\}})}{\nu^0(\phi(x)+\frac2{\sqrt g}\fraka(x)\ge0\colon |x|\le r)}
=\frac{\Xiin_\infty(f)}{\Xiin_\infty(1)}.
\end{equation}
In addition, there is a probability measure~$\nu$ on $[0,\infty)^{\Z^d}$ such that the limit equals~$E_\nu(f(\phi))$. This measure has finite level sets almost surely; in fact, for any $c\in(0,\infty)$,
\begin{equation}
\label{E:5.42}
\nu\Bigl(\bigl\{\phi\not\ge c(\log k)^2\text{\rm\ on }\Delta^k\smallsetminus \Delta^{k-1}\bigr\}\text{\rm\ i.o.}\Bigr)=0.
\end{equation}
\end{proposition}

\begin{proofsect}{Proof}
Abbreviate  (with a slight abuse of our earlier notation)  $\phi'(x):=\phi(x)+\frac2{\sqrt g}\fraka(x)$. By Proposition~\ref{prop-4.14} we have
\begin{equation}
\frac{E_{\nu^0}(f(\phi')\prod_{x\in \Delta^r}\1_{\{\phi'(x)\ge0\}})}{\nu^0(\phi'(x)\ge0\colon x\in \Delta^r)}
=\frac{\Xiin_\ell(f)+\tilde\epsilon_r(\ell)}{\Xiin_\ell(1)+\epsilon_r(\ell)}\,,
\end{equation}
where $\lim_{\ell\to\infty}\limsup_{r\to\infty}|\tilde\epsilon_r(\ell)|=0$ and similarly for~$\epsilon_r(\ell)$.  Since the left hand side does not depend on~$\ell$ while the only dependence on~$r$ of the right hand side comes through~$\tilde\epsilon_r(\ell)$ and $\epsilon_r(\ell)$ which vanish as~$r\to\infty$, the argument from the previous proof shows that both sides converge. 
That argument  also gave the existence of the positive limit $\lim_{\ell\to\infty}\Xiin_\ell(1)$ and so also the limit in~\eqref{E:5.40a} exists for all $f\in\Cloc(\R^{\Z^2})$. This proves \eqref{E:5.40} for~$r$ running along powers of~$2$; for the general~$r\to\infty$ we then assume~$f\ge0$ and argue by monotonicity and existence of the limit in~\eqref{E:2.9y} (already proved above).

The linear functional $f\mapsto \xi(f):=\Xiin_\infty(f)/\Xiin_\infty(1)$ is positive and so,  when restricted to the subspace $\Cc(\R^{\Delta^j})$ of compactly supported functions in~$\Cb(\R^{\Delta^j})$,  continuous and bounded in the supremum norm. By the Riesz-Markov-Kakutani representation theorem, there is a regular Borel measure~$\nu_j$ on~$\R^{\Delta^j}$ such that $\xi(f)=\int \nu_j(\textd\phi)f(\phi)$ for each~$f\in \Cc(\R^{\Delta^j})$. In order to show that~$\nu_j$ is a probability measure, we have to prove tightness  (note that $f:=1$ is not in~$\Cc(\R^{\Delta^j})$).  Here we note that \eqref{E:4.17ueu} and the definition of $R_{\widetilde K}(j)$ imply
\begin{equation}
\bigl\{\phi'\not\le 2k^2\text{\rm\ on }\Delta^j\bigr\}\subseteq\{\widetilde K>k\},\qquad 1\le j\le k,
\end{equation}
once~$k$ is sufficiently large.
Lemmas~\ref{lemma-4.7} and~\ref{lemma-entropy-RW} then show
\begin{equation}
\nu^0\bigl(\phi'\not\le 2k^2\text{\rm\ on }\Delta^j\,\big|\,\phi'(x)\ge0\colon x\in \Delta^r\bigr)\le c\texte^{-c'(\log k)^2}.
\end{equation}
Standard approximation arguments extend this bound to the limit measure; taking $k\to\infty$ then shows $\nu_j(\R^{\Delta^j})=1$. The measures $\{\nu_j\colon j\ge1\}$ are clearly consistent; the Kolmogorov Extension Theorem then ensures that they are restrictions of a unique probability measure $\nu$ on~$\R^{\Z^2}$.  The above tightness argument then also ensures that this  measure obeys $\xi(f)=E_\nu(f(\phi))$ for all~$f\in\Cloc(\R^{\Z^2})$ as desired.

It remains to prove \eqref{E:5.42}. For this we note
\begin{equation}
\bigl\{\phi'\not\ge c(\log k)^2\text{\rm\ on }\Delta^k\smallsetminus \Delta^{k-1}\bigr\}\cap\{K\le k\}
\subseteq \bigl\{S_k\le \tilde c[1+(\log k)^2]\bigr\}.
\end{equation}
Lemmas~\ref{lemma-4.7} and~\ref{lemma-entropy-RW} then show that, for any $k\le k'\le r$ with~$k$ sufficiently large,
\begin{equation}
\nu^0\Bigl(\,\bigcup_{j=k}^{k'}\bigl\{\phi'\not\ge c(\log k)^2\text{\rm\ on }\Delta^k\smallsetminus \Delta^{k-1}\bigr\}\,\Big|\,\phi'(x)\ge0\colon x\in \Delta^r
\Bigr)\le c'k^{-\frac1{16}}
\end{equation}
once~$r$ is large. Taking $r\to\infty$ followed by~$k\to\infty$ then yield the claim. 
\end{proofsect}

\subsection{Limit of full extreme process}
\label{sec-5.3}\noindent
We will now move to the proof of the distributional convergence in Theorem~\ref{thm-main} including the characterization of the cluster law in Theorem~\ref{thm-lessmain}. This will be done modulo the proof of a technical Proposition~\ref{lemma-4.18} which is deferred to the next subsection.

Let~$D\in\mathfrak D$. Given a Radon measure $\eta$ on $ \overline D\times\R\times\R^{\Z^d}$ and a measurable function $f\colon \overline D\times\R\times\R^{\Z^d}\to[0,\infty)$, we will abbreviate
\begin{equation}
\langle \eta,f\rangle:=\int\eta(\textd x\,\textd h\,\textd\phi)f(x,h,\phi).
\end{equation}
Let~$D_N$ be related to~$D$ as in \twoeqref{E:1.1}{E:1.1a} and let~$h^{D_N}$ be the DGFF in~$D_N$. Recall the notation
\begin{equation}
\label{E:5.47}
\Gamma_N^D(t):=\bigl\{x\in D_N\colon h^{D_N}(x)\ge m_N-t\bigr\}
\end{equation}
and, using the notation $\Lambda_r(x)$ from \eqref{E:Lambda_r}, write
\begin{equation}
\Theta^D_{N,r}:=\bigl\{x\in D_N\colon h^{D_N}(x)=\max_{z\in\Lambda_r(x)}h^{D_N}(z)\bigr\}
\end{equation}
for the set of points in $D_N$ that are $r$-local extrema.
Our starting point is the following lemma:

\begin{lemma}
\label{lemma-4.17}
For any~$r_N\to\infty$ with $N/r_N\to\infty$ and any continuous $f\colon \overline D\times\R\times\R^{\Z^d}\to[0,\infty)$ with compact support,
\begin{equation}
\lim_{r\to\infty}\,\,\limsup_{N\to\infty}\,\,\max_{M\colon r\le M\le N/r}\,\Bigl|\, E\bigl(\texte^{-\langle\eta_{N,r_N}^D,f\rangle}\bigr)
-E\bigl(\texte^{-\langle\eta_{N,M}^D,f\rangle}\bigr)\Bigr|=0.
\end{equation}
\end{lemma}

\begin{proofsect}{Proof}
By assumption~$f(x,h,\phi)=0$ is zero unless $|h|\le\lambda$, for some~$\lambda>0$.  The inequality  $\langle\eta_{N,r_N}^D,f\rangle\ne\langle\eta_{N,M}^D,f\rangle$ for some $M$ between~$r$ and~$N/r$ implies $(\Theta^D_{N,r_N}\triangle\Theta^D_{N,M})\cap\Gamma_N(\lambda)\ne\emptyset$. But this in turn forces the existence of two local maxima in~$\Gamma_N(\lambda)$ that are farther than $M\wedge r_N$ and yet closer than $M\vee r_N$. Hence,
for $N$ so large that $r_N\ge r$ and $r_N\le N/r$, 
\begin{multline}
\label{E:3.44qu}
\qquad
\max_{M\colon r\le M\le N/r}\,\Bigl|E\bigl(\texte^{-\langle\eta^D_{N,r_N},f\rangle}\bigr)-E\bigl(\texte^{-\langle\eta^D_{N,M},f\rangle}\bigr)\Bigr|
\\
\le P\Bigl(\,\exists x,y\in\Gamma_N^D(\lambda)\colon r\le|x-y|\le N/r\Bigr).
\qquad
\end{multline}
The right-hand side tends to zero in the stated limits by Lemma~\ref{lemma-separation}.
\end{proofsect}

Thanks to  Lemma~\ref{lemma-4.17}, we may as well pick any~$M\in\{r,\dots,\lfloor N/r\rfloor\}$ and work with $\eta_{N,M}^D$ in place of~$\eta^D_{N,r_N}$. We will choose
\begin{equation}
\label{e:5.61}
M=M(N,r):=\max\{2^n\colon 2^n\le N/r\}.
\end{equation}
Our next step is to introduce an auxiliary process 
\begin{equation}
\widehat\eta^D_{N,M}:=\sum_{x\in D_N}
\1_{\{h_x^{D_N}=\max_{z\in \Lambda_M(x)}(h^{D_N}_x-\Phi^{M,x}_z)\}}
\delta_{\,x/N}\otimes\delta_{\,h^{D_N}_x-m_N}\otimes\delta_{\,\{h^{D_N}_x-h^{D_N}_{x+z}+\Phi^{M,x}_{x+z}\colon z\in\Z^2\}},
\end{equation}
where, using the notation $H^D(x,y)$ for the harmonic measure from~$x$ to~$y$ in~$D$, 
\begin{equation}
\label{E:4.49q}
\Phi^{M,x}(z):=\sum_{y\in D_N\cap\partial\Lambda_M(x)}H^{(D_N\cap\Lambda_M(x))\smallsetminus\{x\}}(z,y) \,h^{D_N}(y).
\end{equation}
Note that $\Phi^{M,x}=h^{D_N}(x)$ outside~$\Lambda_M(x)$ and also $\Phi^{M,x}(x)=0$ because $y:=x$ is not included in the sum. As it turns out, the laws of the processes~$\eta^D_{N,M}$ and $\widehat\eta^D_{N,M}$ are very close:

\begin{proposition}
\label{lemma-4.18}
For any $f\colon \overline D\times\R\times\R^{\Z^d}\to[0,\infty)$ that is continuous with compact support and depends only on a finite number of coordinates of~$\phi$,
\begin{equation}
\lim_{r\to\infty}\,\limsup_{N\to\infty}\,\Bigl|\, E\bigl(\texte^{-\langle\eta_{N,M(N,r)}^D,f\rangle}\bigr)
-E\bigl(\texte^{-\langle\widehat\eta_{N,M(N,r)}^D,f\rangle}\bigr)\Bigr|=0.
\end{equation}
\end{proposition}

We defer the proof of this proposition to the next subsection and  instead proceed with the proof of the point process convergence. There are several reasons why $h^{D_N}-\Phi^{M,\cdot}$ is more convenient to work with than~$h^{D_N}$. All of them can be deduced from the following observation:

\begin{lemma}
\label{lemma-5.10q}
Suppose~$x\in D_N$ is such that $\Lambda_M(x)\subset D_N$ and let $n\in\N$ be such that $M=2^n$. Consider the $\sigma$-algebra
\begin{equation}
\FF_{M,\,x}:=\sigma\bigl(h^{D_N}(z)\colon z\in \{x\}\cup\Lambda_{M}(x)^\cc\bigr).
\end{equation}
Then for Lebesgue almost every~$t\in\R$,
\begin{multline}
\quad
P\bigl(h^{D_N}(x+\cdot)-\Phi^{M,x}(x+\cdot)\in \cdot\big|\FF_{M,x}\bigr)
\\
=P\bigl(h^{\Delta^n\smallsetminus\{0\}}+t\frakg^{\Delta^n}\in\cdot\bigr)
=P\bigl(h^{\Delta^n}\in\cdot\,\big|\,h^{\Delta^n}(0)=t\bigr)\quad\text{\rm on }\{h^{D_N}(x)=t\}.
\quad
\end{multline}
In particular, conditional on $\sigma(h^{D_N}(x))$, the conditional probability on the left-hand side is independent of $\sigma(h^{D_N}(z)\colon z\not\in\Lambda_M(x))$.
\end{lemma}

\begin{proofsect}{Proof}
Our choice of~$M$ ensures~$\Lambda_M(0)=\Delta^n$. By the Gibbs-Markov property (Lemma~\ref{lemma-GM}) and the fact that $H^{(D_N\cap\Lambda_M(x))\smallsetminus\{x\}}(x+z,x)=\frakg^{\Delta^n}(z)$ we have
\begin{equation}
\Phi^{M,x}(x+z)+\frakg^{\Delta^n}(z)h^{D_N}(x)=E\bigl(h^{D_N}(x+z)\,\big|\,\sigma(h^{D_N}(y)\colon y\in\{x\}\cup\Lambda_M(x)^\cc)\bigr).
\end{equation}
Under the assumption that $\Lambda_M(x)\subset D_N$ we thus get
\begin{equation}
h^{D_N}(x+\cdot)-\Phi^{M,x}(x+\cdot)-\frakg^{\Delta^n}(\cdot)h^{D_N}(x)\,\laweq\,h^{\Delta^n\smallsetminus\{0\}}(\cdot).
\end{equation}
Conditional on~$\scrF_{M,x}$ the last two terms on the left-hand side are effectively constant, with the very last one 
equal to $\frakg^{\Delta^n}(\cdot)t$ on $\{h^{D_N}(0)=t\}$. The claim follows by noting, as observed in the proof of Lemma~\ref{lemma-3.1}, $h^{\Delta^n\smallsetminus\{0\}}+t\frakg^{\Delta^n}$ has the law of $h^{\Delta^n}$ conditioned on $h^{\Delta^n}(0)=t$.
\end{proofsect}

The role of the auxiliary process is to arrange that, after conditioning on the position and value at each relevant local maximum, the ``clusters'' associated with distinct local maxima are independent. As foreseen in \twoeqref{E:3.1}{E:3.2}, the following law  naturally arises,
\begin{equation}
\nu^{(M,t)}(\cdot) := P\Bigl(h^{\Delta^n}(0)-h^{\Delta^n}\in\cdot\,\Big|\,h^{\Delta^n}(0)=m_N+t,\, h^{\Delta^n}\le h^{\Delta^n}(0)\Bigr),
\end{equation}
where $n\in\N$ is such that $M=2^n$.
Indeed, we have:

\begin{lemma}
\label{lemma-5.12}
Let~$f=f(x,h,\phi)\colon D\times\R\times\R^{\Z^2}\to[0,\infty)$ be continuous with compact support and depending only on~$\{\phi(x)\colon x\in\Lambda_M(x)\}$. Define $\tilde f_{N,r}\colon D\times\R\to[0,\infty)$ by
\begin{equation}
\label{e:5.70}
\texte^{-\tilde f_{N,r}(x,h)}:=E_{\nu^{(M,h)}}\bigl(\,\texte^{-f(x,h,\phi)}\bigr),
\end{equation}
where $M=M(N,r)$ is as above. Then there is $r_0=r_0(f)$ such that for all $r\ge r_0$ and all~$N$ sufficiently large,
\begin{equation}
E\bigl(\texte^{-\langle\widehat\eta^D_{N,M},f\rangle}\bigr) = E\bigl(\texte^{-\langle\widehat\eta^D_{N,M},\tilde f_{N,r}\rangle}\bigr).
\end{equation}
\end{lemma}

\begin{proofsect}{Proof}
Recall our notation $\Theta_{N,r}^D$ for the set of sites in~$D_N$ where~$h^{D_N}$ has an $r$-local maximum and let~$\widehat\Theta_{N,M}^D$ analogously denote the set
\begin{equation}
\widehat\Theta_{N,M}^D:=\bigl\{x\in D_N\colon \max_{z\in\Lambda_M(x)}h^{D_N}(z)-\Phi^{M,x}(z)=h^{D_N}(x)\bigr\}.
\end{equation}
Using inclusion-exclusion, we then get
\begin{multline}
\label{E:5.73}
\quad
E\bigl(\texte^{-\langle\widehat\eta^D_{N,M},f\rangle}\bigr)=1
\\+\sum_{n=1}^{|D_N|}\sum_{\begin{subarray}{c}
A\subset D_N\\|A|=n
\end{subarray}}
E\biggl(\,\prod_{x\in A}\Bigl(\bigl(\texte^{-f(x/N,\,h^{D_N}(x)-m_N,\,h^{D_N}(x)-h^{D_N}(x+\cdot)+\Phi^{M,x}(x+\cdot))}-1\bigr)\1_{\{x\in\widehat\Theta_{N,M}^D\}}\Bigr)\biggr),
\end{multline}
where the sum actually terminates at the maximal number of distinct translates of~$\Lambda_{M}(0)$ one can center at points of~$D_N$ so that each center point belongs to only one of these sets. 

Since $h^{D_N}$ is continuously distributed, the collection of sets $\{\Lambda_M(x)\colon x\in A\}$ is a.s.\ disjoint for any~$A$ contributing (non-trivially) to the above sum. Thanks to our assumptions on~$f$, as soon as~$r$ is large enough (fixed) and~$N$ is larger than a constant times~$r$, we may assume that $\Lambda_M(x)\subset D_N$ for each~$x\in A$. Under such conditions Lemma~\ref{lemma-5.10q} tells us that, a.s.,
\begin{multline}
\qquad
E\Bigl(\texte^{-f(x/N,\,h^{D_N}(x)-m_N,\,h^{D_N}(x)-h^{D_N}(x+\cdot)+\Phi^{M,x}(x+\cdot))}
\1_{\{x\in\widehat\Theta_{N,M}^D\}}\,\Big|\,\FF_{M,x}\Bigr)
\\
=E_{\nu^{(M,t)}}\bigl(\texte^{-f(x/N,\,t,\phi)}\bigr)\Big|_{t := h^{D_N}(x)-m_N}
E\bigl(\1_{\{x\in\widehat\Theta_{N,M}^D\}}\,\big|\FF_{M,x}\bigr)
\qquad
\end{multline}
and by conditioning on~$\sigma(\FF_{M,x}\colon x\in A)$ we thus get
\begin{multline}
\qquad
E\biggl(\,\,\prod_{x\in A}\Bigl[\bigl(\texte^{-f(x/N,\,h^{D_N}(x)-m_N,\,h^{D_N}(x)-h^{D_N}(x+\cdot)+\Phi^{M,x}(x+\cdot))}-1\bigr)\1_{\{x\in\widehat\Theta_{N,M}^D\}}\Bigr]\biggr)
\\
=E\biggl(\,\prod_{x\in A}\Bigl[\bigl(\texte^{-\tilde f_{N,r}(x/N,\,h^{D_N}(x)-m_N)}-1\bigr)\1_{\{x\in\widehat\Theta_{N,M}^D\}}\Bigr]\biggr).
\qquad
\end{multline}
Plugging this back into the inclusion-exclusion formula \eqref{E:5.73}, the claim follows.
\end{proofsect}

Another reason why the auxiliary process is useful to work with is seen from:

\begin{lemma}
\label{lemma-5.13a}
Fix any $r\ge1$, any $j\ge1$ and any $c_1\in(0,\infty)$. For~$M=M(N,r)$ as above, uniformly in $f\in \Cb(\R^{\Delta^j})$ with $\Vert f\Vert_\infty\le c_1$ and uniformly in~$t$ on compact sets in~$\R$,
\begin{equation}
E_{\nu^{(M,t)}}(f)\,\underset{N\to\infty}\longrightarrow\, E_\nu(f),
\end{equation}
where~$\nu$ is the measure from Proposition~\ref{cor-4.16}.
\end{lemma}

\begin{proofsect}{Proof}
Let~$n$ be such that $M=2^n$ and let, as before, $\frakm_n(x)=(m_N+t)(1-\frakg^{\Delta^n}(x))$. From Proposition~\ref{prop-4.13} and $h^{\Delta^n}\leftrightarrow -h^{\Delta^n}$ symmetry we get
\begin{equation}
E_{\nu^{(M,t)}}(f)
=\frac{E(f(\frakm_n-h^{\Delta^n})\1_{\{h^{\Delta^n}\le\frakm_n\}}\,|\,h^{\Delta^n}(0)=0)}
{E(\1_{\{h^{\Delta^n}\le\frakm_n\}}\,|\,h^{\Delta^n}(0)=0)}
=\frac{\Xiin_\ell(f)\Xiout_{n,\ell}(t)+\tilde\epsilon_n(\ell)}{\Xiin_\ell(1)\Xiout_{n,\ell}(t)+\epsilon_n(\ell)},
\end{equation}
where $\epsilon_n(\ell)$ and~$\tilde\epsilon_n(\ell)$ tend to zero as~$n\to\infty$ followed by~$\ell\to\infty$, uniformly on compact sets of~$t\in\R$ and for~$f$ as above. In light of the bounds in \eqref{E:4.33z}, the right-hand side tends to that of \eqref{E:5.40} which by Proposition~\ref{cor-4.16} equals $E_\nu(f)$.
\end{proofsect}

We are finally ready to give the proof of the first main result of this paper:

\begin{proofsect}{Proof of Theorem~\ref{thm-main}}
Let~$f\colon D\times\R\times\R^{\Z^2}\to[0,\infty)$ be a function $f=f(x,h,\phi)$ that is continuous with compact support and dependent only on a finite number of coordinates of~$\phi$. Consider the expectation $E(\texte^{-\langle\widehat\eta^D_{N,M(N,r)},f\rangle})$ which, thanks to Lemma~\ref{lemma-5.12}, we can replace by $E(\texte^{-\langle\widehat\eta^D_{N,M(N,r)},\tilde f_{N,r}\rangle})$ once~$N\gg r\gg1$ with $\tilde f_{N,r}$ as in~\eqref{e:5.70}. 
Our aim is to derive a limit for this expectation using the results of Biskup and Louidor~\cite{BL2}, but for that we will need to replace $\tilde f_{N,r}$ by a continuous, compactly-supported function that does not depend on~$N$ and~$r$.
	  
For concreteness suppose that $f(x,h,\phi)$ is zero unless $|h|\le\lambda$ and that~$f$ depends only on $\{\phi(x)\colon x\in\Lambda_{r_0}(0)\}$ for some $r_0\ge1$. Recall also the definition of $M$ and $n$ from~\eqref{e:5.61}. By our assumptions on~$f$, the functions $\{\tilde f_{N,r}\colon  N,r\ge1\}$ are all supported in the same compact set in~$D\times\R$. Lemma~\ref{lemma-5.13a} then gives that, for each~$r\ge1$,
\begin{equation}
\tilde f_{N,r}(x,h)\,\underset{N\to\infty}\longrightarrow\, \tilde f(x,h),\qquad\text{uniformly in~$x$ and~$h$},
\end{equation}
where $\tilde f$ is defined by
\begin{equation}
\texte^{-\tilde f(x,h)}=E_\nu\bigl(\texte^{-f(x,h,\phi)}\bigr).
\end{equation}
Both~$\tilde f_{N,r}$ and~$\tilde f$ vanish unless~$|h|\le\lambda$ and so
\begin{equation}
\label{E:5.79}
\Bigl|\bigl\langle\widehat\eta^D_{N,M(N,r)},\,(\tilde f-\tilde f_{N,r})\bigr\rangle\Bigr|\le 2\Vert \tilde f-\tilde f_{N,r}\Vert_\infty
|\Gamma_N^D(\lambda)|\,,
\end{equation}
where the random variables~$\{|\Gamma_N^D(\lambda)|\colon N\ge1\}$ are known to be tight thanks to Lemma~\ref{lemma-level-set}.
Using this in conjunction with Lemma~\ref{lemma-4.17}, Proposition~\ref{lemma-4.18} and Lemma~\ref{lemma-5.12} yields
\begin{equation}
\begin{aligned}
\label{E:4.52}
E\bigl(\texte^{-\langle\eta^D_{N,r_N},f\rangle}\bigr)+o(1)
&=E\bigl(\texte^{-\langle\widehat\eta^D_{N,M(N,r)}, f\rangle}\bigr)
\\
&=E\bigl(\texte^{-\langle\widehat\eta^D_{N,M(N,r)}, \tilde f_{N,r}\rangle}\bigr)
\\
&=E\bigl(\texte^{-\langle\widehat\eta^D_{N,M(N,r)}, \tilde f\rangle}\bigr)+o(1)
\\
&=E\bigl(\texte^{-\langle\eta^D_{N,r_N},\tilde f\rangle}\bigr)+o(1),
\end{aligned}
\end{equation}
where the~$o(1)$ terms tend to zero in the limits $N\to\infty$ followed by~$r\to\infty$.

Since~$\tilde f\colon D\times\R\to[0,\infty)$ is continuous with compact support, Theorem~2.1 of Biskup and Louidor~\cite{BL2} shows
\begin{equation}
\label{E:4.53}
E\bigl(\texte^{-\langle\eta^D_{N,r_N},\tilde f\rangle}\bigr)
\,\underset{N\to\infty}\longrightarrow\,
E\biggl(\exp\Bigl\{-\int Z^D(\textd x)\otimes\texte^{-\alpha h}\textd h\,\,(1-\texte^{-\tilde f(x,h)})\Bigr\}\biggr).
\end{equation}
Invoking the definition of~$\tilde f$, the integral in \eqref{E:4.53} can be recast as
\begin{equation}
\int Z^D(\textd x)\otimes\texte^{-\alpha h}\textd h\otimes\nu(\textd\phi)\,\,(1-\texte^{-f(x,h,\phi)}).
\end{equation}
This yields the desired expression
\begin{equation}
E\bigl(\texte^{-\langle\eta^D_{N,r_N},f\rangle}\bigr)
\,\underset{N\to\infty}\longrightarrow\,
E\biggl(\exp\Bigl\{-\int Z^D(\textd x)\otimes\texte^{-\alpha h}\textd h\otimes\nu(\textd\phi)\,\,(1-\texte^{-f(x,h,\phi)})\Bigr\}\biggr)
\end{equation}
whenever~$f$ is as assumed above. But the class of such~$f$ is dense in the class of all continuous compactly-supported functions $f\colon D\times\R\times\R^{\Z^2}\to[0,\infty)$ with respect to the supremum norm, and so the claim follows by invoking a version of the estimate \eqref{E:5.79} (on the left-hand side) and the fact that~$Z^D(D)<\infty$ a.s.\ (on the right-hand side).
\end{proofsect}

\begin{proofsect}{Proof of Theorem~\ref{thm-lessmain}}
Let~$\nu$ be the measure constructed by the limit in Proposition~\ref{cor-4.16}. The proof of Theorem~\ref{thm-main} (specifically, Lemma~\ref{lemma-5.13a}) then shows that~$\nu$ is indeed the cluster law.
\end{proofsect}

\begin{proofsect}{Proof of Corollary~\ref{cor-cluster-process}}
Let~$f\colon\overline D\times\R\to[0,\infty)$ be a continuous function with compact support. For~$r\ge1$ define $f_r\colon\overline D\times\R\times\R^{\Z^2}\to[0,\infty)$ by
\begin{equation}
f_r(x,h,\phi):=\sum_{z\in\Lambda_r(0)}f(x,h-\phi_z).
\end{equation}
Let~$\lambda>0$ be such that $f(x,h)=0$ unless~$|h|\le\lambda$. We then observe that, if~$r_N>2r$, on the event when~$x,y\in\Gamma_N^D(\lambda)$ imply either $|x-y|<r$ or~$|x-y|>r_N$ (and assuming that no two values of~$h^{D_N}$ are the same) each point contributing to $\langle\eta_N^D,f\rangle$ lies within $r$-neighborhood of a unique $r$-local maximum of~$h^{D_N}$. Under such circumstances we have 
\begin{equation}
\langle \eta_N^D,f\rangle = \sum_{x\in D_N}\sum_{z\in\Lambda_r(x)}\1_{\{h^{D_N}_x=\max_{z\in\Lambda_r(x)}h^{D_N}_z\}} f\bigl(x/N,h^{D_N}_x-(h^{D_N}_x-h^{D_N}_z)\bigr)
\end{equation}
But on the same event we can freely replace $\Lambda_r(x)$ in the indicator on the right-hand side by~$\Lambda_{r_N}(x)$ thus leading to $\langle \eta_N^D,f\rangle=\langle\eta_{N,r_N}^D, f_r\rangle$. Using Lemma~\ref{lemma-separation}, it then follows that
\begin{equation}
\lim_{r\to\infty}\limsup_{N\to\infty}P\bigl(\langle \eta_N^D,f\rangle\ne\langle\eta_{N,r_N}^D, f_r\rangle\bigr)=0.
\end{equation}
As the maximum of $h^{D_N}-m_N$ is tight, Theorem~\ref{thm-main} can be applied for the test function~$f_r$ despite the fact that it does not have compact support. Invoking also the Monotone Convergence Theorem to deal with the limit~$r\to\infty$, we get
\begin{equation}
\label{E:5.88a}
E\bigl(\texte^{-\langle \eta_N^D,f\rangle}\bigr)\,\underset{N\to\infty}\longrightarrow\,
E\Bigl(\exp\Bigl\{-\int Z^D(\textd x)\otimes\texte^{-\alpha h}\textd h\otimes\nu(\textd\phi)\bigl(1-\texte^{- f_\infty(x,h,\phi)}\bigr)\Bigr\}\Bigr),
\end{equation}
where
\begin{equation}
f_\infty(x,h,\phi):=\lim_{r\to\infty} f_r(x,h,\phi)=\sum_{z\in\Z^2}f(x,h-\phi_z).
\end{equation}
The tightness of~$\eta_N^D$-processes --- or the growth estimate on samples from~$\nu$ in~\eqref{E:5.42} --- ensure that $f_\infty$ is finite almost everywhere under the intensity measure. Since,
\begin{equation}
\int\nu(\textd\phi)\bigl(1-\texte^{- f_\infty(x,h,\phi)}\bigr) = 1-E_\nu\exp\Bigl\{-\sum_{z\in\Z^2}f(x,h-\phi_z)\Bigr\}
\end{equation}
we easily check that the right-hand side of \eqref{E:5.88a} is the Laplace transform of the cluster process in the statement and that this process is locally finite as claimed.
\end{proofsect}

\subsection{Control of auxiliary process}
\label{sec-5.4}\noindent
In order to prove Proposition~\ref{lemma-4.18}, we need additional lemmas. The first one shows that the field~$\Phi^{M,x}$ is small uniformly in~$x$ and its argument.

\begin{lemma}
\label{lemma-max-auxfield}
For $M=M(N,r)$ as above and any $\delta>0$,
\begin{equation}
\label{E:5.69a}
\lim_{r\to\infty}\limsup_{N\to\infty}\,
P\biggl(\,\,\max_{x\in D_N^\delta}\,\max_{z\in\Lambda_r(x)}|\Phi^{M,x}(z)|>\frac{(\log r)^2}{\sqrt{\log N}}\biggr)=0
\end{equation}
\end{lemma} 

\begin{proofsect}{Proof}
Fix~$\delta>0$ and assume that~$N$ and~$r$ are so large that $\Lambda_M(x)\subset D_N^{\delta/2}$ for each~$x\in D_N^\delta$. Let $H(x,y)$ abbreviate the harmonic measure $H^{\Lambda_M(0)\smallsetminus\{0\}}(x,y)$. Then
\begin{equation}
\Var\bigl(\Phi^{M,x}(x+z)\bigr) = \sum_{y,\tilde y\in\partial\Lambda_M(0)}H(z,y)H(z,\tilde y)G^{D_N}(x+y,x+\tilde y).
\end{equation}
Plugging the asymptotic $G^{D_N}(x+y,x+\tilde y)\le c\log(N/(1+|y-\tilde y|))$ (cf Lemmas~\ref{lemma-G-potential}--\ref{lemma-potential}) implied by the containment $\Lambda_M(x)\subset D_N^\delta$ and using the standard bound on the harmonic measure (cf Lemma~\ref{lemma-HM})
\begin{equation}
\label{E:H-bound}
H(z,y)\le c_1\frac{\log r}{M\log M},\qquad z\in\Lambda_r(0),\,y\in\partial \Lambda_M(0),
\end{equation}
where we assume, e.g., $r\le M/2$,
we thus get that, for some constant $c_2\in(0,\infty)$,
\begin{equation}
\label{E:5.88}
\max_{x\in D_N^\delta}\,\max_{z\in\Lambda_r(x)}|\Var\bigl(\Phi(x+z)\bigr)\le c_2\Bigl(\frac{\log r}{\log M}\Bigr)^2\log(N/M).
\end{equation}
Since~$N/M$ is of order~$r$, the union bound combined with exponential Chebyshev inequality readily yield the claim.
\end{proofsect}

The next lemma uses the above control to show that the nearly-maximal $M$-local maxima of~$h^{D_N}$ exhaust, with high probability, those of~$h^{D_N}-\Phi^{M,x}$.

\begin{lemma}
\label{lemma-5.13}
For any $\lambda>0$ and $M=M(N,r)$ as above,
\begin{equation}
\lim_{r\to\infty}\,\limsup_{N\to\infty}\, 
P\left(\exists x\in \Gamma_N^D(\lambda)\colon\begin{aligned}
&h^{D_N}-\Phi^{M,x}\le h^{D_N}(x)\text{\rm\ in }\Lambda_M(x)
\\
&h^{D_N}\not\le h^{D_N}(x)\text{\rm\ in }\Lambda_M(x)
\end{aligned}\right)=0.
\end{equation}
\end{lemma}

\begin{proofsect}{Proof}
Note that, on the event whose probability we are to estimate, there are points $x\in D_N$ and $y\in\Lambda_M(x)$ such that $h^{D_N}(x)\ge m_N-\lambda$ and $h^{D_N}(y)>h^{D_N}(x)\ge h^{D_N}(y)-\Phi^{M,x}(y)$. Unless these points obey $r<|x-y|\le M$, this~$y$ must in fact lie in~$\Lambda_r(x)$. Since also $h^{D_N}-\Phi^{M,x}\le h^{D_N}(x)$ is assumed in~$\Lambda_M(x)$, the field $h^{D_N}-\Phi^{M,x}$ then has an $r$-local maximum at~$x$ but with a gap to the next value in~$\Lambda_r(x)$ less than $\max_{z\in\Lambda_r(x)}|\Phi^{M,x}(z)|$. Utilizing our uniform bound on this maximum from Lemma~\ref{lemma-max-auxfield}, we will show this to be unlikely to happen anywhere in~$D_N$.

Let $\lambda>0$ be fixed, pick~$\delta>0$ small and assume that $N$ is so large that, for a given $r\ge1$, we have $\delta N\gg r$. Abbreviate $a_N:=\log\log N/\log N$. Using the above observations, we bound the probability in the statement by
\begin{equation}
\label{E:5.82}
\begin{aligned}
\quad
P\bigl(\Gamma_N^D(\lambda)&\smallsetminus D_N^\delta\ne\emptyset\bigr)+P\bigl(\,\max_{x\in D_N}h^{D_N}(x)>m_N+\lambda'\bigr)
\\
&+P\bigl(\exists x,y\in\Gamma_N(\lambda)\colon r<|x-y|\le M\bigr)
\\
&\qquad\quad+P\Bigl(\,\max_{x\in D_N^\delta}\max_{y\in\Lambda_r(x)}|\Phi^{M,x}(y)|>a_N\Bigr)
\\
&\qquad\qquad\qquad+\sum_{x\in D_N^\delta}P\left(\begin{aligned}
&-\lambda\le h^{D_N}(x)-m_N\le\lambda'\\
&\,\,h^{D_N}-\Phi^{M,x}\le h^{D_N}(x)\text{\rm\ in }\Lambda_M(x)
\\
&\max_{\begin{subarray}{c}
y\in\Lambda_r(x)\\ y\ne x
\end{subarray}}
\bigl[h^{D_N}(y)-\Phi^{M,x}(y)\bigr]>h^{D_N}(x)-a_N
\end{aligned}\right)\,,
\end{aligned}
\end{equation}
where $\lambda'>0$ and where we used the union bound in the last step. Invoking Lemmas~\ref{lem:6}, \ref{lemma-max-tail}, \ref{lemma-separation} and~\ref{lemma-max-auxfield}, the first four probabilities tend to zero as~$N\to\infty$,~$r\to\infty$ and $\lambda'\to\infty$. Thanks to Lemma~\ref{lemma-5.10q}, the last probability (without the sum) is equal to
\begin{equation}
\int_{m_N-\lambda}^{m_N+\lambda'} P\bigl(h^{D_N}(x)\in\textd t\bigr) P\Bigl(h^{\Delta^n}\le t\text{ in }\Delta^n,\, h^{\Delta^n}\not\le t-a_N \text{ in }\Delta^k\smallsetminus\{0\}\,\Big|\,h^{\Delta^n}(0)=t\Bigr),
\end{equation}
where $n\in\N$ is such that $M=2^n$ and~$k\in\N$ is such that $2^k\ge r>2^{k-1}$. Noting that $m_N-m_{2^n}$ remains bounded as~$N\to\infty$ and using Lemma~\ref{lemma-3.1} to convert the conditioning to $h^{\Delta^n}(0)=0$, the second part of Lemma~\ref{lemma-4.20} shows that the integral is bounded by
\begin{equation}
\label{E:5.84}
\frac{c_N}n P\bigl(h^{D_N}(x)\ge m_N-\lambda\bigr),
\end{equation}
where $c_N\to0$ as~$N\to\infty$. For $x\in D_N^\delta$ we have $\Var(h^{D_N}(x))\ge g\log N-c$ for some constant $c\in(0,\infty)$. Plugging this to the standard Gaussian asymptotic and using some straightforward manipulations (cf, e.g., \eqref{E:Gauss-calc}), the expression \eqref{E:5.84} is bounded by a constant times $c_N/N^2$ once~$N$ is sufficiently large. As $|D_N^\delta|$ is at most a constant times~$N^2$, the sum in \eqref{E:5.82} tends to zero as $N\to\infty$, thus proving the claim.
\end{proofsect}

The next lemma complements this by showing that the correspondence between the local maxima of $h^{D_N}$ and those of $h^{D_N}-\Phi^{M,x}$ is, in fact, one-to-one  with high probability.

\begin{lemma}
\label{lemma-5.14}
For any $\lambda>0$ and $M=M(N,r)$ as above,
\begin{equation}
\lim_{r\to\infty}\,\limsup_{N\to\infty}\, 
P\left(\exists x\in \Gamma_N^D(\lambda)\colon\begin{aligned}&h^{D_N}\le h^{D_N}(x)\text{\rm\ in }\Lambda_M(x)\\
&h^{D_N}-\Phi^{M,x}\not\le h^{D_N}(x)\text{\rm\ in }\Lambda_M(x)\end{aligned}\right)=0.
\end{equation}
\end{lemma}

\begin{proofsect}{Proof}
Given any~$\delta>0$ and~$\lambda'>0$, by Lemma~\ref{lem:6}, we may restrict the event to $x\in D_N^\delta$ and $h^{D_N}(x)\le m_N+\lambda'$ as soon as~$N$ is sufficiently large. We will deal separately with  two cases depending on whether  $h^{D_N}-\Phi^{M,x}\not\le h^{D_N}(x)$ occurs in~$\Lambda_{M/2}(x)$  or  in~$\Lambda_M(x)\smallsetminus\Lambda_{M/2}(x)$.

\smallskip\noindent
\textsl{CASE 1:} Suppose that $h^{D_N}-\Phi^{M,x}\not\le h^{D_N}(x)$ occurs in~$\Lambda_{M/2}(x)$ for some~$x\in D_N^\delta$ and note that, since $h^{D_N}\le h^{D_N}(x)$, this in particular forces $\Phi^{M,x}\not\ge0$ in~$\Lambda_{M/2}(x)$. Applying the union bound, the relevant probability is at most
\begin{equation}
\label{E:5.103}
\sum_{x\in D_N^\delta}P\Bigl(
h^{D_N}(x)-m_N\in[-\lambda,\lambda'],\, 
h^{D_N}\le h^{D_N}(x) \text{ in }\Lambda_M(x),\,\Phi^{M,x}\not\ge 0\text{ in }\Lambda_{M/2}(x)
\Bigr).
\end{equation}
We will bound the probability under the sum by conditioning on $h^{D_N}(x)=m_N + t$ (with~$t \in[-\lambda,\lambda'])$ and on $h^{D_N}\le m_N + t$ in~$\Lambda_M(x)$. 

Focusing on the conditional probability only, the strong-FKG property (Lemma~\ref{lemma-FKG}) and the fact that~$\Phi^{M,x}$ is an increasing function of $h^{D_N}$ then imply
\begin{multline}
\qquad
P\Bigl(\Phi^{M,x}\not\ge 0\text{ in }\Lambda_{M/2}(x)\,\Big|\, h^{D_N}\le m_N + t
\text{ in }\Lambda_M(x),\,h^{D_N}(x)=m_N + t\Bigr)
\\
\le
P\Bigl(\Phi^{M,x}\not\ge0\text{ in }\Lambda_{M/2}(x)\,\Big|\, h^{D_N}\le m_N + t
\text{ in }D_N,\,h^{D_N}(x)=m_N + t\Bigr).
\qquad
\end{multline}
We now assume for simplicity (and without loss of generality) that~$x=0$ and, abusing our earlier notation, interpret $D_N$ as a domain such that
\begin{equation}
\Lambda_M(0)=\Delta^n\subseteq D_N=: \Delta^{n+q},
\end{equation}
where $n\in\N$ is such that~$M=2^n$ and where $q=q(r)=r\log 2+O(1)$. In this notation, Lemma~\ref{lemma-3.1} shows that the conditional probability becomes
\begin{equation}
\label{E:5.107}
P\Bigl(\Phi^{M,0}+(m_N+t)\frakg^{D_N}\not\ge0\text{ in }\Delta^{n-1}\,\Big|\, h^{D_N}\le (m_N+t)(1-\frakg^{D_N}),\,\,h^{D_N}(x)=0\Bigr)
\end{equation}
where $\Phi^{M,0}$ admits the representation
\begin{equation}
\label{E:5.108}
\Phi^{M,0}(z) = \sum_{k=n+1}^{n+q}\bigl(\frakb_k(z)\varphi_k(0)+\chi_k(z)\bigr),\qquad z\in\Delta^n(0).
\end{equation}
 This  is checked by noting that the field on the right agrees with $h^{D_N}$ on~$\partial\Delta^n$, is harmonic in $\Delta^n\smallsetminus\{0\}$ and equal to zero at~$z=0$. In light of the harmonicity of $z\mapsto\Phi^{M,0}(z)+(m_N+t)\frakg^{D_N}(z)$ on $D_N\smallsetminus\{0\}$ and the fact that this function vanishes at $z:=0$, by the maximum principle it suffices to estimate the probability that this function takes a negative value on $\partial\Delta^{n-1}$. For that we note
\begin{equation}
m_N\frakg^{D_N}(z)=2\sqrt g\log r+O(1),\qquad z\in\partial\Delta^{n-1},
\end{equation}
and so we can bound the probability in \eqref{E:5.107} by
\begin{equation}
\label{E:5.99}
P\Bigl(\Phi^{M,0}\not\ge -2\sqrt g\log r-c\text{ on }\partial\Delta^{n-1}\,\Big|\, h^{D_N}\le (m_N+t)(1-\frakg^{D_N}),\,\,h^{D_N}(x)=0\Bigr)
\end{equation}
for some~$c=c(\lambda)\in(0,\infty)$. 

We are now ready to derive the desired bound. Comparing the event in \eqref{E:5.99} with \eqref{E:5.108} and assuming~$r$ to be large, on the event in question we have
\begin{equation}
\max\Bigl\{\max_{n\le k\le n+q}|\varphi_k(0)|,\,\max_{n\le k\le n+q}\max_{z\in\Delta^n}|\chi_k(z)|\Bigr\} > c'\log r
\end{equation}
for some~$c'\in(0,\infty)$ depending only on the constant in Lemma~\ref{lemma-3.6a}. As was used in the proof of Lemma~\ref{lemma-4.2u}, the conditioning on~$h^{D_N}(0)=0$ amounts to changing $\varphi_k(0)$ to~$\varphi_k(0)-c_n(k)S_{n+1}$ which have a uniform Gaussian tail under the conditioning. By Lemma~\ref{lemma-3.6a} a similar statement holds for the~$\chi_k$'s. Following the proof of Lemmas~\ref{lemma-4.7}, and invoking Lemma~\ref{lemma-4.9new} for the lower bound on the conditional event, the probability in \eqref{E:5.99} is thus at most $c_1\texte^{-c_2(\log r)^2}$ uniformly in~$t\in[-\lambda,\lambda']$ (and uniformly in all shifts of~$D_N$ for which~$0$ lies in~$D_N^\delta$). 

The quantity in \eqref{E:5.103} is therefore bounded~by
\begin{equation}
\label{E:5.101}
c_1\texte^{-c_2(\log r)^2}\sum_{x\in D_N^\delta}P\Bigl(x\in\Gamma_N^D(\lambda),\, h^{D_N}\le h^{D_N}(x)\text{ in }\Lambda_M(x)\Bigr).
\end{equation}
Writing the sum as an expectation of the sum of indicators, the fact that the field values are continuously distributed and~$M\ge N/(2r)$ shows that the sum (of indicators) is bounded pointwise (a.s.) by a constant times~$r^2$, uniformly in~$N$ (sufficiently large). The expression in \eqref{E:5.101} thus tends to zero as~$r\to\infty$, thus taking care of the case when $h^{D_N}-\Phi^{M,x}\not\le h^{D_N}(x)$ occurs in~$\Lambda_{M/2}(x)$.

\smallskip\noindent
\textsl{CASE 2:} Next let us assume that $h^{D_N}-\Phi^{M,x}\not\le h^{D_N}(x)$ occurs in~$A_M(x):=\Lambda_{M}(x)\smallsetminus\Lambda_{M/2}(x)$. By way of a union bound, we can fix that~$x$ and estimate the probability for that~$x$ only. Recall the notation~$\FF_{M,x}$ for the $\sigma$-algebra from Lemma~\ref{lemma-5.10q}. Noting that $\Phi^{M,x}$ and $h^{D_N}(x)$ are measurable with respect to~$\FF_{M,x}$, by conditional independence (cf Lemma~\ref{lemma-5.10q})
\begin{multline}
\label{E:5.102}
\qquad
P\Bigl(h^{D_N}\le h^{D_N}(x)\text{\rm\ in }\Lambda_M(x),\,h^{D_N}-\Phi^{M,x}\not\le h^{D_N}(x)\text{\rm\ in }A_M(x)\,\Big|\,\FF_{M,x}\Bigr)
\\
= P\Bigl(h^{D_N}\le h^{D_N}(x)\text{\rm\ in }\Lambda_M(x)\,\Big|\,\FF_{M,x}\Bigr)
\\
\times 
P\Bigl(h^{D_N}-\Phi^{M,x}\not\le h^{D_N}(x)\text{\rm\ in }A_M(x)\,\Big|\,\FF_{M,x}\Bigr).
\qquad
\end{multline}
Letting $n\in\N$ be such that $M=2^n$, Lemma~\ref{lemma-5.10q} along with the fact that $(m_N-\lambda)(1-\frakg^{\Delta^n}) - m_M \ge 2\sqrt g\log r-c$ on~$\Delta^n\smallsetminus\Delta^{n-1}$ then shows that on $\{h^{D_N}(x)\ge m_N-\lambda\}$ we have the pointwise (a.s.) inequality
\begin{multline}
\qquad
P\Bigl(h^{D_N}-\Phi^{M,x}\not\le h^{D_N}(x)\text{\rm\ in }A_M(x)\,\Big|\,\FF_{M,x}\Bigr)
\\
= P\Bigl(h^{\Delta^n\smallsetminus\{0\}}-(m_N+t)(1-\frakg^{\Delta^n})\not\le0\text{ in }\Delta^n\smallsetminus\Delta^{n-1}\Bigr)\Bigr|_{t:=h^{D_N}(x)}
\\
\le P\Bigl(\,\,\max_{z\in\Delta^n\smallsetminus\Delta^{n-1}} h^{\Delta^n\smallsetminus\{0\}}(z) > m_M + 2\sqrt g\log r-c\Bigr).
\qquad
\end{multline}
Using Lemma~\ref{lemma-stoch-order} to replace $h^{\Delta^n\smallsetminus\{0\}}$ by $h^{\Delta^n}$, the sharp upper tail of the maximum in Lemma~\ref{lem:6} bounds this probability by a constant times~$r^{-4}\log r$. Using this in \eqref{E:5.102}, the probability in the case under consideration is then bounded by a constant times
\begin{equation}
\frac{\log r}{r^4}\sum_{x\in D_N^\delta}P\bigl(h^{D_N}\le h^{D_N}(x)\text{\rm\ in }\Lambda_M(x)\bigr).
\end{equation}
As argued before, the sum is at most a constant times~$r^2$ uniformly in~$N$ and so the claim follows by taking~$r\to\infty$.
\end{proofsect}

With the above lemmas in hand, the proof of the last missing step in the proof of our main theorems is quite immediate:

\begin{proofsect}{Proof of Proposition~\ref{lemma-4.18}}
We will write $M$ in place of $M(N,r)$. Suppose that~$f(x,h,\phi)$ is zero unless~$|h|\le\lambda$ and assume~$f$ only depends on~$\phi(z)$ for~$z\in\Lambda_{r_0}(0)$ for some $r_0\ge1$. Introduce an intermediate (auxiliary) process
\begin{equation}
\widetilde\eta^D_{N,M}:=\sum_{x\in D_N}
\1_{\{h_x^{D_N}=\max_{z\in \Lambda_M(x)}h^{D_N}_x\}}
\delta_{\,x/N}\otimes\delta_{\,h^{D_N}_x-m_N}\otimes\delta_{\,\{h^{D_N}_x-h^{D_N}_{x+z}+\Phi^{M,x}_{x+z}\colon z\in\Z^2\}}.
\end{equation}
Lemmas~\ref{lemma-5.13}--\ref{lemma-5.14} show that the $M$-local extrema of $h^{D_N}$ and $h^{D_N}-\Phi^{M,\cdot}$ are in one-to-one correspondence with probability tending to one in the limit as $N\to\infty$ and~$r\to\infty$. It follows that, for any~$f$ as above,
\begin{equation}
\lim_{r\to\infty}\lim_{N\to\infty}P\bigl(\langle\widetilde\eta^D_{N,M},f\,\rangle\ne\langle\widehat\eta^D_{N,M},f\,\rangle\bigr)=0.
\end{equation}
But the assumptions on continuity and support of~$f$ imply that for each~$\epsilon>0$ there is~$\delta>0$ such that for all~$x$ and~$h$,
\begin{equation}
\max_{x\in\Lambda_r(0)}\bigl|\phi(x)-\phi'(x)\bigr|<\delta\quad\Rightarrow\quad \bigl|f(x,h,\phi)-f(x,h,\phi')\bigr|<\epsilon.
\end{equation}
In light of Lemma~\ref{lemma-max-auxfield}, for each~$\epsilon>0$ we thus have
\begin{equation}
\lim_{N\to\infty}P\Bigl(\bigl|\langle\widetilde\eta^D_{N,M},f\,\rangle-\langle\eta^D_{N,M},f\,\rangle\bigr|>\epsilon\Bigr)=0
\end{equation}
as soon as $r>r_0$. The claim follows.
\end{proofsect}

\section{Local limit theorem and freezing}
\label{sec6}\noindent
In this section we complete the proofs of the remaining results: the local limit theorem for the absolute maximum (Theorem~\ref{thm-LLL}) and then the results on the Liouville measure in glassy phase and freezing (Theorem~\ref{thm-2.7} and Corollaries~\ref{cor-2.7}--\ref{cor-2.8}).

\subsection{Local limit theorem for the absolute maximum}
We begin with the local limit theorem for both position and value of the global maximum. The starting point is a reformulation of Proposition~\ref{prop-4.13} for general outer domains. For any integer $q\in\Z$ define the square domain
\begin{equation}
S^q:=(-2^{q},2^q)\times (-2^{q},2^q).
\end{equation}
Thanks to translation invariance of the DGFF and scaling, we can and will restrict attention to continuum domains in the class
\begin{equation} 
\mathfrak D_q:=\bigl\{D\in\mathfrak D\colon \overline{S^{1}}\subseteq D\subseteq S^q\bigr\}
\end{equation}
where~$q>1$ is an arbitrary (but fixed) integer.
The definition ensures the following property: If~$\{D_N\}$ is a sequence of approximating domains satisfying~\twoeqref{E:1.1}{E:1.1a} and
\begin{equation}
\label{E:6.3q}
n:=\min\{m\in\N\colon 2^m\ge N\},
\end{equation}
then for all $N$ sufficiently large,
\begin{equation}
\label{E:6.4q}
\Delta^{n-\ell-1}\subseteq
S_N^{-\ell}\subseteq\Delta^{n-\ell}\subseteq\Delta^n\subseteq D_N\subseteq\Delta^{n+q},
\end{equation}
where we set $S^{-\ell}_N:=\{x\in\Z^2\colon x/N\in S^{-\ell}\}$. The reason for inserting $S^{-\ell}_N$ in-between $\Delta^{n-\ell-1}$ and~$\Delta^{n-\ell}$ is that, unlike these two domains (assuming~$n$ is tied to~$N$ as in \eqref{E:6.3q}), $S^{-\ell}_N$ has a well defined scaling limit as~$N\to\infty$. 

Pick $\ell\ge1$ and given~$D\in\mathfrak D_q$, consider a sequence $\{D_N\}$ of approximating domains satisfying \twoeqref{E:1.1}{E:1.1a} and \eqref{E:6.4q}. For $h^{D_N}$ the DGFF in~$D_N$, set
\begin{equation}
\Psi_{N,\ell}(x):=E\bigl(h^{D_N}(x)\,\big|\,\sigma(h^{D_N}(z)\colon z\in \partial S^{-\ell}_N)\bigr).
\end{equation}
For each $t\in\R$ and for~$m_N$ as in \eqref{E:mN}, we then define an analogue of $\Xiout_{n,\ell}(t)$ from \eqref{E:5.2ua} as
\begin{equation}
\label{E:5.2ua-gen}
\Xi^{D}_{N,\ell}(t):=
E\biggl(\Psi_{N,\ell}(0)\,\1_{\{\Psi_{N,\ell}(0)\in[\ell^{1/6},\ell^2]\}}\prod_{x\in D_N\smallsetminus S^{-\ell+1}_N}\1_{\{h^{D_N}(x)\le\frakm_N(x,t)\}}\,\bigg|\,h^{D_N}(0)=0\biggr),
\end{equation}
where, abusing our earlier notation slightly,
\begin{equation}
\frakm_N(x,t):= (m_N+t)(1-\frakg^{D_N}(x))
\end{equation}
with $\frakg^D$ as defined in Section~\ref{sec-3.1}. Recall that $\Xiin_\ell(1)$ denotes the quantity from \eqref{E:5.1ua} for~$f:=1$ and that $\Xiin_\infty(1)=\lim_{\ell\to\infty}\Xiin_\ell(1)$ exists. Our rewrite of Proposition~\ref{prop-4.13} is as follows:

\begin{proposition}
\label{prop-4.13-gen}
For each~$q>1$, each $\epsilon>0$ and each $t_0>0$ there is $\ell_0\ge1$ such that for all $t\in\R$ with $|t|<t_0$, all $\ell\ge\ell_0$ and all $D\in\mathfrak D_q$,
\begin{equation}
\label{E:4.30z-gen}
\biggl|\,P\Bigl(h^{D_N}\le m_N+t\,\Big|\,h^{D_N}(0)=m_N+t\Bigr)
-\frac 2{g\log N}\,\Xiin_\infty(1)\,\Xi^D_{N,\ell}(t)\biggr|\le \frac \epsilon {\log N}
\end{equation}
holds true for any sequence $\{D_N\}$ corresponding to~$D$ via \twoeqref{E:1.1}{E:1.1a} as soon as $N$ is so large that $\ell\le (\log N)^{1/8}$ and \eqref{E:6.4q} apply.
\end{proposition}

\begin{proofsect}{Proof}
Applying the observations from Lemma~\ref{lemma-gen-domain}, the proof of Proposition~\ref{prop-4.13} carries over to this case as soon as~$N$ and~$\ell$ are such that $\ell\le (\log N)^{1/8}$ and \eqref{E:6.4q} hold. Recalling~\eqref{E:6.3q} and denoting, for $a>0$,
\begin{equation}
\widetilde \Xi^{D}_{N,\ell}(t,a) := E\biggl(S_{n-\ell}\,\1_{\{S_{n-\ell}\in[a^{-1}\ell^{1/6},\,a\ell^2]\}}\prod_{x\in D_N\smallsetminus \Delta^{n-\ell}}\1_{\{h^{D_N}(x)\le\frakm_N(x,t)\}}\,\bigg|\,S_{n+1}=0\biggr)
\end{equation}
where the random walk $\{S_k\}$ is related to $h^{D_N}$ via $S_{n+1}=h^{D_N}(0)$ and
\begin{equation}
S_{k}=E\bigl(h^{D_N}(0)\,\big|\,\sigma(h^{D_N}(z)\colon z\in \partial \Delta^{k-1})\bigr),
\qquad k=1,\dots,n,
\end{equation}
we thus get
\begin{multline}
\qquad
\frac2{g\log 2}\,\frac1n\,\Xiin_{\ell}(1)\widetilde \Xi^{D}_{N,\ell}(t,\ffrac12)-\frac\epsilon n
\\
\le
P\Bigl(h^{D_N}\le m_N+t\,\Big|\,h^{D_N}(0)=m_N+t\Bigr)
\\
\le \frac2{g\log 2}\,\frac1n\,\Xiin_{\ell+1}(1)\widetilde \Xi^{D}_{N,\ell+1}(t,2)+\frac\epsilon n.
\qquad
\end{multline}
Now the inclusions $\Delta^{n-\ell-1}\subseteq S^{-\ell}_N\subseteq\Delta^{n-\ell}$ from \eqref{E:6.4q} and the Gibbs-Markov property (Lemma~\ref{lemma-GM}) ensure that $S_k-\Psi_{N,\ell}(0)$ is a centered Gaussian with a uniformly bounded variance. A straightforward estimate (of the kind done in Lemma~\ref{lemma-4.2u}) then shows that the contribution to $\widetilde \Xi^{D}_{N,\ell}(t,a)$ of the event when this random variable is larger than $\frac12\ell^{1/6}$ is at most $\epsilon$ as soon as~$\ell$ is large enough (we use that $S_{n-\ell}$ is bounded by $2\ell^2$ when $a\le 2$). This yields
\begin{equation}
\widetilde \Xi^{D}_{N,\ell+1}(t,2)-\epsilon
\le\Xi^{D}_{N,\ell}(t)
\le \widetilde \Xi^{D}_{N,\ell}(t,\ffrac12)+\epsilon.
\end{equation}
Invoking $n\log2 = (1+o(1))\log N$, applying the corresponding analogue of \twoeqref{E:4.32z}{E:4.33z} and using that~$\Xiin_\ell(1)$ can be made as close to~$\Xiin_\infty(1)$ by taking~$\ell$ large enough, the claim follows (relabeling~$\epsilon$ to a constant times~$\epsilon$ to absorb numerical prefactors).
\end{proofsect}

We record one consequence of the proof:

\begin{corollary}
For each $q>1$ and each~$t_0>0$ there are~$c_1,c_2\in(0,\infty)$ and $\ell_0\ge1$ such that for all $\ell\ge\ell_0$, all $D\in\mathfrak D_q$ and any sequence $\{D_N\}$ of lattice domains such that \twoeqref{E:1.1}{E:1.1a} apply, 
\begin{equation}
c_1<\Xi^D_{N,\ell}(t)<c_2
\end{equation}
holds true uniformly in~$t\in[-t_0,t_0]$ as soon as~$N$ is sufficiently large. 
\end{corollary}

\begin{proofsect}{Proof}
This follows from from Proposition~\ref{prop-4.13-gen}, the fact that $\Xiin_\infty(1)\in(0,\infty)$ and the bounds in Lemma~\ref{lemma-4.9new} adapted to the present situation.
\end{proofsect}

The behavior of $N\mapsto \Xi^D_{N,\ell}(t)$ as $N\to\infty$ was not important for the construction of the cluster law as this quantity factors out from all relevant formulas. This is different for the local limit theorem for the maximum, where we will need the limit of $\Xi^D_{N,\ell}(t)$ as $N\to\infty$ to exist and even have some regularity properties. To describe the limit object, let~$\widehat\Psi_\ell$ denote the mean-zero Gaussian process on $D\smallsetminus \partial S^{-\ell}$ with covariance
\begin{equation}
C(x,y):=G^D(x,y)-G^{S^{-\ell}}(x,y),
\end{equation}
where $G^D$ is the continuum Green function in~$D$ with Dirichlet boundary conditions on~$\partial D$; see Section~\ref{sec-GF}. As it turns out, the field $\widehat\Psi_\ell$ is the scaling limit of $\Psi_{N,\ell}$ defined above: 

\begin{lemma}
\label{lemma-6.3ua}
Fix~$\ell\ge1$, $q>1$ and let $D\in\mathfrak D_q$. The field $\widehat\Psi_\ell$ has continuous sample paths on $D\smallsetminus\partial S^{-\ell}$ a.s. Moreover, for each $N\ge1$ there is a coupling of $\Psi_{N,\ell}$ and~$\widehat\Psi_\ell$ such that, for each~$\delta>0$ sufficiently small,
\begin{equation}
\label{E:6.17a}
\max_{\begin{subarray}{c}
x\in D_N\\\dist(x,\partial D_N\cup \partial S^{-\ell}_N)>\delta N
\end{subarray}}
\bigl|\,\Psi_{N,\ell}(x)-\widehat\Psi_\ell(x/N)\bigr|\,\underset{N\to\infty}\longrightarrow\,0,\qquad\text{\rm in probability}.
\end{equation}
\end{lemma}

\begin{proofsect}{Proof (sketch)}
As $\Psi_{N,\ell}$, resp., $\widehat\Psi_\ell$ is the discrete and continuum ``binding'' field relating the GFF in~$D$ to that in~$D\smallsetminus\partial S^{-\ell}$, this reduces to Lemma~\ref{lemma-BF-converge}.\end{proofsect}

The claim about $\Xi^D_{N,\ell}(t)$ we will need is then as follows:

\begin{proposition}
\label{prop-6.2ua}
Let $q>1$ be an integer. For each $\ell\ge1$, each $t\in\R$, each $D\in\mathfrak D_q$ and each sequence $\{D_N\}$ of domains related to~$D$ as in \twoeqref{E:1.1}{E:1.1a}, the limit
\begin{equation}
\Xi^D_{\infty,\ell}(t):=\lim_{N\to\infty}\Xi^D_{N,\ell}(t)
\end{equation}
exists and is finite, strictly positive, non-decreasing and continuous in~$t$. Moreover, the limit is independent of the sequence $\{D_N\}$ and, in fact, admits the explicit representation
\begin{equation}
\label{E:6.15ua}
\Xi^D_{\infty,\ell}(t)=E\Bigl(\widehat\Psi_{\ell}(0)\,Q^D_{\ell,t}(\widehat\Psi_\ell)\1_{\{\widehat\Psi_{\ell}(0)\in[\ell^{1/6},\ell^2]\}}
\Bigr),
\end{equation}
where $Q^D_{\ell,t}(\varphi)$ is, for each Borel measurable $\varphi\colon\overline D\smallsetminus S^{-\ell+1}\to\R$, given by
\begin{equation}
\label{E:6.16}
Q^D_{\ell,t}(\varphi) :=E\biggl(\exp\Bigl\{-\alpha^{-1}\texte^{-\alpha t}
\int_{D\smallsetminus S^{-\ell+1}}\, Z^{D\smallsetminus S^{-\ell}}(\textd x)\texte^{\alpha\varphi(x)+\alpha^2 G^D(0,x)}\Bigr\}\biggr)
\end{equation}
Here $Z^D$ is the measure from Theorem~\ref{thm-main}.
\end{proposition}

\begin{proofsect}{Proof}
Fix~$\ell\ge1$ and define, for~$\alpha:=2/\sqrt g$ and $G^D$ the continuum Green function in~$D$,
\begin{equation}
\label{E:6.18}
\widehat \Xi^D_{N,\ell}(t):=E\biggl(\widehat\Psi_{\ell}(0)\,\1_{\{\widehat\Psi_{\ell}(0)\in[\ell^{1/6},\ell^2]\}}\prod_{x\in D_N\smallsetminus S^{-\ell+1}_N}\1_{\{h^{D_N\smallsetminus S^{-\ell}_N}(x)+\widehat\Psi_\ell(x/N)\le m_N+t-\alpha G^D(0,x/N)\}}\biggr)\,,
\end{equation}
where we regard $h^{D_N\smallsetminus S^{-\ell}_N}$ and~$\widehat\Psi_\ell$ as independent. Abusing our earlier notation, let $\widetilde \Xi^D_{N,\ell}(t)$ be the same quantity but with all occurrences of~$\widehat\Psi_\ell(\cdot/N)$ replaced by $\Psi_{N,\ell}(\cdot)$. Thanks to the Gibbs-Markov property (Lemma~\ref{lemma-GM}), $\widetilde \Xi^D_{N,\ell}(t)$ is the quantity defined as in \eqref{E:5.2ua-gen} but without the conditioning on $h^{D_N}(0)=0$ and with $\frakm_N(x,t)$ replaced by $m_N+t-\alpha G^D(0,x/N)$. 

We claim that the three objects $\Xi^D_{N,\ell}(t)$, $\widetilde \Xi^D_{N,\ell}(t)$ and $\widehat \Xi^D_{N,\ell}(t)$ are equal in the limit as $N\to\infty$. Indeed, to see the closeness of the former two note that
\begin{equation}
\label{E:6.17b}
\lim_{N\to\infty}\,\,
\max_{x\in D_N\smallsetminus S^{-\ell}_N}\,\Bigl|\,\frakm_N(x,t)-\bigl(m_N+t-\alpha G^D(0,x/N)\bigr)\Bigr|=0
\end{equation}
uniformly on compact sets of~$t$. In addition, observe that
\begin{equation}
\bigl(h^{D_N}(\cdot)\big|h^{D_N}(0)=0\bigr)\laweq h^{D_N}(\cdot)-\frakg^{D_N}(\cdot)h^{D_N}(0)
\end{equation}
and
\begin{equation}
\max_{x\in D_N\smallsetminus S^{-\ell}_N}\frakg^{D_N}(x)\le c\frac{\ell}{\log N}
\end{equation}
for some~$c\in(0,\infty)$. Let~$\epsilon_N$ be a sequence with $\epsilon_N\downarrow0$ such that $\epsilon_N$ is larger than both $c\frac{\ell}{\log N}$ and the maximum in \eqref{E:6.17b} and such that the probability that $h^{D_N}(0)\in[-(\log N)^{1/3},(\log N)^{1/3}]$ or that $\Psi_{N,\ell}(0)$ is within~$\epsilon_N$ of $\ell^{1/3}$ or~$\ell^2$ is at most~$1/N$. Then
\begin{equation}
\label{E:6.19}
(1-\epsilon_N)
\widetilde \Xi^D_{N,\ell}\bigl(t-2\epsilon_N\bigr)-2\ell^2/N
\le
\Xi^D_{N,\ell}(t)
\le
(1+\epsilon_N)\widetilde \Xi^D_{N,\ell}\bigl(t+2\epsilon_N\bigr)+2\ell^2/N.
\end{equation}
and so $\Xi^D_{N,\ell}(t)-\widetilde \Xi^D_{N,\ell}(t)$ indeed converges to zero as~$N\to\infty$.

Moving to the corresponding relation with $\widehat\Xi^D_{N,\ell}(t)$, consider the coupling of $\Psi_{N,\ell}$ and~$\widehat\Psi_\ell$ guaranteed by Lemma~\ref{lemma-6.3ua} for~$\delta$ such that $\dist(0,\partial S_N^{-\ell})>\delta N$ as well as $\dist(\partial S_N^{-\ell+1},S_N^{-\ell})>\delta N$. (This ensures that \eqref{E:6.17a} applies to all occurrences of~$\Psi_{N,\ell}$ in \eqref{E:6.18}.) Let~$\tilde\epsilon_N$ be a sequence $\tilde\epsilon_N\downarrow0$ such that with probability at least $1-1/N$, the maximum in \eqref{E:6.17a} is at most~$\tilde\epsilon_N$ and $\widehat\Psi_\ell(0)$ lies further than $\tilde\epsilon_N$ of the endpoints of the interval $[\ell^{1/6},\ell^2]$. Then
\begin{equation}
\label{E:6.20}
(1-\tilde\epsilon_N)\widehat \Xi^D_{N,\ell}\bigl(t-2\tilde\epsilon_N\bigr)-2\ell^2/N
\le
\widetilde\Xi^D_{N,\ell}(t)
\le
(1+\tilde\epsilon_N)\widehat \Xi^D_{N,\ell}\bigl(t+2\tilde\epsilon_N\bigr)+2\ell^2/N
\end{equation}
and so $\widehat\Xi^D_{N,\ell}(t)-\widetilde \Xi^D_{N,\ell}(t)$ also converges to zero as~$N\to\infty$
Combining \twoeqref{E:6.19}{E:6.20}, it suffices to prove the claim for $\widehat\Xi^D_{N,\ell}$ instead of~$\Xi^D_{N,\ell}$.

First we note that the product in \eqref{E:6.18} is the $a\to\infty$ limit of $\texte^{-a\langle\eta^D_N,f_{\widehat\Psi}\rangle}$ where
\begin{equation}
f_{\widehat\Psi}(x,h):=\1_{[\,t-\widehat\Psi_\ell(x)-\alpha G^D(0,x),\,\infty)}(h)\1_{D\smallsetminus S^{-\ell+1}}(x).
\end{equation}
Given a sample path of~$\widehat\Psi_\ell(x)$, this function can in turn be approximated by continuous functions with compact support in $\overline D\times \R$. The full (unstructured) process convergence in Corollary~\ref{cor-cluster-process} and some routine use of the Monotone Convergence Theorem then show that
\begin{equation}
\label{E:6.25}
\lim_{N\to\infty}\widehat \Xi^D_{N,\ell}(t) = E\Bigl(\widehat\Psi_{\ell}(0)\,Q^D_{\ell, t}(\widehat\Psi_\ell)\,\1_{\{\widehat\Psi_{\ell}(0)\in[\ell^{1/6},\ell^2]\}}
\Bigr),
\end{equation}
holds for all~$t$, where $\eta^D$ is the limit point process on the right-hand side of \eqref{E:1.6a} and $Q^D_{\ell,t}(\varphi)$ is the probability
\begin{equation}
Q_{\ell, t}^D(\varphi) := P\biggl(\eta^D\Bigl(\bigl\{(x,h)\in(\overline D\smallsetminus S^{-\ell+1})\times\R\colon h+\varphi(x)+\alpha G^D(0,x)> t\bigr\}\Bigr)=0\biggr) \,,
\end{equation}
and the right-hand side is continuous in $t$.

Thanks to Theorem~\ref{thm-main}, this probability admits the explicit representation~\eqref{E:6.16}. The function~$t\mapsto Q_{\ell,t}^D(\varphi)$, being essentially a Laplace transform of a non-negative and finite random variable, is automatically continuous, non-vanishing and finite for all~$t\in\R$ and all~$\varphi$ as above. In particular, \eqref{E:6.25} applies for all~$t$ and, by monotonicity, the convergence is locally uniform.
\end{proofsect}

Using a similar argument as in Proposition~\ref{cor-4.16}, we get:

\begin{corollary}
\label{cor-6.5}
For all~$t\in\R$, all integers~$q>1$ and all $D\in\mathfrak D_q$, the limit
\begin{equation}
\label{E:6.27ua}
\Xi^D_{\infty,\infty}(t):=\lim_{\ell\to\infty}\Xi^D_{\infty,\ell}(t)
\end{equation}
exists and is positive, finite, non-decreasing and continuous in~$t$. Moreover, for any sequence $\{D_N\}$ of domains related to~$D$ via \twoeqref{E:1.1}{E:1.1a}, 
\begin{equation}
\label{E:lim-S}
\lim_{N\to\infty}
(\log N)\,\,P\Bigl(h^{D_N}\le m_N+t\,\Big|\,h^{D_N}(0)=m_N+t\Bigr)
=(2/g)\,\Xiin_\infty(1)\,\Xi^D_{\infty,\infty}(t).
\end{equation}
The sequence on the left is bounded uniformly in~$N$.
\end{corollary}

\begin{proofsect}{Proof}
By Propositions~\ref{prop-4.13-gen} and~\ref{prop-6.2ua}, for each~$\epsilon>0$ and each~$t_0>0$, the sequence under the limit on the left of \eqref{E:lim-S} is within~$\epsilon$ of $(2/g)\Xiin_\infty(1)\,\Xi^D_{\infty,\ell}(t)$ uniformly in $t\in[-t_0,t_0]$ as soon as~$N\gg\ell\gg1$. As the former of these two sequences does not depend on~$\ell$, the limit in \eqref{E:6.27ua} exists and \eqref{E:lim-S} holds. Since $t\mapsto \Xi^D_{\infty,\ell}(t)$ is continuous, $t\mapsto\Xi^D_{\infty,\infty}(t)$ is continuous as well. 
\end{proofsect}

In light of the scaling relations for the $Z^D$-measure (cf Corollary~2.2 of~\cite{BL2}), the limit \eqref{E:6.27ua} exists and defines $\Xi^{D}_{\infty,\infty}(t)$ for any~$D\in\mathfrak D$ with~$0\in D$. For arbitrary~$D\in\mathfrak D$ we then set
\begin{equation}
\Xi^{D}_{\infty,\infty}(t,x):=\Xi^{-x+D}_{\infty,\infty}(t),\qquad x\in D.
\end{equation}
Our last item of concern is the regularity of~$x\mapsto\Xi^D_{\infty,\infty}(t,x)$:

\begin{lemma}
Let~$D\in\mathfrak D$ with~$0\in D$. Then $x\mapsto\Xi^D_{\infty,\infty}(t,x)$ is continuous on~$D$ for each~$t\in\R$. 
\end{lemma}

\begin{proofsect}{Proof}
The argument from the proof of Corollary~\ref{cor-6.5} shows that $\Xi^D_{\infty,\ell}(t)$ approximates~$\Xi^D_{\infty,\infty}(t)$ uniformly in~$D\in\mathfrak D_q$. It thus suffices to show that $x\mapsto\Xi^{-x+D}_{\infty,\ell}(t)$ is continuous at~$x:=0$ for each~$\ell$ and each~$D\in\mathfrak D_q$. For this observe that $x\mapsto G^{-x+D}(0,\cdot)$ varies continuously on $D\smallsetminus S^{-\ell+1}$ while both $x\mapsto Z^{-x+D}$ and $x\mapsto\widehat\Psi_\ell^{-x+D}$, where $\widehat\Psi_\ell^D$ marks the explicit dependence of the above field~$\widehat\Psi_\ell$ on the underlying domain, are continuous in law for~$x$ small. The continuity of $x\mapsto\Xi^{-x+D}_{\infty,\ell}(t)$ is then checked from \twoeqref{E:6.15ua}{E:6.16}. 
\end{proofsect}

\begin{proofsect}{Proof of Theorem~\ref{thm-LLL}}
Let~$D\in\mathfrak D$ and suppose without loss of generality that~$0\in D$. We will prove the claim with~$\rho^D(t,x)$ given by
\begin{equation}
\label{E:6.30a}
\rho^D(x,t):=c\texte^{-\alpha t}\exp\Bigl\{2\int_{\partial D}\Pi^D(x,\textd z)\log|x-z|\Bigr\}
\Xiin_\infty(1)\,\Xi^{D}_{\infty,\infty}(t,x),
\end{equation}
where~$c$ is a constant to be determined, $\Pi^D(x,\cdot)$ is the harmonic measure on~$\partial D$ for the Brownian motion started from~$x$.

Let~$\{D_N\}$ be a sequence of domains related to~$D$ via \twoeqref{E:1.1}{E:1.1a} and let~$x_N:=\lfloor xN\rfloor$. The probability density of~$h^{D_N}(x_N)$ evaluated at~$m_N+t$ is then
\begin{equation}
f_N(x,t):= \frac1{\sqrt{2\pi\Var(h^{D_N}(x_N))}}\texte^{-\frac12\frac{(m_N+t)^2}{\Var(h^{D_N}(x_N))}}.
\end{equation}
The standard representation of the discrete Green function using the potential yields
\begin{equation}
\Var(h^{D_N}(x_N)) = g\log N + g\int_{\partial D}\Pi^D(x,\textd z)\log|x-z|+ c_0 +O(N^{-2}).
\end{equation}
After some straightforward manipulations, this shows
\begin{multline}
\label{E:Gauss-calc}
\qquad
\frac12\frac{(m_N+t)^2}{\Var(h^{D_N}(x_N))}
=2\log N-\frac32\log\log N
\\+\alpha t
-2\int_{\partial D}\Pi^D(x,\textd z)\log|x-z| - 2c_0/g +o(1),
\qquad
\end{multline}
where $o(1)\to0$ uniformly on compact sets of~$t$ and on compact sets of~$x\in D$. Hence we get
\begin{equation}
\label{E:6.34ua}
f_N(x,t) = \frac{\texte^{2c_0/g+o(1)}}{2}\texte^{-\alpha t}\exp\Bigl\{2\int_{\partial D}\Pi^D(x,\textd z)\log|x-z|\Bigr\}\,\,\frac{\log N}{N^2}
\end{equation}
and, using also \eqref{E:lim-S} and the translation invariance of the DGFF,
\begin{equation}
\label{E:6.35}
\lim_{N\to\infty}
\,N^2\,P\Bigl(h^{D_N}\le m_N+t\,\Big|\,h^{D_N}(x_N)=m_N+t\Bigr)
f_N(x,t)
=\rho^D(x,t)
\end{equation}
provided we set~$c:=\texte^{2c_0/g}/g$ in \eqref{E:6.30a}. Since
\begin{multline}
\label{E:6.37}
\qquad
P\bigl(h^{D_N}\le h^{D_N}(z),\,h^{D_N}(z)-m_N\in(a,b)\bigr)
\\
=\int_a^b
P\Bigl(h^{D_N}\le m_N+t\,\Big|\,h^{D_N}(z)=m_N+t\Bigr)f_N(z/N,t)\,\textd t
\qquad
\end{multline}
and since the integrands on the right are bounded uniformly on compact sets of~$t$, the Bounded Convergence Theorem then proves \eqref{E:2.11c}. 

It remains to connect~$\rho^D(x,t)$ to the measure \eqref{E:2.12c}. Pick~$A\subset D$ open with $\overline A\subset D$ and recall that by Corollary~1.2 of Biskup and Louidor~\cite{BL1} generalized, with the help of Theorem~\ref{thm-main}, to arbitrary domains in~$\mathfrak D$,
\begin{multline}
\label{E:6.41}
\qquad\quad
P\Bigl(N^{-1}\argmax_{D_N}h^{D_N}\in A,\,\max_{z\in D_N}h^{D_N}(z)-m_N\in(a,b)\Bigr)
\\
\,\underset{N\to\infty}\longrightarrow\, 
\int_a^b\,E\bigl(\widehat Z(A)\texte^{-\alpha^{-1}Z^D(D)\texte^{-\alpha t}}\bigr)\textd t.
\quad\qquad
\end{multline}
Writing the left-hand side as
\begin{equation}
\int_{A\times(a,b)} N^2
P\Bigl(h^{D_N}\le m_N+t\,\Big|\,h^{D_N}(x_N)=m_N+t\Bigr)\,f_N(x,t)\,\textd x\,\textd t
\end{equation}
and recalling that, by Corollary~\ref{cor-6.5} and \eqref{E:6.34ua}, the integrand is bounded uniformly in~$N$, \eqref{E:6.35}, the Bounded Convergence Theorem, the fact that $\rho^D$ is measurable and \eqref{E:6.41} then show
\begin{equation}
\int_{A\times(a,b)}\rho^D(x,t)\textd x\textd t
=\int_a^b\,E\bigl(\widehat Z(A)\texte^{-\alpha^{-1}Z^D(D)\texte^{-\alpha t}}\bigr)\textd t.
\end{equation}
As this holds for a generating class of sets~$A$, the continuity of~$\rho^D$ yields the desired claim.
\end{proofsect}

\begin{remark}
\label{remark-6.7}
Recall that, for each~$D\in\mathfrak D$, the function $\psi^D$ takes the form
\begin{equation}
\psi^D(x)=c_\star \exp\Bigl\{2\int_{\partial D}\Pi^D(x,\textd z)\log|x-z|\Bigr\}
\end{equation}
for some (exis\-ten\-tial) constant~$c_\star\in(0,\infty)$ (same as that in \eqref{E:2.13}); see Biskup and Louidor~\cite{BL2}. Comparing \eqref{E:6.30a} with \eqref{E:2.13ua} and recalling the notation $c:=\texte^{2c_0/g}/g$ from the above proof, it thus follows that
\begin{equation}
c_\star = c \,\Xiin_\infty(1)\,\lim_{t\to\infty}\frac{\Xi^{D}_{\infty,\infty}(t,x)}t,
\end{equation}
where, in particular, the limit exists and is independent of~$x$. The limit depends only on the global characteristics of the extreme value statistics (as expressed by the~$Z^D$ measure); all local properties are encoded into~$\Xiin_\infty(1)$ and, to some extent, also the constant~$c$ (which depends on the potential~$\fraka$ via the constant~$c_0$). \textit{Update in revision}: In the proof of the identification of the~$Z^D$ measure with the LQG, the above limit is shown to be one; see~\cite{BL2}. 
\end{remark}

\subsection{Liouville measure, PD statistics and freezing}
\noindent
The last statements to be still proved are those dealing with the limit of the Liouville measure, Poisson Dirichlet statistics of the corresponding atomic law and the freezing phenomenon. All of these pertain to $\beta>\beta_\cc$ where, we recall,~$\beta_\cc=\alpha:=2/\sqrt g$. Throughout this section we suppose that~$D\in\mathfrak D$ and a sequence~$\{D_N\}$ satisfying \twoeqref{E:1.1}{E:1.1a} are given and fixed.

Fix~$\beta>\beta_\cc$ and recall the definition of $\Sigma_{s,Q}$ from before Theorem~\ref{thm-2.7}. Our principal goal is to prove that, for every bounded and continuous function $f\colon D\to[0,\infty)$,
\begin{equation}
\label{E:6.9}
\sum_{z\in D_N}\texte^{\beta (h^{D_N}(z)-m_N)}f(x/N)
\,\,\,\underset{N\to\infty}\Lawarrow\,\,\,
c(\beta) Z^D(D)^{\beta/\beta_\cc}\,\,\,\int_D\Sigma_{\beta_\cc/\beta,\,\widehat Z^D}(\textd x)f(x).
\end{equation}
A natural approach is to write the left-hand side as $\langle\eta^D_N,\tilde f\rangle$, where $\tilde f(x,h):=\texte^{\beta h}f(x)$, and apply Corollary~\ref{cor-cluster-process}. It does not bother us much that $\tilde f$ is unbounded as we can always truncate the maximum from above with high probability. However, the fact that the support of~$\tilde f$ extends all the way to $-\infty$ in the~$h$ variable is much more serious as this could lead to potential blow-ups, which need to be ruled out before the limit $N\to\infty$ is taken.

 Recall the definition of~$\Gamma_N^D(t)$ from \eqref{E:5.47} and for~$\delta>0$ small, set
\begin{equation}
D_N^\delta:=\{x\in D_N\colon \dist(x,D_N^\cc)>\delta N\}\quad\text{and}\quad \Gamma_{N,\delta}^D(t):=\Gamma_N^D(t)\cap D_N^\delta.
\end{equation}
The said blow-ups will be controlled with the help of the following claim:

\begin{proposition}
\label{prop-6.1}
For $D\in\mathfrak D$ there is a constant~$c\in(0,\infty)$ such that for each~$\epsilon>0$ small enough, all $N\ge1$ large and all $t\ge0$,
\begin{equation}
P\bigl(|\Gamma_{N}^D(t)|\ge2\texte^{(\beta_\cc+\epsilon)t}\bigr)\le c(1+t)^2\texte^{-\epsilon t}.
\end{equation}
\end{proposition}

We begin with a lemma:

\begin{lemma}
\label{lemma-6.5}
For each~$D\in\mathfrak D$ and each $\delta>0$ there is a constant $c\in(0,\infty)$ depending only on~$\delta$ and the diameter of~$D$ such that for all $t,s\ge0$,
\begin{equation}
\max_{x\in D_N^\delta}\,P\bigl(h^{D_N}(x)\ge m_N-t,\,h^{D_N}\le m_N+s\bigr)\le c(1+s+t)^2\texte^{\alpha t}\,\frac1{N^2}\,.
\end{equation}
\end{lemma}

\begin{proofsect}{Proof}
For notational convenience suppose that $0\in D_N^\delta$ and instead of the maximum over~$x\in D_N$ in the statement, let us set $x=0$ and take maximum over all shifts of~$D_N$ such that~$0\in D_N^\delta$. The fact that~$\delta>0$ implies that there~$n,q\ge1$, with $n-\log_2N$ and~$q$ bounded by $\delta$-dependent constants uniformly in~$N\ge1$ and all shifts of~$D_N$ for which $0\in D_N^\delta$, such that
\begin{equation}
\Delta^n\subseteq D_N\subseteq \Delta^{n+q}.
\end{equation}
Thanks to Lemma~\ref{lemma-gen-domain}, we can use Lemma~\ref{lemma-4.9new} with~$D_N$ instead of~$\Delta^n$ and from \eqref{E:4.24z} thus get
\begin{equation}
P\Bigl(h^{D_N}\le m_N+s-(m_N+t)\frakg^{D_N}\,\Big|\,h^{D_N}(0)=0\Bigr)\le \frac{c_1}n(1+s+t)^2
\end{equation}
for some $c\in(0,\infty)$.
Lemma~\ref{lemma-3.1} tells us that the probability on the left equals
\begin{equation}
P\bigl(h^{D_N}\le m_N+s\,\big|\,h^{D_N}(0)=m_N-t\bigr)
\end{equation}
and the probability in the statement is thus bounded by
\begin{equation}
\label{E:6.38}
\frac {c_1}n\int_0^{s+t}\textd x\,\,\frac{(1+s+t-x)^2}{\sqrt{2\pi\Var(h^{D_N}(0))}}\,\,\texte^{-\frac12\frac{(m_N+x-t)^2}{\Var(h^{D_N}(0))}}.
\end{equation}
It remains to carefully estimate the integral on the right-hand side. 

Since $\Var(h^{D_N}(0))-g\log N$ is bounded by a constant uniformly in~$N$ and uniformly in the position of~$D_N$  subject to $0\in D_N^\delta$, we have
\begin{equation}
\frac1{\sqrt{\Var(h^{D_N}(0))}}\,\texte^{-\frac12\,m_N^2/\Var(h^{D_N}(0))}\le c_2\,\frac{\log N}{N^2}
\end{equation}
for some $ c_2\in(0,\infty)$. Expanding the square in the exponent and using that $m_N/\Var(h^{D_N}(0))\ge\alpha-c_3(\log\log N)/\log N$ for some $c_3\in(0,\infty)$ we bound
\begin{equation}
\frac{m_N}{\Var(h^{D_N}(0))}(x-t)+\frac12\frac{(x-t)^2}{\Var(h^{D_N}(0))}
\ge\alpha(x-t)-c_3(x-t)\frac{\log\log N}{\log N}+c_4\frac{(x-t)^2}{\log N}
\end{equation}
for some~$c_4\in(0,\infty)$. The last two terms are minimized by $x-t$ of order~$\log\log N$ and so the right-hand side is at least $\alpha(x-t)-c_5$ for some $c_5\in(0,\infty)$. Since $n$ is order~$\log N$, \eqref{E:6.38} is bounded by
\begin{equation}
\frac{c_6}{N^2}(1+s+t)^2\int_0^{s+t}\texte^{-\alpha(x-t)}\,\,\textd x
\end{equation}
for some~$c_6\in(0,\infty)$. The claim follows by simple integration.
\end{proofsect}

\begin{proofsect}{Proof of Proposition~\ref{prop-6.1}}
Let $\delta>0$ be fixed small and let $\widetilde D\supset D$ be such that $\widetilde D^\delta_N\supset D_N$ holds for all~$N\ge1$. From Lemma~\ref{lemma-6.5} we get
\begin{equation}
E\bigl(|\Gamma_{N,\delta}^{\widetilde D}(t)|\1_{\{h^{D_N}\le m_N+s\}}\bigr)\le c_1 (1+s+t)^2\texte^{\alpha t}.
\end{equation}
Lemma~\ref{lemma-Gamma-monotonicity} then allows us to estimate
\begin{equation}
\begin{aligned}
P\bigl(|\Gamma_{N}^D&(t)|\ge2\texte^{(\alpha+\epsilon)t}\bigr)
\le 2P\bigl(|\Gamma_{N,\delta}^{\widetilde D}(t)|\ge\texte^{(\alpha+\epsilon)t}\bigr)
\\
&\le 2P\bigl(\max h^{\widetilde D_N}\le m_N+s,\,|\Gamma_{N,\delta}^{\widetilde D}(t)|\ge\texte^{(\alpha+\epsilon)t}\bigr)+P(\max h^{\widetilde D_N}> m_N+s)
\\
&\le 2c_1(1+s+t)^2\texte^{-\epsilon t}+2c_2 (1+s)\texte^{-\alpha s},
\end{aligned}
\end{equation}
where we used the Markov inequality in the first term and Lemma~\ref{lemma-max-tail} in the second.
Setting $s:=t$ then yields the desired claim. 
\end{proofsect}

This will permit us to complete:

\begin{proofsect}{Proof of Theorem~\ref{thm-2.7}}
Let~$\beta>\beta_\cc$ and let~$\epsilon\in(0,\beta-\beta_\cc)$. In light of the argument at the beginning of this section, the key point is to reduce the sum on the left of \eqref{E:6.9} to those~$z$ where $|h^{D_N}_z-m_N|$ is bounded uniformly in~$N$. This is done as follows: By Proposition~\ref{prop-6.1} and a union bound, there are constants $c,\tilde c\in(0,\infty)$ such that
\begin{equation}
\label{E:6.10}
|\Gamma_{N}^D(t)|\le c\texte^{(\beta_\cc+\epsilon)t},\qquad t\ge t_0,
\end{equation}
occurs with probability at least $1-\tilde c\texte^{-\epsilon t_0/2}$ uniformly in~$N\ge1$. When \eqref{E:6.10} is in force, we have
\begin{multline}
\qquad
\sum_{z\in D_N}\texte^{\beta (h^{D_N}(z)-m_N)}\1_{\{h^{D_N}(z)\le m_N-t_0\}}
\\
\le\sum_{n\ge0}\texte^{-\beta(t_0+n)}\bigl|\Gamma_{N,\delta}^D(t_0+n+1)\bigr|
\le c'\texte^{(\beta_\cc+\epsilon-\beta)t_0}
\qquad
\end{multline}
which is small for~$t_0$ large by our choice of~$\epsilon$. On the other hand, by Lemma~\ref{lemma-max-tail}, the probability that there is any~$z\in D_N$ where $h^{D_N}(z)-m_N\ge t$ is small in~$t$, uniformly in~$N$. It follows that
\begin{equation}
\label{E:6.12}
\lim_{t_0\to\infty}\limsup_{N\to\infty}\,P\biggl(\,\,\sum_{z\in D_N}\texte^{\beta (h^{D_N}(z)-m_N)}f(x/N)\1_{\{h^{D_N}(z)-m_N\not\in[-t_0,t_0]\}}>\delta'\biggr)=0
\end{equation}
holds for each~$\delta'>0$.

We will now invoke the full process convergence to control the contribution of the points where $|h^{D_N}(x)-m_N|\le t_0$. Instead of Corollary~\ref{cor-cluster-process}, it will be more convenient to aim directly at Theorem~\ref{thm-main}. First we invoke Lemma~\ref{lemma-separation} which says that the probability that $|h^{D_N}(z)-m_N|\le t_0$ at $z=x,y$ with $r|x-y|\le N/r$ tends to zero as~$N\to\infty$ and~$r\to\infty$. Thanks to uniform continuity of~$f$, defining $Y^\beta_r(\phi):=\sum_{z\in\Lambda_r(0)}\texte^{-\beta\phi(x)}$ this implies
\begin{multline}
\label{E:6.48}
\quad\lim_{t_0\to\infty}\,\limsup_{r\to\infty}\,\limsup_{N\to\infty}\,
P\biggl(\,\Bigl|\,
\sum_{z\in D_N}\texte^{\beta (h^{D_N}(z)-m_N)}f(x/N)
\\
-\sum_{x\in\Theta_{N,r}}\texte^{\beta (h^{D_N}(x)-m_N)}Y_r(h^{D_N}_x-h^{D_N}_{x+\cdot})f(x/N)\1_{\{h^{D_N}(x)-m_N\in[-t_0,t_0]\}}\Bigr|>\delta\biggr)=0,
\quad
\end{multline}
where we also invoked \eqref{E:6.12} in order to be able to write only an indicator involving~$h_x^{D_N}$ in the second sum. This sum can be written as $\langle\eta^D_{N,r},\,F_{r,\,t_0}\rangle$, where
\begin{equation}
F_{r,\,t_0}(x,h,\phi):=f(x)\,\texte^{\beta h}\1_{\{|h|\le t_0\}}\,Y_r^\beta(\phi).
\end{equation}
Thanks to \eqref{E:6.48}, the limit proved in Theorem~\ref{thm-main} shows
\begin{multline}
\label{E:6.13}
\quad
E\biggl(\exp\Bigl\{-\sum_{z\in D_N}\texte^{\beta (h^{D_N}(z)-m_N)}f(x/N)\Bigr\}\biggr)
\\
\underset{N\to\infty}\longrightarrow\,\,\,\lim_{r\to\infty}
E\biggl(\exp\Bigl\{-\int Z^D(\textd x)\otimes \texte^{-\alpha h}\textd h\otimes\nu(\textd\phi)(1-\texte^{-F_r(x,h,\phi)})\Bigr\}\biggr),
\quad
\end{multline}
where, with the help of the Monotone Convergence Theorem, we already took the limit $t_0\to\infty$ inside the expectation and replaced $F_{r,\,t_0}$ by
\begin{equation}
F_r(x,h,\phi):=f(x)\,\texte^{\beta h}\,Y^\beta_r(\phi).
\end{equation}
Note that the integral on the right of \eqref{E:6.13} is finite a.s.\ because $Z^D(D)<\infty$ a.s.\ and $1-\texte^{-F_r(x,h,\phi)}$ decays proportionally to $\texte^{\beta h}$ as~$h\to-\infty$ and~$\beta>\alpha$.

Next we will address the limit $r\to\infty$ in \eqref{E:6.13}. Focusing first on the integral on the right-hand side, the change of variables $t:=\texte^{\beta h}$ gives us
\begin{equation}
\label{E:6.14}
\begin{aligned}
\int\texte^{-\alpha h}\textd h(1-\texte^{-F_r(x,h,\phi)})
&=\int_0^\infty\frac{\textd t}{\beta t}\, t^{-\alpha/\beta}\bigl(1-\texte^{-f(x)Y^\beta_r(\phi)t}\bigr)
\\
&=\frac1\beta \bigl(Y^\beta_r(\phi)\bigr)^{\alpha/\beta}\int_0^\infty\textd t\,t^{-1-\alpha/\beta}(1-\texte^{-tf(x)}).
\end{aligned}
\end{equation}
By our arguments above, the left-hand side of \eqref{E:6.13} is positive for $f(x):=1$, and since $Z^D(D)>0$ a.s., the right-hand side of \eqref{E:6.14} must remain bounded as~$r\to\infty$ even after taking expectation with respect to~$\nu$. The Monotone Convergence Theorem then gives $E_\nu(Y^\beta(\phi)^{\alpha/\beta})<\infty$ and, in particular, $Y(\phi)<\infty$ $\nu$-a.s. The $r\to\infty$ limit can be taken inside \eqref{E:6.13} thus replacing $F_r$ by
\begin{equation}
F(x,h,\phi)=f(x)\,\texte^{\beta h}\,Y^\beta(\phi).
\end{equation}
It remains to identify the resulting expression on the right-hand side of \eqref{E:6.13} with the Laplace transform of the right-hand side of \eqref{E:6.9}.

The definition of $\Sigma_{s,Q}$ reads
\begin{equation}
\label{E:6.15}
E\biggl(\exp\Bigl\{-\lambda\int\Sigma_{s,Q}(\textd x)f(x)\Bigr\}\biggr)
=\exp\Bigl\{-\lambda^s\int Q(\textd x)\otimes\textd t\,\,t^{-1-s}(1-\texte^{-tf(x)})\Bigr\}
\end{equation}
Setting
\begin{equation}
s:=\alpha/\beta,\quad\lambda^s:=\beta^{-1}E_\nu\bigl(Y^\beta(\phi)^{\alpha/\beta}\bigr)\,Z^D(D)\quad\text{and}\quad Q:=\widehat Z^D
\end{equation}
 \eqref{E:6.14} identifies the exponential on the right-hand side of \eqref{E:6.13} (with $F_r$ replaced by~$F$) with that on the right of \eqref{E:6.15}. Taking expectation with respect to~$Z^D$ yields~\eqref{E:6.9} with $c(\beta)$ as given in~\eqref{E:2.9b}.
\end{proofsect}

\begin{proofsect}{Proof of Corollary~\ref{cor-2.7}}
By using test functions of the form $\lambda_1+\lambda_2 f(x)$ in the convergence in Theorem~\ref{thm-2.7} we find out that
\begin{multline}
\quad
\biggl(\,\,\sum_{z\in D_N}\texte^{\beta (h^{D_N}_z-m_N)},\,\sum_{z\in D_N}\texte^{\beta (h^{D_N}_z-m_N)}f(x/N)\biggr)
\\
\underset{N\to\infty}\Lawarrow\,\,\,
c(\beta) Z^D(D)^{\beta/\beta_\cc}\,\,\,\biggl(\,\Sigma_{\beta_\cc/\beta,\,\,\widehat Z^D}(D),\,\,\int_D\Sigma_{\beta_\cc/\beta,\,\widehat Z^D}(\textd x)f(x)\biggr).
\quad
\end{multline}
The claim follows by dividing these two terms and noting that if $\{q_i\}$ are the points of the Poisson process defining $\Sigma_{s,Q}$ as in \eqref{E:2.17}, then $\Sigma_{s,Q}(D)=\sum_i q_i$ and $\{q_i/\sum_jq_j\colon i\ge1\}$, reordered according to size, constitute a sample from~PD($s$).
\end{proofsect}

\begin{proofsect}{Proof of Corollary~\ref{cor-2.8}}
Recall the function $G_{N,\beta}(t)$ from \eqref{E:2.22ue}. Using the same arguments as in the proof of Theorem~\ref{thm-2.7}, setting $f(h,\phi):=\texte^{\beta (h-y)}Y^\beta(\phi)$ yields
\begin{equation}
G_{N,\beta}(y+m_N)\,\underset{N\to\infty}\longrightarrow\,
E\biggl(\exp\Bigl\{-Z^D(D)\int\texte^{-\alpha h}\textd h\otimes\nu(\textd\phi)(1-\texte^{-f(h,\phi)})\Bigr\}\biggr).
\end{equation}
By change of variables again,
\begin{equation}
\begin{aligned}
\int\texte^{-\alpha h}\textd h(1-\texte^{-f(h,\phi)})
&=\int_0^\infty\frac{\textd t}{\beta t}\, t^{-\alpha/\beta}\bigl(1-\texte^{-\texte^{-\beta y}Y^\beta(\phi)t}\bigr)
\\
&=\frac1\beta \bigl(\texte^{-\beta y}Y^\beta(\phi)\bigr)^{\alpha/\beta}\int_0^\infty\textd t\,t^{-1-\alpha/\beta}(1-\texte^{-t}).
\end{aligned}
\end{equation}
Defining
\begin{equation}
\label{E:6.54a}
\tilde c(\beta):=\frac1\alpha\log\biggl(\frac1\beta E_\nu\bigl(Y(\theta)^{\alpha/\beta}\bigr)\int_0^\infty\textd t\,t^{-1-\alpha/\beta}(1-\texte^{-t})\biggr),
\end{equation}
we then get
\begin{equation}
G_{N,\beta}(y+m_N+\tilde c(\beta))\,\underset{N\to\infty}\longrightarrow\,E\bigl(\texte^{-Z^D(D)\,\texte^{-\alpha y}}\bigr)
\end{equation}
exactly as desired.
\end{proofsect}

\newcounter{appendA}
\renewcommand{\theappendA}{A}
\section*{Appendix A: Brownian paths above a curve}
\addcontentsline{toc}{section}
{{\tocsection {}{7}{\!\!\!\!Appendix A: Brownian paths above a curve\dotfill}}{}}
\refstepcounter{appendA}
\label{appendA}
\renewcommand{\thesection}{A}
\setcounter{subsection}{0}
\setcounter{theorem}{0}
\setcounter{equation}{0}
\renewcommand{\thesubsection}{A.\arabic{subsection}}
\renewcommand{\theequation}{\thesection.\arabic{equation}}
\noindent
The goal of this section is prove Propositions~\ref{prop-uncond-BM-positive-cor}, \ref{prop-mine-corollary},~\ref{prop-4.9ue}, \ref{prop-mine-too-corollary}, \ref{prop-entropy-BM} and~\ref{prop-entropy-BB} dealing with probabilities that Brownian motion and Brownian bridge avoid hitting a given curve. Many of the calculations presented here appear in some form in other places including the literature on the \DGFF{} and the Branching Brownian Motion. However, as discussed in Remark~\ref{remark-Bramson}, our applications require a level of precision and generality that forces us to furnish 
independent proofs.

\subsection{Consequences of the Reflection Principle}
\label{sec-A1}\noindent
Let $\{B_t\colon t\ge0\}$ be a standard Brownian motion. For a continuous function $g\colon[0,\infty)\to[0,\infty)$, define
\begin{equation}
\tau_g:=\inf\bigl\{s\ge0\colon B_s-g(s)=0\bigr\}.
\end{equation}
(In particular, for $g(s):=x$ we will write $\tau_x$ and for $g(s):=x+\zeta(s)$ we will write $\tau_{x+\zeta}$, etc.) We begin by dealing with Brownian motion and Brownian bridge above a \emph{constant} curve which boils down to standard applications of the Reflection Principle. The following specific facts will be quite handy in various calculations below:

\begin{lemma}
\label{lem:2.8}
For all~$x>0$ and all~$t>0$, 
\begin{equation}
\label{E:4.16b}
P^0\bigl(\tau_x\in\textd t\bigr)\le \frac1{\sqrt{2\pi}}\,\frac x{t^{3/2}}\,\textd t
\end{equation}
and 
\begin{equation}
\label{E:4.16q}
\sqrt{\frac2\pi}\,\frac x{\sqrt t}\Bigl(1-\frac{x^2}{2t}\Bigr)\le P^0\bigl(\tau_x> t\bigr)\le\sqrt{\frac2\pi}\,\frac x{\sqrt t}\,.
\end{equation}
Similarly, for each $t>0$ and each $x,y>0$,
\begin{equation}
\label{E:4.17q}
\Bigl(1-\frac{xy}t\Bigr)\frac{2xy}t\le
P^x \bigl( \tau_0>t\,\big|\, B_t = y \bigr)\le\frac{2xy}t.
\end{equation}
Finally, let $M_t^\star:=\max_{s\le t}B_s$ and $T_t^\star:=\sup\{s\le t\colon B_s=M_t^\star\}$. Then $B_{T_t^\star}=M_t^\star$ and
\begin{equation}
\label{E:max-BM}
P^0\bigl(T_t^\star\in \textd s,\,M^\star_t\in \textd z\bigr)=\frac{z\,\texte^{-\frac{z^2}{2s}}}{\pi\,s^{3/2}\sqrt{t-s}}\,\1_{\{0\le s\le t\}}\1_{\{z\ge0\}}\,\textd s\,\textd z.
\end{equation}
 (Here and henceforth, $P(X\in\textd x)=f(x)\textd x$ is a shorthand for $P(X\in A)=\int_A f(x)\textd x$.)
\end{lemma}

\begin{proof}
These claims are standard and can be found (albeit perhaps not in such  a  compact form) in various textbooks. We  provide proofs for completeness of exposition. 

The Reflection Principle shows
\begin{equation}
\label{E:7.6u}
P^0\bigl(\tau_x>t\bigr)
\\
=P^0\bigl(|B_1|\le\tfrac x{\sqrt t}\bigr).
\end{equation}
Then \twoeqref{E:4.16b}{E:4.16q} follow by bounding the probability density of~$B_1$ on~$[-\frac x{\sqrt t},\frac x{\sqrt t }]$ by~$(2\pi)^{-1/2}$ from above and by $(2\pi)^{-1/2}(1-\frac{x^2}{2t})$ from below.
Another application of the Reflection Principle yields 
\begin{equation}
\label{E:4.20}
	P^x \bigl( \tau_0>t\,\big|\, B_t = y \big)
	= 1 - 
		\exp \bigl\{ -\tfrac{2yx}{t} \big\}.
\end{equation}
The bounds \eqref{E:4.17q} again follow from $a-\frac{a^2}2\le 1-\texte^{-a}\le a$ valid for all~$a\ge0$. For \eqref{E:max-BM} we note that, by the Strong Markov Property and path-continuity of the Brownian motion,
\begin{equation}
P^0\bigl(T_t^\star>s,\,M^\star_t\ge z\bigr)=\lim_{\epsilon\downarrow0}\,
\sum_{k\ge0}\int_{(s,t]}P^0\bigl(\tau_{z+k\epsilon}\in\textd u\bigr)P^0(\tau_\epsilon>t-u).
\end{equation}
Using that \eqref{E:4.16q} is sharp in the limit $x\downarrow0$, this shows
\begin{equation}
P^0\bigl(T_t^\star>s,\,M^\star_t\ge z\bigr)=\sqrt{\frac2\pi}\int_{[z,\infty)}\textd\tilde z\int_{(s,t]}P^0(\tau_{\tilde z}\in\textd u)\frac1{\sqrt{t-u}}.
\end{equation}
The result then follows by \eqref{E:7.6u} and differentiation. That $B_{T_t^\star}=M_t^\star$ holds is a consequence of path continuity of the Brownian motion.
\end{proof}

\subsection{Positive curves: Brownian motion}
\label{sec-A2a}\noindent
Our next task is the control of Brownian motion above a positive curve. The following is key in the proof of Proposition~\ref{prop-uncond-BM-positive-cor}:

\begin{proposition}
\label{prop-uncond-BM-positive}
Let~$\zeta\colon[0,\infty)\to[0,\infty)$ be non-decreasing, continuous and obeys $\zeta(s)=o(s^{1/2})$ as~$s\to\infty$. Then for all~$x>\zeta(0)$ and all $t>0$,
\begin{equation}
\label{E:A.11ab}
P^0\bigl(\tau_{x-\zeta}\le t<\tau_x\bigr)
\le2\,\frac{\rho(x)^{2/3}\,x^{1/3}}{\sqrt t},
\end{equation}
where $\rho(x)$ be the quantity from \eqref{E:rho}.
\end{proposition}

This is indeed the case, as we see from:

\begin{proofsect}{Proof of Proposition~\ref{prop-uncond-BM-positive-cor}}
 From~$\zeta\ge0$ we  have
\begin{equation}
P^0(\tau_x>t)=P^0(\tau_{x-\zeta}>t)+P^0\bigl(\tau_{x-\zeta}\le t<\tau_x\bigr).
\end{equation}
To get the claim we just subtract \eqref{E:A.11ab} from the left-hand side of \eqref{E:4.16q}.
\end{proofsect}

Also the proof of Proposition~\ref{prop-uncond-BM-positive} begins by a slightly stricter estimate:

\begin{lemma}
\label{lemma-negative}
Let~$\zeta$ and~$\rho$ be as above. Then for all $x>\zeta(0)$ and all $\delta>0$,
\begin{equation}
\label{E:A.11}
P^0\Bigl(\tau_{x-\zeta}<t,\,\tau_x>(1+\delta) t\Bigr)
\le\frac1{\sqrt{\delta}}\,\frac{\rho(x)}{\sqrt t}.
\end{equation}
\end{lemma}

\begin{proofsect}{Proof}
A routine approximation argument (based on path continuity of the Brownian motion) permits us to assume that~$\zeta$ is continuously differentiable.
Using the Strong Markov Property  for  the stopping time~$\tau_{x-\zeta}$ yields
\begin{equation}
P^0\Bigl(\tau_{x-\zeta}<t,\,\tau_x>(1+\delta)  t\Bigr)=\int_0^t P^0\bigl(\tau_{x-\zeta}\in\textd s\bigr)P^0\bigl(\tau_{\zeta(s)}>(1+\delta) t-s\bigr).
\end{equation}
By \eqref{E:4.16q}, the second probability is at most $\zeta(s)/\sqrt{\delta t}$ for all $s\in[0,t]$. Integrating by parts and using  the  positivity of~$\zeta(t)$  along with $2/\pi\le1$,  the integral is then at most
\begin{equation}
\frac1{\sqrt t}\,\frac1{\sqrt\delta}
\Bigl[\zeta(0)+\int_0^t\zeta'(s)P^0\bigl(\tau_{x-\zeta}>s\bigr)\textd s\Bigr].
\end{equation}
Our goal is to show that the term in the brackets is bounded by~$\rho(x)$. For this we first take~$t\to\infty$ and then bound $P^0(\tau_{x-\zeta}>s)$ by $P^0(\tau_x>s)$. Then we integrate by parts and use that~$\zeta(s)=o(s^{1/2})$ (because otherwise~$\rho(x)=\infty$) to find that the  square  bracket is bounded by the left-hand side of
\begin{equation}
\label{E:A.14}
\int_0^\infty P^0(\tau_x\in\textd s)\zeta(s)\le\zeta(x^2)+\frac x2\int_{x^2}^\infty\frac{\zeta(s)}{s^{3/2}}\,\textd s.
\end{equation}
Here, to get the bound on the right, we split the integral at~$s:=x^2$, invoked the monotonicity of~$\zeta$ in the first one while applied \eqref{E:4.16q} and also used $\sqrt{2\pi}\ge2$ in the second one.
\end{proofsect}

\begin{proofsect}{Proof of Proposition~\ref{prop-uncond-BM-positive}}
We have
\begin{equation}
P^0\bigl(\tau_{x-\zeta}<t<\tau_x\bigr)\le
P^0\Bigl(\tau_{x-\zeta}<t,\,\tau_x>(1+\delta) t\Bigr)
+P^0\Bigl(t\le\tau_{x}<(1+\delta) t\Bigr).
\end{equation}
Using \eqref{E:4.16b} and \eqref{E:A.11}, we can bound the right-hand side by
\begin{equation}
\frac1{\sqrt\delta}\,\frac{\rho(x)}{\sqrt t}+\frac1{\sqrt{2\pi}}\frac{x\delta}{\sqrt t}.
\end{equation}
 Invoking  $\sqrt{2\pi}\ge2$, the resulting expression is, as a function of~$\delta$, minimized at $\delta:=(\rho(x)/x)^{2/3}$. Plugging this back in, we get the claim.
\end{proofsect}

\subsection{Positive curves: Brownian bridge}
\label{sec-A2}\noindent
Next we will move to the case of the Brownian bridge above a positive curve. The goal is to prove Proposition~\ref{prop-mine-corollary} whose key part is the following claim:

\begin{proposition}
\label{prop-mine}
Let~$\zeta\colon[0,\infty)\to[0,\infty)$ be non-decreasing, continuous  with  $\zeta(s)=o(s^{1/2})$ as~$s\to\infty$. Then for all~$x,y>\zeta(0)$ and all $t>0$,
\begin{multline}
\label{E:4.41}
\qquad
P^x\Bigl( \,\min_{0\le s\le t}\bigl[B_s-\zeta(s\wedge(t-s))\bigr]<0<\min_{0\le s\le t} B_s\,\Big|\,B_t=y\Bigr)
\\
\le  8\biggl(\sqrt{\frac{\rho(x)}x}+\sqrt{\frac{\rho(y)}y}\biggr)\,\frac{xy}t\,\,\texte^{\frac{(x-y)^2}{2t}},
\qquad
\end{multline}
where $\rho(x)$ is the quantity from \eqref{E:rho}.
\end{proposition}

We first  check that this indeed implies the desired bound:

\begin{proofsect}{Proof of Proposition~\ref{prop-mine-corollary}}
We just combine \eqref{E:4.41} with the bound on the left of~\eqref{E:4.17q}.
\end{proofsect}

The proof of Proposition~\ref{prop-mine} will be based on Lemma~\ref{lemma-negative} and some interesting technical ingredients. The first one is inspired by arguments from~Bramson~\cite[proof of Proposition~1']{Bramson-CPAM}:

\begin{lemma}[Decoupling lemma]
\label{lemma-decouple}
Pick~$x,y\in\R$, let $t,t_1,t_2>0$ be such that $t_1+t_2<t$ and let $A_1$ and~$A_2$ be events such that $A_i\in\sigma(B_s\colon 0\le s\le t_i)$, $i=1,2$.
Consider also the event
\begin{equation}
\label{E:4.29}
A_2':=\bigl\{\,\text{\rm path }\{B_{t-s}\colon 0\le s\le t_2\}\text{\rm \ lies in }A_2\bigr\}.
\end{equation}
 Then
\begin{equation}
P^x\bigl(A_1\cap A_2'\big|B_t=y\bigr)\le\sqrt{\frac t{t-t_1-t_2}}\,\,\texte^{\frac{(x-y)^2}{2t}}\,\,
P^x(A_1)P^y(A_2).
\end{equation}
\end{lemma}

\begin{proofsect}{Proof}
Define the functions
\begin{equation}
f_1(z):=P^x(A_1|B_{t_1}=z)\quad\text{and}\quad f_2(z):=P^y(A_2|B_{t_2}=z).
\end{equation}
Conditioning on~$B_{t_1}$ and $B_{t-t_2}$ and invoking the reversibility of the Brownian bridge then yields
\begin{equation}
\label{E:4.33}
P^x\bigl(A_1\cap A_2'\big|B_t=y\bigr) = E^x\bigl( f_1(B_{t_1}) f_2(B_{t-t_2})\,\big|\, B_t=y\bigr).
\end{equation}
 Next  let $p^{x,y}_{t_1,\,t-t_2}(x_1,x_2)$ denote the joint probability density of $(B_{t_1},B_{t-t_2})$ in measure $P^x(-|B_t=y)$ and let $g_t(x)$ denote the probability density of a normal random variable with mean zero and variance~$t$. Then
\begin{equation}
p^{x,y}_{t_1,\,t-t_2}(x_1,x_2)=\frac{g_{t_1}(x_1-x)\,g_{t-t_1-t_2}(x_2-x_1)\,g_{t_2}(y-x_2)}{g_t(y-x)}.
\end{equation}
Using the explicit form of~$g_t$ in the denominator and the bound $g_t(x)\le\frac1{\sqrt{2\pi t}}$ for the middle term in the numerator gives
\begin{equation}
p^{x,y}_{t_1,\,t-t_2}(x_1,x_2)
\le \sqrt{\frac t{t-t_1-t_2}}\,\,\texte^{\frac{(x-y)^2}{2t}}\,\, g_{t_1}(x_1-x)\,g_{t_2}(y-x_2).
\end{equation}
Since both~$f_1$ and~$f_2$ are positive, plugging this in \eqref{E:4.33} then yields the claim.
\end{proofsect}

Lemma~\ref{lemma-decouple} permits us to effectively ``tear'' a Brownian bridge apart into two independent Brownian paths. The next lemma will in turn let us symmetrize the events about the midpoint of the interval, and thus assume that the starting and ending points of the Brownian bridge are the~same. 

\begin{lemma}[Symmetrization lemma]
\label{lemma-symmetrize}
Let $t>0$ be given and let $A_1,A_2\in\sigma(B_s\colon 0\le s\le t/2)$. Let $A'_1$, resp.,~$A_2'$ be related to $A_1$, resp.,~$A_2$ as in \eqref{E:4.29}. Then for all $x,y\in\R$,
\begin{equation}
P^x(A_1\cap A_2'|B_t=y)
\le P^x(A_1\cap A_1'|B_t=x)^{\ffrac12}\,P^y(A_2\cap A_2'|B_t=y)^{\ffrac12}\,\texte^{\frac{(x-y)^2}{2t}}.
\end{equation}
\end{lemma}

\begin{proofsect}{Proof}
Denote
\begin{equation}
f_{x,\,y}(z):=\sqrt{\frac2{\pi t}}\,\texte^{-\frac2t(z-\frac{x+y}2)^2}.
\end{equation}
As is easy to check, $f_{x,y}$ is the (probability) density of $P^x(B_{t/2}\in \cdot|B_t=y)$ with respect to the Lebesgue measure. Therefore,
\begin{equation}
\label{E:4.37}
P^x(A_1\cap A_2'|B_t=y)=\int_{\R}\textd z\,f_{x,\,y}(z)\,P^x(A_1|B_{t/2}=z)\,P^y(A_2|B_{t/2}=z).
\end{equation}
A calculation shows
\begin{equation}
f_{x,\,y}(z)=f_{x,\,x}(z)^{1/2}\,f_{y,\,y}(z)^{1/2}\,\texte^{\frac{(x-y)^2}{2t}}.
\end{equation}
The claim follows by plugging this in the above integral, invoking the Cauchy-Schwarz inequality and wrapping the result together using again \eqref{E:4.37}.
\end{proofsect}

\begin{proofsect}{Proof of Proposition~\ref{prop-mine}}
Consider the events
\begin{equation}
A_1:=\bigl\{\min_{s\le t/2}(B_s-\zeta(s))<0<\min_{s\le t/2}B_s\bigr\}
\quad\text{and}\quad A_2:=\bigl\{\min_{s\le t/2}B_s>0\bigr\}\,.
\end{equation}
Using $A'_i$ to denote the path reversal of event~$A_i$ as in Lemma~\ref{lemma-decouple}, the event in the statement is contained in $(A_1\cap A_2')\cup(A_2\cap A_1')$. By Lemma~\ref{lemma-symmetrize}, the desired probability is thus at most
\begin{multline}
\label{E:4.41e}
\qquad
P^x(A_1\cap A_1'|B_t=x)^{1/2}P^y(A_2\cap A_2'|B_t=y)^{1/2}
\\+P^x(A_2\cap A_2'|B_t=x)^{1/2}P^y(A_1\cap A_1'|B_t=y)^{1/2}
\qquad
\end{multline}
times $\,\texte^{\frac{(x-y)^2}{2t}}$. By symmetry, it suffices to bound the first term in \eqref{E:4.41e} by $8\sqrt{\frac{\rho(x)}x}\frac{xy}t$.

We introduce two additional events
\begin{equation}
A_3:=\bigl\{\min_{s\le t/2}(B_s-\zeta(s))<0<\min_{s\le \frac34t}B_s\bigr\}
\quad\text{and}\quad A_4:=\bigl\{\min_{s\le \frac18 t}B_s>0\bigr\}.
\end{equation}
Then $A_1\cap A_1'\subseteq A_3\cap A_4'$ and so, by Lemma~\ref{lemma-decouple} with $t_1:=\frac34t$ and $t_2:=\frac18t$,
\begin{equation}
P^x(A_1\cap A_1'|B_t=x)
\le P^x(A_3\cap A_4'|B_t=x)\le\sqrt8\,P^x(A_3)P^x(A_4).
\end{equation}
We now note that, by Lemma~\ref{lemma-negative},
\begin{equation}
P^x(A_3)=P^0\bigl(\tau_{x-\zeta}<\tfrac12t,\,\tau_x>\tfrac34t\bigr)
\le 2\,\frac{\rho(x)}{\sqrt t},
\end{equation}
while Lemma~\ref{lem:2.8} gives
\begin{equation}
P^x(A_4)=P^0\bigl(\tau_x>\tfrac18t\bigr)\le\sqrt8\,\frac x{\sqrt t}
\end{equation}
and
\begin{equation}
P^y(A_2\cap A_2'|B_t=y)=P^y(\tau_0>t|B_t=y)\le\frac{2y^2}t.
\end{equation}
Combining these facts, we get the desired statement.
\end{proofsect}

\subsection{Negative curves}
\label{sec-A3}\noindent
The control of Brownian motion and Brownian bridge above negative curves involves considerably heavier calculations but we will also be able to enjoy the benefits of our earlier work. The following is the main ingredient for the proof of Proposition~\ref{prop-4.9ue}:

\begin{proposition}
\label{prop-uncond-BM-negative}
Let~$\zeta\colon[0,\infty)\to[0,\infty)$ be non-decreasing, continuous  with  $\zeta(s)=o(s^{1/4})$ as~$s\to\infty$. Then for all~$x>0$ and all $t>0$,
\begin{equation}
\label{E:A.11a}
P^0\bigl(\tau_x<t<\tau_{x+\zeta}\bigr)
\le2\,\biggl(1+\Bigl(\frac{\tilde\rho(x)}x\Bigr)^{2/3}\biggr)\frac{\tilde\rho(x)^{2/3}\,x^{1/3}}{\sqrt t},
\end{equation}
where $\tilde\rho(x)$ be the quantity from \eqref{E:tilde-rho}.
\end{proposition}
  
\begin{proofsect}{Proof of Proposition~\ref{prop-4.9ue}}
Just add \eqref{E:A.11a} to the right-hand side of \eqref{E:4.16q}.
\end{proofsect}

The proof of Proposition~\ref{prop-uncond-BM-negative} will follow the same strategy as for the case of positive curves. We start with a slightly stricter version of the desired estimate:

\begin{lemma}
\label{lemma-positive}
For~$\zeta$ and $\tilde\rho$ as above, any $x>0$, any $t>0$ and any $\delta>0$,
\begin{equation}
\label{E:4.23u}
P^0\Bigl(\tau_x<t,\,\tau_{x+\zeta}>(1+\delta)t\Bigr)\le \frac1{\sqrt\delta}\,\frac{\tilde\rho(x)}{\sqrt t}.
\end{equation}
\end{lemma}

The proof is an augmented version of the argument from the proof of Lemma~\ref{lemma-negative}. However, we will need to replace the bounds in Lemma~\ref{lem:2.8} by the following estimate:

\begin{lemma}
\label{lemma-A.9u}
For~$\zeta\colon[0,\infty)\to[0,\infty)$ non-decreasing and continuously differentiable and $t>0$,
\begin{equation}
P^0(\tau_\zeta>t)\le\frac 1{\sqrt t}\biggl(\zeta(0)+2\int_0^t\textd u\,\frac{\zeta(u)\zeta'(u)}{\sqrt u}\biggr).
\end{equation}
\end{lemma}

\begin{proofsect}{Proof}
The proof is inspired by an argument from the proof of Bramson~\cite[Proposition~1]{Bramson-CPAM}.
Recall our earlier definitions $M_t^\star:=\max_{s\le t}B_s$ and $T_t^\star:=\sup\{s\le t\colon B_s=M_t^\star\}$. By \eqref{E:max-BM} and \eqref{E:4.16q}, we have
\begin{equation}
\begin{aligned}
P^0(\tau_\zeta>t)
&\le P^0(\tau_{\zeta(0)}>t)+P^0\bigl(\zeta(T^\star_t)>M_t^\star\ge\zeta(0)\bigr)
\\
&\le\frac{\zeta(0)}{\sqrt t}+\int_0^t\textd s\int_{\zeta(0)}^{\zeta(s)}\textd z\,
\frac{z\,\texte^{-\frac{z^2}{2s}}}{\pi\,s^{3/2}\sqrt{t-s}}\,.
\end{aligned}
\end{equation}
We need to estimate the integral on the right. By way of a routine approximation argument we may assume that $\zeta$ is invertible and $\zeta^{-1}$ thus exists on $[\zeta(0),\zeta(t)]$. This permits us to write
\begin{equation}
\label{E:4.55}
\int_0^t\textd s\int_{\zeta(0)}^{\zeta(s)}\textd z\,
\frac{z\,\texte^{-\frac{z^2}{2s}}}{\pi\,s^{3/2}\sqrt{t-s}}
=\int_{\zeta(0)}^{\zeta(t)}\textd z\int_{\zeta^{-1}(z)}^t\textd s\,\frac{z\,\texte^{-\frac{z^2}{2s}}}{\pi\,s^{3/2}\sqrt{t-s}}.
\end{equation}
We estimate the inner integral by splitting the domain around the point $(t/2)\wedge\zeta^{-1}(z)$. After some straightforward calculations,
the double integral in \eqref{E:4.55} is bounded by
\begin{equation}
\frac{4\sqrt2}{\pi\sqrt t}\int_{\zeta(0)}^{\zeta(t)}\textd z\,\frac{z}{\sqrt{\zeta^{-1}(z)}}.
\end{equation}
Substituting $z:=\zeta(s)$ and using that $4\sqrt2\le2\pi$ we then readily get the claim.
\end{proofsect}

\begin{proofsect}{Proof of Lemma~\ref{lemma-positive}}
As before, approximation arguments permit us to assume that~$\zeta$ is continuously differentiable.
As in the proof of Lemma~\ref{lemma-negative}, we write
\begin{equation}
P^0\Bigl(\tau_x<t,\,\tau_{x+\zeta}>(1+\delta)t\Bigr)
=\int_{[0,t]}P^0\bigl(\tau_{x}\in\textd s\bigr)P^0\bigl(\tau_{\zeta(s+\cdot)}>(1+\delta)t-s\bigr).
\end{equation}
Using that $(1+\delta)t-s\ge \delta t$ throughout the domain of integration, we then use Lemma~\ref{lemma-A.9u} to bound the (last) probability on the right by $f(s)/\sqrt{\delta t}$, where
\begin{equation}
f(s):=\zeta(s)+2\int_0^\infty\textd u\,\frac{\zeta(s+u)\zeta'(s+u)}{\sqrt u}.
\end{equation}
The contribution of the first term on the right-hand side is handled by the argument in the proof of Lemma~\ref{lemma-negative} (specifically, \eqref{E:A.14}). Using also \eqref{E:4.16b}, we thus get
\begin{multline}
\label{E:4.61}
\qquad
P^0\Bigl(\tau_x<t,\,\tau_{x+\zeta}>(1+\delta)t\Bigr)\le
\frac{\rho(x)}{\sqrt{\delta t}}
\\+\frac{2}{\sqrt{2\pi \delta t\,}}\int_0^\infty\textd s\int_0^\infty\textd u\,\frac{x}{s^{3/2}}\texte^{-\frac{x^2}{2s}}\frac{\zeta(s+u)\zeta'(s+u)}{\sqrt u}.
\qquad
\end{multline}
Next we invoke the substitution $w:=s+u$ and $v:= u/s$. The Jacobian of the transformation equals $w/(1+v)^2$. After some cancelations, the double integral in \eqref{E:4.61} thus becomes
\begin{equation}
\int_0^\infty\textd w\,\,\texte^{-\frac{x^2}{2w}}\,\frac{x\zeta(w)\zeta'(w)}{w}\int_0^\infty\textd v\,\frac1{\sqrt v}\,\texte^{-\frac{x^2}{2w}v}.
\end{equation}
The inner integral evaluates to $\sqrt{2\pi w}/x$ and so the last term on the right-hand side of \eqref{E:4.61} is bounded by $1/\sqrt{\delta t}$ times
\begin{equation}
\label{E:A.46a}
2\int_0^\infty\textd s\,\,\texte^{-\frac{x^2}{2s}}\frac{\zeta(s)\zeta'(s)}{\sqrt s}
=\frac12\int_0^\infty\textd s\,\zeta(s)^2\texte^{-\frac{x^2}{2s}}\Bigl(\frac{x^2}{s^{5/2}}+s^{-3/2}\Bigr),
\end{equation}
where we integrated by parts and used that~$\zeta(s)=o(s^{1/4})$ as~$s\to\infty$. We thus have to show that \eqref{E:A.46a} is bounded by the sum of the last two terms in the definition of~$\tilde\rho$ in~\eqref{E:tilde-rho}.

Consider the integral on the right of \eqref{E:A.46a}. We first absorb $x^2/(2s)$ from the first term in the parenthesis at the cost of changing ``2'' to~``4'' in the denominator of the exponent of the exponential. The right-hand side of \eqref{E:A.46a} is then bounded by
\begin{equation}
\label{E:A.46}
\frac12(4\texte^{-1}+1)\int_0^\infty\textd s\,\frac{\zeta(s)^2}{s^{3/2}}\texte^{-\frac{x^2}{4s}}.
\end{equation}
We split the integral at~$s:=x^2$ and, in the part corresponding to $s\in[0,x^2]$, use $s^{-1}\texte^{-\frac{x^2}{4s}}\le 4\texte^{-1}x^{-2}$ to bound the expression in \eqref{E:A.46} by
\begin{equation}
2\texte^{-1}(4\texte^{-1}+1)\frac1{x^2}\int_0^{x^2}\textd s\,\frac{\zeta(s)^2}{s^{1/2}}+\frac12(4\texte^{-1}+1)
\int_{x^2}^\infty\textd s\,\frac{\zeta(s)^2}{s^{3/2}}.
\end{equation}
Using $\zeta(s)\le\zeta(x^2)$ inside the first integral, the result follows by elementary calculations.
\end{proofsect}

\begin{proofsect}{Proof of Proposition~\ref{prop-uncond-BM-negative}}
Fix $\delta:=(\tilde\rho(x)/x)^{2/3}$ and note that 
\begin{equation}
P^0\Bigl(\tau_x<(1+\delta)t<\tau_{x+\zeta}\Bigr)
\le P^0\Bigl(\tau_x<t,\,\tau_{x+\zeta}>(1+\delta)t\Bigr)+P^0\Bigl(t<\tau_x\le(1+\delta)t\Bigr).
\end{equation}
As in the proof of Proposition~\ref{prop-uncond-BM-positive}, the right-hand side is bounded using \eqref{E:4.16b} and \eqref{E:4.23u} by an expression that evaluates to $2\tilde\rho(x)^{2/3} x^{1/3}/\sqrt t$. The claim follows by relabeling $(1+\delta)t$ for~$t$.
\end{proofsect}

Moving over to the case of Brownian bridge, the key estimate in the proof of Proposition~\ref{prop-mine-too-corollary} is as follows:

\begin{proposition}
\label{prop-mine-too}
Let~$\zeta\colon[0,\infty)\to[0,\infty)$ be non-decreasing, continuous and obeys $\zeta(s)=o(s^{1/4})$ as~$s\to\infty$. Then for all~$x,y>0$ and all $t>0$,
\begin{multline}
\quad
\label{E:4.41b}
P^x\Bigl( \,\min_{0\le s\le t} B_s< 0< \min_{0\le s\le t}\bigl[B_s+\zeta(s\wedge(t-s))\bigr]\,\Big|\,B_t=y\Bigr)
\\
\le 96\biggl(1+\frac{\tilde\rho(x)}x\biggr)\biggl(1+\frac{\tilde\rho(y)}y\biggr)
\biggl(\sqrt{\frac{\tilde\rho(x)}{x+\tilde\rho(x)}}+\sqrt{\frac{\tilde\rho(y)}{y+\tilde\rho(y)}}\biggr)\,\frac{xy}t\,\texte^{-\frac{(x-y)^2}{2t}}\,,
\end{multline}
where $\tilde\rho$ is the quantity from \eqref{E:tilde-rho}.
\end{proposition}

\begin{proofsect}{Proof}
The proof follows closely that of Proposition~\ref{prop-mine}. The events $A_1$ and~$A_2$ are now given by\begin{equation}
A_1:=\Bigl\{\,\min_{s\le t/2}B_s<0< \min_{s\le t/2}\bigl[B_s+\zeta(s)\bigr]\Bigr\}
\quad\text{and}\quad A_2:=\Bigl\{\,\min_{s\le t/2}\bigl[B_s+\zeta(s)\bigr]>0\Bigr\}
\end{equation}
and the symmetrization argument from Lemma~\ref{lemma-symmetrize} again reduces the problem to bounding the probabilities of $A_1\cap A_1'$ and $A_2\cap A_2'$.
Denoting
\begin{equation}
A_3:=\Bigl\{\,\min_{s\le t/2}B_s<0< \min_{s\le \frac34t}\bigl[B_s+\zeta(s)\bigr]\Bigr\}
\quad\text{and}\quad A_4:=\Bigl\{\,\min_{s\le \frac18t}\bigl[B_s+\zeta(s)\bigr]>0\Bigr\}
\end{equation}
from the monotonicity of~$\zeta$ we have $A_1\cap A_1'\subseteq A_3\cap A_4'$ while $A_2\cap A_2'\subseteq A_4\cap A_4'$. The decoupling argument in Lemma~\ref{lemma-decouple} then yields
\begin{equation}
\label{E:4.65}
P^x(A_1\cap A_1'|B_t=x)\le \sqrt8\,P^x(A_3)P^x(A_4)\quad\text{and}\quad P^y(A_2\cap A_2'|B_t=y)\le 2\, P^y(A_4)^2.
\end{equation}
Lemma~\ref{lemma-positive} shows
\begin{equation}
P^x(A_3)=P^0\bigl(\tau_x<t/2,\,\tau_{x+\zeta}>\tfrac34t\bigr)\le 2\frac{\tilde\rho(x)}{\sqrt t}
\end{equation}
while for~$P^y(A_4)$ we get
\begin{equation}
\begin{aligned}
P^y(A_4)&=P^0\bigl(\tau_{y+\zeta}>\tfrac18t\bigr)
=P^0\bigl(\tau_y>\tfrac1{12}t\bigr)+P^0\bigl(\tau_y<\tfrac1{12}t,\,\tau_{y+\zeta}>\tfrac18t\bigr)
\\&\le\sqrt{12}\,\frac{y}{\sqrt t}+\sqrt2\,\sqrt{12}\,\frac{\tilde\rho(y)}{\sqrt t}
\le12\Bigl(1+\frac{\tilde\rho(y)}{y}\Bigr)\frac{y}{\sqrt t}.
\end{aligned}
\end{equation}
To get the desired conclusion, just plug these in \eqref{E:4.65} and use \eqref{E:4.41e} along with some straightforward algebraic manipulations.
\end{proofsect}

Finally, we use the above to dismiss our earlier claim:

\begin{proofsect}{Proof of Proposition~\ref{prop-mine-too-corollary}}
Just combine the right-hand sides of the bounds \eqref{E:4.17q} and~\eqref{E:4.41b}.
\end{proofsect}

\subsection{Entropic repulsion}
\label{sec-A5}\noindent
The above results permit us to give the proofs of Propositions~\ref{prop-entropy-BM} and~\ref{prop-entropy-BB} dealing with the phenomenon of entropic repulsion. We begin by proving the statement in Propositions~\ref{prop-entropy-BM} for unconditioned Brownian motion. First we show that, on the said event, the Brownian path is already quite high at time~$u$:

\begin{lemma}
\label{lemma-7.9}
For~$\zeta$ as in Proposition~\ref{prop-entropy-BM} there is a constant~$c_1=c_1(a,\sigma)>0$ such that for all sufficiently large~$t>0$, all $u\in[0,t/4]$ and all $x\ge \zeta(u)\vee1$,
\begin{equation}
P^0\Bigl(\,\min_{0\le s\le t}\bigl[B_s+\zeta(s)\bigr]>0,\, B_u\le x\Bigr)
\le c_1\,\frac {x^2}{\sqrt u}\,\frac1{\sqrt t}.
\end{equation}
\end{lemma}

\begin{proofsect}{Proof}
For any~$u\in[0,t]$, abbreviate
\begin{equation}
\label{E:7.49}
A_u:=\Bigl\{\,\min_{u\le s\le t}\bigl[B_s+\zeta(s)\bigr]>0\Bigr\}.
\end{equation}
The event in the statement can then be written as $A_0\cap\{B_u\le x\}$. 
From $A_0\subseteq A_u$ we have 
\begin{equation}
\label{E:A.58}
P^0\bigl(A_0\cap\{B_u\le x\}\bigr)
\le E\Bigl(\1_{\{-\zeta(u)< B_u\le x\}}\,P^0\bigl(A_u\,\big|\,\sigma(B_u)\bigr)\Bigr).
\end{equation}
Introducing $\zeta_u(s):=\zeta(u+s)$, on the event $\{B_u=x_1\}$ we then have
\begin{equation}
P^0\bigl(\,A_u\,\big|\,\sigma(B_u)\bigr)
=P^{\,x_1}\Bigl(\,\min_{0\le s\le t-u}\bigl[B_u+\zeta_u(s)\bigr]>0\Bigr).
\end{equation}
A straightforward monotonicity argument then shows that, for~$x_1\in(-\zeta(u),x]$, this is maximized at~$x_1=x$. As~$x\ge\zeta(u)$, Lemma~\ref{lemma-A.9} ensures that $\tilde\rho(x)\le c x$ for some constant~$c=c(a,\sigma)$ independent of~$\zeta(0)$. Proposition~\ref{prop-uncond-BM-positive-cor} and the fact that $t-u\ge t/2$ then show
\begin{equation}
\label{E:A.60}
P^0\bigl(\,A_u\,\big|\,\sigma(B_u)\bigr)
\le c'\frac{x}{\sqrt t},\qquad\text{on }\{B_u\le x\}
\end{equation}
for some constant $c'=c'(a,\sigma)$. Plugging this in the expectation above, the claim follows by noting that $P^0(-\zeta(0)< B_u\le x)\le 2x/\sqrt u$.
\end{proofsect}

With Lemma~\ref{lemma-7.9} in hand, we are ready to tackle to proof of the main claim:

\begin{proofsect}{Proof of Proposition~\ref{prop-entropy-BM}}
Let us augment our earlier notation by writing
\begin{equation}
\label{E:7.49a}
A_u^\pm:=\Bigl\{\,\min_{u\le s\le t}\bigl[B_s\pm\zeta(s)\bigr]>0\Bigr\}.
\end{equation}
Our goal is to bound the probability $P^0(A_0^+\smallsetminus A_u^-)$.
First we note that, by \eqref{E:4.16q}, 
\begin{equation}
\label{E:A.63}
P^0(A_0^+)\ge \frac13\frac{\zeta(0)}{\sqrt t}\quad\text{whenever}\quad t>\zeta(0)^2.
\end{equation}
The condition $t>\zeta(0)^2$ will be ensured by assuming $c'>\frac12\zeta(0)^2$. Lemma~\ref{lemma-7.9} then shows
\begin{equation}
\label{E:A.64}
\frac{P^0\bigl(A_0^+\cap\{B_u\le x\}\bigr)}{P^0(A_0^+)}
\le 3c_1\frac1{\zeta(0)}\,\frac{x^2}{\sqrt u}
\end{equation}
whenever~$u\in[c',t/2]$ and $x\ge\zeta(u)\vee1$.
To get the claim, it thus suffices to derive a good lower bound on the conditional probability $P(A_u^-\cap\{B_u\ge x\}\,|\,A_0^+)$.

For simplicity of certain bounds later, we may and will assume that~$x>\zeta(u)\vee\texte$. For any $0\le u\le t$ denote $A_{0,u}^+:=\{B_s>-\zeta(s)\colon 0\le s\le u\}$. Then $A_0^+\cap A_u^-=A_{0,u}^+\cap A_u^-$ and so setting $\FF_u:=\sigma(B_s\colon 0\le s\le u)$ and noting that $A_{0,u}^+\in\FF_u$,
\begin{equation}
P^0\bigl(A_0^+\cap A_u^-\cap\{B_u\ge x\}\bigr)
=E^0\Bigl(\1_{A_{0,u}^+\cap\{B_u\ge x\}}\,P^0(A_u^-\,|\,\FF_u)\Bigr).
\end{equation}
We will now derive a uniform estimate on the conditional probability on the right. First we note that, on $\{B_u=x_1\}$ we have
\begin{equation}
P^0( A_u^-\,|\,\FF_u) 
= P^{\,x_1}\bigl(B_s\ge\zeta_u(s)\colon s\in[0,t-u]\bigr)
\end{equation}
where, as before, $\zeta_u(s):=\zeta(u+s)$. Thanks to Proposition~\ref{prop-uncond-BM-positive-cor} and Lemma~\ref{lemma-A.9} we then get, for some constant $c_2=c_2(a,\sigma)\in(0,\infty)$,
\begin{equation}
P^0( A_u^-\,|\,\FF_u)\ge\sqrt{\frac2\pi}\frac{B_u}{\sqrt t}\,\biggl(1-\frac{x^2}t-c_2\Bigl(\frac{\zeta(u)+\log x}x\Bigr)^{2/3}\,\biggr)\qquad\text{on }\{B_u\ge x\}.
\end{equation}
Similarly, Proposition~\ref{prop-4.9ue} and Lemma~\ref{lemma-A.9} yield
\begin{equation}
P^0( A_u^+\,|\,\FF_u)\le\sqrt{\frac2\pi}\frac{B_u}{\sqrt t}\,\biggl(1+c_3\Bigl(\frac{\zeta(u)+\log x}x\Bigr)^{2/3}\,\biggr)\qquad\text{on }\{B_u\ge x\},
\end{equation}
for some constant $c_3=c_3(a,\sigma)\in(0,\infty)$. It then follows that
\begin{equation}
P^0\bigl(\,A_0^+\cap A_u^-\cap\{B_u\ge x\}\bigr)
\ge\frac{1-\frac{x^2}t-c_2\bigl(\frac{\zeta(u)+\log x}x\bigr)^{2/3}}{1+c_3\bigl(\frac{\zeta(u)+\log x}x\bigr)^{2/3}}\,
P^0\bigl(\,A_0^+\cap\{B_u\ge x\}\bigr).
\end{equation}
In combination with \eqref{E:A.64}, we then get
\begin{equation}
\label{E:A.70}
\frac{P^0(A_0^+\smallsetminus A_u^-)}{P^0(A_0^+)}
\le 3c_1\,\frac1{\zeta(0)}\frac {x^2}{ \sqrt u}+\frac{x^2}t+(c_2+c_3)\Bigl(\frac{\zeta(u)+\log x}x\Bigr)^{2/3}.
\end{equation}
subject to~$u\in[c',t/2]$ and $x\ge\zeta(u)\vee1$.

Proposition~\ref{prop-4.9ue} gives $P^0(A_0^+)\le c_4/\sqrt t$ for some $c_4=c_4(a,\sigma,\zeta(0))$ and so $P^0(A_0^+\smallsetminus A_u^-)$ is bounded by $c_4/\sqrt t$ times the expression on the right of \eqref{E:A.70}. We then choose $x:=u^{\frac7{32}}$ and note that the first term then dominates the other two as soon as~$u$ is sufficiently large. The claim then follows by noting that $x^2/\sqrt u=u^{-\frac1{16}}$ for our choice of~$x$.
\end{proofsect}

Having dealt with entropic repulsion of unconditioned paths, the claim for the Brownian bridge follows readily as well:

\begin{proofsect}{Proof of Proposition~\ref{prop-entropy-BB}}
Consider the events
\begin{equation}
A_1:=\Bigl\{\,\min_{s\le t/2}\bigl[B_s+\zeta(s)\bigr]>0>\min_{u\le s\le t/2}\bigl[B_s-\zeta(s)\bigr]\Bigr\}
\end{equation}
and
\begin{equation}
A_2:=\Bigl\{\,\min_{s\le \frac14t}\bigl[B_s+\zeta(s)\bigr]>0\Bigr\}.
\end{equation}
The event in \eqref{E:4.34r} is than contained in $(A_1\cap A_2')\cup(A_2\cap A_1')$ and so, by the union bound and the decoupling trick in Lemma~\ref{lemma-decouple}, its probability is bounded by $2\sqrt8\,P^0(A_1)P^0(A_2)$. Then $P^0(A_1)$ is bounded using Proposition~\ref{prop-entropy-BM} while $P^0(A_2)$ using Proposition~\ref{prop-4.9ue}.
\end{proofsect}

\newcounter{appendB}
\renewcommand{\theappendB}{B}
\section*{Appendix B: Useful properties and bounds}
\addcontentsline{toc}{section}
{{\tocsection {}{8}{\!\!\!\!Appendix B: Useful properties and bounds\dotfill}}{}}
\refstepcounter{appendB}
\label{appendB}
\renewcommand{\thesection}{B}
\setcounter{subsection}{0}
\setcounter{theorem}{0}
\setcounter{equation}{0}
\renewcommand{\thesubsection}{B.\arabic{subsection}}
\renewcommand{\theequation}{\thesection.\arabic{equation}}
\noindent
In this short section we collect various useful facts relevant for the study of the \DGFF{}. We also restate the results about the behavior of its extreme values that are used in this work. Having these explicated here will ease referencing throughout the rest of the article.  Detailed proofs of many of these facts can be found in Biskup~\cite{Biskup-PIMS}. 

\subsection{Gaussian processes}
We begin with two standard results concerning boundedness and continuity of rather general Gaussian processes. A good reference for this material are the books of Adler~\cite{Adler} and the introductory part of Adler and Taylor~\cite{Adler-Taylor}.

\begin{lemma}[Fernique majorization]
\label{lemma-Fernique}
There is~$K\in(0,\infty)$ such that the following holds: Let $X=\{X_t\colon t\in\mathfrak X\}$ be a separable centered Gaussian field indexed by points in a totally-bounded (pseudo)metric space $(\mathfrak X,\rho)$, where $\rho(t,s):=[E((X_t-X_{s})^2)]^{1/2}$. Then for any Borel probability measure~$\mu$ on~$\mathfrak X$ 
\begin{equation}
\label{E:B.1}
E\bigl(\,\sup_{t\in\mathfrak X}\,X_t\bigr)\le K\,\sup_{t\in\mathfrak X}\,\int_0^\infty \textd r\,\sqrt{\log\frac1{\mu(B_\rho(t,r))}},
\end{equation}
where $B_\rho(t,r):=\{s\in\mathfrak X\colon \rho(t,s)\le r\}$. In addition, we also get
\begin{equation}
\label{E:B.2}
E\Bigl(\,\,\,\sup_{\begin{subarray}{c}
s,t\in\mathfrak X\\ \rho(s,t)\le \epsilon
\end{subarray}}
\,|X_t-X_s|\Bigr)\le K\,\sup_{t\in\mathfrak X}\,\int_0^\epsilon \textd r\,\sqrt{\log\frac1{\mu(B_\rho(t,r))}}.
\end{equation}
\end{lemma}
 
\begin{proofsect}{Proof}
For \eqref{E:B.1} see, e.g., Adler~\cite[Theorem~4.1]{Adler}. For \eqref{E:B.2} see  the calculation in the proof of Adler~\cite[Theorem~4.5]{Adler}.
\end{proofsect}
 
\begin{lemma}[Borell-Tsirelson inequality]
\label{lemma-BT}
Let~$\mathfrak X$ be a metric space and suppose $\{X_t\colon t\in\mathfrak X\}$ is a separable centered Gaussian process with $\sup_{t\in X}X_t<\infty$ a.s. Then
\begin{equation}
P\Bigl(\,\sup_{t\in\mathfrak X}\,X_t-E\bigl(\,\sup_{t\in\mathfrak X}\,X_t\bigr)>\lambda\Bigr)
\le \texte^{-\frac{\lambda^2}{2\sigma^2}},\qquad \lambda>0,
\end{equation}
where $\sigma^2:=\sup_{t\in\mathfrak X}E(X_t^2)$.
\end{lemma}

\begin{proofsect}{Proof}
See, e.g., Adler~\cite[Theorem~2.1]{Adler}.
\end{proofsect}

\subsection{Harmonic analysis}
\label{sec-GF}\noindent
An attractive feature of the \DGFF{} as defined above is its connection with discrete harmonic analysis on~$\Z^2$. This is a subject that has been heavily studied in the past; see, e.g., the books by Lawler~\cite{Lawler} and Lawler and Limi\'c~\cite{Lawler-Limic}.
We need three  objects:
\begin{enumerate}
\item[(1)] the (discrete) Green function $G^D(x,y)$ , defined for each~$D\subsetneq\Z^2$, as the expected number of visits to~$y$ of the simple random walk started from~$x$ and killed upon exit from~$D$,
\item[(2)] the harmonic measure $H^D(x,y)$, defined as the probability that the random walk started from~$x$ first hits~$\Z^2\smallsetminus D$ at vertex~$y$, and
\item[(3)] the potential  kernel  $\fraka\colon\Z^2\to[0,\infty)$ defined, e.g., by \eqref{E:2.8u}.
\end{enumerate}
The simple random walk on the square lattice is recurrent and so $H^D(x,\cdot)$ is a probability measure on~$\partial D$ for each~$D\subsetneq \Z^2$. Alternative definitions of the potential  kernel  exist, e.g., using the Green function
\begin{equation}
\fraka(x) = \lim_{N\to\infty}\bigl[G^{\wt V_N}(0,x)-G^{\wt V_N}(0,0)\bigr],
\end{equation}
where $\wt V_N:=(-N,N)^2\cap\Z^2$. The connection works the other way round as well:

\begin{lemma}
\label{lemma-G-potential}
For each $D\subset\Z^2$ finite, the Green function $G^D$ in~$D$ obeys
\begin{equation}
\label{E:B.4}
G^D(x,y)=-\fraka(x-y)+\sum_{y\in\partial D}H^D(x,z)\fraka(y-z).
\end{equation}
\end{lemma}

\begin{proofsect}{Proof (sketch)}
The key is to check that $y\mapsto G^D(x,y)+\fraka(x-y)$ is discrete harmonic on~$D$. The stated identity then follows from  the well-known representation of  the solution to a discrete Dirichlet problem.
\end{proofsect}

The Green function $G^{\Z^2\smallsetminus\{0\}}$ is given by \eqref{E:2.8} which appears to have an additional term compared to \eqref{E:B.4}. This term arises from the limit argument that is needed to make the Dirichlet problem uniquely solvable. 
The bound \eqref{E:B.4} is useful in estimates, particularly, in light of:

\begin{lemma}
\label{lemma-potential}
The potential  kernel  $\fraka$ admits the following asymptotic expression
\begin{equation}
\label{E:B.5}
\fraka(x)=g\log|x|+c_0+O(|x|^{-2}),\qquad|x|\to\infty,
\end{equation}
where  $c_0$ is a constant, $g=2/\pi$ and~$|\cdot|$ is the Euclidean norm on~$\R^2$. 
\end{lemma}

\begin{proofsect}{Proof}
This was apparently first proved by St\"ohr~\cite{Stoehr} using Fourier analysis. See also Fukai and Uchiyama~\cite{Fukai-Uchiyama} and Kozma and Schreiber~\cite{Kozma-Schreiber} for a more general approach to this.
\end{proofsect}

Concerning the harmonic measure $H(x,\cdot)$, we need to control regularity in~$x$ uniformly in the second argument. Fortunately, it suffices to do this for square-like domains: 

\begin{lemma}
\label{lemma-HM}
Recall that $V_N:=(0,N)^2\cap\Z^2$ and let $\epsilon\in(0,1/2)$. There is a constant $c=c(\epsilon)$ such that for any~$x,y\in V_N$ satisfying $\dist(x,V_N^\cc)\ge \epsilon N$ and $\dist(y,V_N^\cc)\ge \epsilon N$,
\begin{equation}
\max_{z\in\partial V_N}\,H^{V_N}(x,z)\le \frac{c}N
\end{equation}
and
\begin{equation}
\max_{z\in\partial V_N}\,\bigl|\,H^{V_N}(x,z)-H^{V_N}(y,z)\bigr|\le c\frac{|x-y|}{N^2}.
\end{equation}
\end{lemma}

\begin{proofsect}{Proof (idea)}
This can be proved, e.g., by invoking the continuum approximation of the harmonic measure (cf Lawler and Limi\'c~\cite[Proposition 8.1.4]{Lawler-Limic}) and the corresponding (standard) estimate for the continuum Poisson kernel.
\end{proofsect}

To keep out notations light, we will abuse it by occasionally writing~$G^D$ to denote also the continuum Green function. This object may as well be defined by an analogue of \eqref{E:B.4},
\begin{equation}
\label{E:B.7}
G^D(x,y)=-g\log|x-y|+g\int_{\partial D}\Pi^D(x,\textd z)\log|y-z|,
\end{equation}
where~$\Pi^D(x,\cdot)$ is the harmonic measure (a.k.a.\ Poisson kernel) associated with the standard Brownian motion killed upon exit from~$D$. In light of  \eqref{E:B.5},  given a sequence $\{D_N\}$ of scaled-up lattice domains approximating~$D\in\mathfrak D$ via \twoeqref{E:1.1}{E:1.1a}, \eqref{E:B.4} converges (pointwise and locally uniformly) to \eqref{E:B.7} away from the diagonal in $D\times D$.

\subsection{Discrete Gaussian Free Field}
Next we will move to the properties of the \DGFF{}. Recall that~$h^D$ denotes the \DGFF{} in~$D\subsetneq\Z^2$ where we regard $h^D$ as zero outside~$D$. A very useful property is the behavior of~$h^D$ under conditioning on values in a subset of~$D$.  The following is sometimes called the \emph{domain Markov property} in the literature: 

\begin{lemma}[Gibbs-Markov property]
\label{lemma-GM}
Let $\wt D\subsetneq D\subsetneq\Z^2$ and denote
\begin{equation}
\varphi^{D,\wt D}(x):=E\bigl(h^D(x)\,\big|\,\sigma(h^D(z)\colon z\in D\smallsetminus \wt D)\bigr).
\end{equation}
Then we have:
\begin{enumerate}
\item[(1)] A.e.\ sample of $x\mapsto\varphi^{D,\wt D}$ is discrete harmonic on~$\wt D$ with ``boundary values'' determined by $\varphi^{D,\wt D}(x)=h^D(x)$ for each $x\in D\smallsetminus \wt D$.
\item[(2)] The field $h^D-\varphi^{D,\wt D}$ is independent of~$\varphi^{D,\wt D}$ and, in fact, $h^D-\varphi^{D,\wt D} \laweq h^{\wt D}$.
\end{enumerate}
\end{lemma}

\begin{proofsect}{Proof (idea)}
This is a consequence of an explicit representation of the probability law of~$h^D$ as a Gibbs measure with nearest-neighbor interactions only.
\end{proofsect}

As a simple consequence of the Gibbs-Markov property we get:

\begin{lemma}
\label{lemma-stoch-order}
If $A\subseteq\wt D\subseteq D\subsetneq\Z^2$ then
\begin{equation}
P\bigl(\max_{x\in A} h^{\wt D}(x)\ge\lambda\bigr)\le 
2P\bigl(\max_{x\in A} h^{D}(x)\ge\lambda\bigr)
\end{equation}
holds for each $\lambda\ge0$.
\end{lemma}

\begin{proofsect}{Proof}
Just write $h^D=h^{\wt D}+\varphi^{D,\wt D}$ and impose $\varphi^{D,\wt D}\ge0$ at the maximizer of~$h^{\wt D}$ on~$A$.
\end{proofsect}

Another useful feature of the \DGFF{} are positive correlations. Recall that a probability measure~$\mu$ on the product space $\R^{\Z^2}$ is \emph{strong FKG} if for any finite $\Lambda\subset\Z^2$ and any increasing events~$A$ and~$B$ --- with ``increasing'' defined with respect to the usual partial order on $\R^{\Z^2}$ --- we have
\begin{equation}
\mu\bigl(A\cap B\,|\,\FF_\Lambda\bigr)
\ge \mu\bigl(A\,|\,\FF_\Lambda\bigr) \mu\bigl(B\,|\,\FF_\Lambda\bigr)
\end{equation}
where $\FF_\Lambda$ is the $\sigma$-algebra generated by the values of the field in~$\Lambda$. We have:

\begin{lemma}[Positive correlations]
\label{lemma-FKG}
For any $D\subsetneq\Z^2$, the law of $h^D$ is strong FKG.
\end{lemma}

\begin{proofsect}{Proof (idea)}
For~$D$ finite, the probability density of $h^D$ with respect to the product Lebesgue measure satisfies the so called FKG lattice condition which is sufficient to imply the strong FKG property. The case of infinite~$D$ is obtained by suitable limits.
\end{proofsect}

\subsection{Extreme values}
Our final set of results to be reviewed here concern the extreme values of the \DGFF{}. Recall the notation $\mathfrak D$ from Section~\ref{sec-2.1} for the class of admissible continuum domains and~$\Gamma_N^D(t)$ from \eqref{E:5.47} for the set of values where the DGFF in~$D_N$ is above~$m_N-t$. We then have:

\begin{lemma}
\label{lemma-level-set}
For all $D\in\mathfrak D$ and all $t \in \R$, 
\begin{equation}
	\lim_{a \to \infty}\, \limsup_{N \to \infty}\,
		P \bigl( |\Gamma_N^D(t)|\ge a\bigr) = 0 \,.
\end{equation}
\end{lemma}

\begin{proofsect}{Proof}
For square-like domains, this follows from Ding and Zeitouni~\cite[Theorem~1.2]{DZ}. The extension to more general domains can be deduced from Lemma~\ref{lemma-Gamma-monotonicity} below.
\end{proofsect}

Some applications require knowledge of the size of the intersection of~$\Gamma^D_N(t)$ with an underlying set. The following comparison lemma is then quite useful:

\begin{lemma}
\label{lemma-Gamma-monotonicity}
For each $U\subset V\subset W$, each $a\in\R$ and each integer~$b>0$,
\begin{equation}
P\Bigl(\bigl|\{x\in U\colon h^V(x)\ge a\}\bigr|\ge2b\Bigr)
\le 2P\Bigl(\bigl|\{x\in U\colon h^W(x)\ge a\}\bigr|\ge b\Bigr).
\end{equation}
\end{lemma}

\begin{proofsect}{Proof}
This is a restatement of the last part of Lemma~3.4 in Biskup and Louidor~\cite{BL2}.
\end{proofsect}

\begin{lemma}
\label{lemma-separation}
For all $D\in\mathfrak D$ and all $t \in \R$, 
\begin{equation}
	\lim_{r \to \infty}\, \limsup_{N \to \infty}\,
		P \bigl( \exists x,y \in \Gamma_N^D( t)\colon r<|x-y| < N/r\bigr) = 0 \,.
\end{equation}
\end{lemma}

\begin{proofsect}{Proof}
For square domains, this follows from Theorem~1.1 of~Ding and Zeitouni~\cite{DZ} (where one allows even for $t$ to increase as a constant times $\log\log r$). The (simple) extension to general domains $D\in\mathfrak D$ is provided in Proposition~3.1 of Biskup and Louidor~\cite{BL2}. 
\end{proofsect}

\begin{lemma}
\label{lem:6}
Recall that $\alpha:=2/\sqrt g$ and $V_N:=(0,N)^2\cap\Z^2$. There is a constant~$c\in(0,\infty)$ such that for all $s\ge0$, all $t\ge1$, all $N \geq 1$ and all sets $A \subseteq D\subseteq V_N$, 
\begin{equation}
		P \Bigl(\,\, \max_{x\in A}h^D(x) \geq m_N + t-s \Bigr)
			\leq c\,\Bigl(\frac{|A|}{N^2}\Bigr)^{1/2}\, 
			t\,\texte^{-\alpha (t-s)} \,.
\end{equation}
\end{lemma}

\begin{proofsect}{Proof}
For $D:=V_N$ this is  Lemma~3.8 of Bramson, Ding and Zeitouni~\cite{BDingZ}. The more general case is implied by the bound in Lemma~\ref{lemma-stoch-order}. 
\end{proofsect}

\begin{lemma}
\label{lemma-max-tail}
There are constants $c_1,c_2\in(0,\infty)$ such that for $V_N:=(0,N)^2\cap\Z^2$,
\begin{equation}
P\Bigl(\,\bigl|\max_{x\in V_N}h^{V_N}(x)-m_N\bigr|>\lambda\Bigr)\le c_1\texte^{-c_2\lambda}.
\end{equation}
\end{lemma}

\begin{proofsect}{Proof}
This is a restatement of Theorem~1.1 in~Ding~\cite{Ding}.
\end{proofsect}

Finally, let us address the passage to continuum limit. Given $\wt{D},D\in\mathfrak D$ with $\wt D\subseteq D$, we will write~$\varphi^{ D, \wt{ D}}_N$ as a shorthand for~$\varphi^{ D_N, \wt{ D}_N}$. Then we have:

\begin{lemma}
\label{lemma-BF-converge}
Let $\wt{D},D\in\mathfrak D$ obey $\wt D\subseteq D$. Then for all $x,y \in \wt{ D}$
\begin{equation}
\text{\rm Cov}\bigl(
	\varphi^{ D, \wt{ D}}_{N}(\lf xN \rf) \,,\,
	\varphi^{ D, \wt{ D}}_{N}(\lf yN \rf) \bigr)
\,\underset{N\to\infty}\longrightarrow\, 
C^{D,\wt D}(x,y),
\end{equation}
with the convergence uniform over closed subsets of~$\wt D\times\wt D$. In particular, for $\delta>0$ and each $N\ge1$ there is a coupling of $\varphi_N^{D,\wt D}$ and $\Phi^{D,\wt D}$ such that
\begin{equation}
\sup_{\begin{subarray}{c}
x\in \wt D\\\dist(x, \partial\wt D)>\delta
\end{subarray}}
\bigl|\,\Phi^{D,\wt D}(x)-\varphi_N^{D,\wt D}(x/N)\bigr|\,\underset{N\to\infty}\longrightarrow\,0,\qquad\text{\rm in probability}.
\end{equation}
\end{lemma}

\begin{proofsect}{Proof (sketch)}
The convergence of covariances follows from the stated convergence of \eqref{E:B.4} to \eqref{E:B.7}.
Fix~$\delta>0$ and recall $\wt D^\delta:=\{x\in\R^2\colon\dist(x,\partial\wt D)>\delta\}$. Given~$r>0$ small and let $x_1,\dots,x_k$ be an $r$-net in $\wt D^\delta$. As convergence of the covariances implies convergence in law, and convergence in law on~$\R^n$ can be realized as convergence in probability, for each~$N\ge1$ there is a coupling of $\varphi_N^{D,\wt D}$ and $\Phi^{D,\wt D}$ such that
\begin{equation}
P\Bigl(\,\,\max_{i=1,\dots,k}\bigl|\,\Phi^{D,\wt D}(\lfloor Nx_i\rfloor)-\varphi_N^{D,\wt D}(x_i)\bigr|>\epsilon\Bigr)
\,\underset{N\to\infty}\longrightarrow\,0.
\end{equation}
The claim will then follow if we can show that
\begin{equation}
\lim_{r\downarrow0}\,\limsup_{N\to\infty}\,P\biggl(\,\,\sup_{\begin{subarray}{c}
x,y\in \wt D^\delta\\|x-y|<r
\end{subarray}}
\bigl|\Phi^{D,\wt D}(x)-\Phi^{D,\wt D}(y)\bigr|>\epsilon\biggr)=0
\end{equation}
and similarly for $\Phi^{D,\wt D}(\cdot)$ replaced by $\varphi_N^{D,\wt D}(\lfloor N\cdot\rfloor)$. This is checked using Lemmas~\ref{lemma-Fernique} and~\ref{lemma-BT} and some elementary regularity of $C^{D,\wt D}$.
\end{proofsect}

\section*{Acknowledgments}
\noindent
This research has been partially supported by European Union's Seventh Framework Programme (FP7/2007-2013]), under grant no.~276923-MC--MOTIPROX, the Israeli Science Foundation grant~1382/17 and the German-Israeli Foundation for Scientific Research and Development grant I-2494-304.6/2017, and also by the NSF grant DMS-1407558 and GA\v CR project P201/16-15238S. We wish to thank Miika Nikula and Scott Sheffield for useful discussions.


\begin{thebibliography}{AX}
\addcontentsline{toc}{section}
{{\tocsection {}{9}{\!\!\!\!References\dotfill}}{}}

\bibitem{Abe}
Y.~Abe (2016).
Extremes of local times for simple random walks on symmetric trees.
 arXiv:1603.09047

\bibitem{Adler}
R.J.~Adler (1990). \textit{An introduction to continuity, extrema, and related topics for general Gaussian processes}. Institute of Mathematical Statistics Lecture Notes--Monograph Series, vol.~12. Institute of Mathematical Statistics, Hayward, CA, x+160 pp.

\bibitem{Adler-Taylor}
R.J.~Adler and J.E.~Taylor (2007).
\textit{Random fields and geometry}. Springer Monographs in Mathematics. Springer, New York.

\bibitem{Aidekon}
E. A\"id\'ekon (2013).
Convergence in law of the minimum of a branching random walk. 
\textit{Ann. Probab.}~\textbf{41}, no. 3A, 1362--1426. 

\bibitem{ABBS}
E. A\"id\'ekon, J.~Berestycki, E.~Brunet and Z.~Shi (2013). 
The branching Brownian motion seen from its tip. 
\textit{Probab. Theory Rel. Fields}~\textbf{157}, no.~1--2, 405--451.


\bibitem{ABK1} 
L.-P. Arguin, A. Bovier, and N. Kistler (2011). 
Genealogy of extremal particles of branching brownian motion. 
\textit{Commun. Pure Appl. Math.}~\textbf{64}, 1647--1676.

\bibitem{ABK2}
L.-P. Arguin, A. Bovier, and N. Kistler (2013). 
Poissonian statistics in the extremal process of branching brownian motion. 
\textit{Ann. Appl. Probab.} (to appear) arXiv:1010.2376.

\bibitem{ABK3}
L.-P. Arguin, A. Bovier, and N. Kistler (2011). 
The extremal process of branching Brownian motion. 
\textit{Probab. Theory Rel. Fields}~\textbf{157}, no.~3--4, 535--574.

\bibitem{Arguin-Zindy}
L.-P. Arguin and O. Zindy (2015). 
Poisson-Dirichlet Statistics for the extremes of a log-correlated Gaussian field.
\textit{Electron. J. Probab.}~\textbf{20}, no. 59, 1--19


\bibitem{Biskup-PIMS}
M. Biskup (2017). 
Extrema of the two-dimensional Discrete Gaussian Free Field.
arXiv:1712.09972 


\bibitem{BL1}
M. Biskup and O. Louidor (2016).
Extreme local extrema of two-dimensional discrete Gaussian free field. \textit{Commun. Math. Phys.} \textbf{345}, no. 1, 271--304.

\bibitem{BL2}
M. Biskup and O. Louidor (2014).
Conformal symmetries in the extremal process of two-dimensional discrete Gaussian Free Field. arXiv:1410.4676.


\bibitem{BL3}
M. Biskup and O. Louidor (2016).
On intermediate level sets of two-dimensional discrete Gaussian Free Field.
 arXiv:1612.01424


\bibitem{BDZ}
E. Bolthausen, J.-D. Deuschel, and O. Zeitouni (2011). 
Recursions and tightness for the maximum of the discrete, two dimensional gaussian free field. 
\textit{Elect. Commun. Probab.}~\textbf{16}, 114--119.


\bibitem{Bovier-Hartung} 
A.~Bovier and L.~Hartung (2015).  
Extended convergence of the extremal process of branching Brownian motion.
arXiv:1412.5975

\bibitem{Bramson-CPAM}
M. Bramson (1978). 
Maximal displacement of branching Brownian motion.
\textit{Commun. Pure Appl. Math.}~\textbf{31}, no.~5,  531--581.

\bibitem{BDingZ}
M. Bramson, J. Ding and O. Zeitouni (2016). 
Convergence in law of the maximum of the two-dimensional discrete Gaussian free field. 
\textit{Commun. Pure Appl. Math.}~\textbf{69} 62--123.


\bibitem{BZ}
M. Bramson and O. Zeitouni (2011). 
Tightness of the recentered maximum of the two-dimensional discrete Gaussian free field. 
\textit{Commun. Pure Appl. Math.}~\textbf{65}, 1--20.

\bibitem{Brydges-Park-City}
D.C.~Brydges, \textit{Lectures on the renormalisation group}. Statistical mechanics, pp.\ 7Ð93, 
IAS/Park City Math. Ser., vol.~16, Amer. Math. Soc., Providence, RI, 2009. 

\bibitem{Carpentier-LeDoussal}
D.~Carpentier and P.~Le Doussal (2001).
{Glass transition of a particle in a random potential, front selection in nonlinear renormalization group, and entropic phenomena in Liouville and sinh-Gordon models},
\textit{Phys. Rev.~E}~\textbf{63}, 026110.

\bibitem{CCH1}
A.~Chiarini, A.~Cipriani, R.S.~Hazra (2015).
Extremes of the supercritical Gaussian Free Field.
arXiv:1504.07819

\bibitem{CCH2}
A.~Chiarini, A.~Cipriani, R.S.~Hazra (2015).
A note on the extremal process of the supercritical Gaussian Free Field.
arXiv:1505.05324

\bibitem{CCH3}
A.~Chiarini, A.~Cipriani, R.S.~Hazra (2015).
Extremes of some Gaussian random interfaces.
arXiv: 1509.08903

\bibitem{Daviaud}
O. Daviaud (2006). 
{Extremes of the discrete two-dimensional Gaussian free field}, 
\textit{Ann. Probab.}~\textbf{34}, no.~3, \hbox{962--986}.

\bibitem{DRZ}
J.~Ding, R.~Roy, O.~Zeitouni (2015).
Convergence of the centered maximum of log-correlated Gaussian fields.
arXiv: 1503.04588

\bibitem{Derrida-Spohn}
B. Derrida and H. Spohn (1988). 
Polymers on disordered trees, spin glasses, and traveling waves.
\textit{J. Statist. Phys.}~\textbf{51}, no. 5-6, 817--840.


\bibitem{Ding}
J. Ding (2013). 
Exponential and double exponential tails for maximum of two-dimensional discrete Gaussian free field. 
\textit{Probab. Theory Rel. Fields}~\textbf{157}, no.~1--2, 285--299.


\bibitem{DZ}
J. Ding and O. Zeitouni (2014). 
Extreme values for two-dimensional discrete Gaussian free field. 
\textit{Ann. Probab.}~\textbf{42}, no.~4, 1480--1515

\bibitem{Fukai-Uchiyama}
Y.~Fuka and K.~Uchiyama (1996).
Potential kernel for two-dimensional random walk.
\textit{Ann. Probab.} \textbf{24}, no.~4, 1979--1992.

\bibitem{Fyodorov-Bouchard}
Y. V. Fyodorov and J.-P. Bouchaud (2008). 
Freezing and extreme-value statistics in a random energy model with logarithmically correlated potential. 
\textit{J. Phys.~A}~\textbf{41}, no. 37, 372001.

\bibitem{HMP10}
X.~Hu, J.~Miller and Y.~Peres (2010). Thick point of the Gaussian free field.
\textit{Ann. Probab.} \textbf{38}, no.~2, 896--926.


\bibitem{Kozma-Schreiber}
G.~Kozma and E.~Schreiber  (2004). 
An asymptotic expansion for the discrete harmonic potential.
\textit{Electron. J. Probab.}~\textbf{9}, Paper no. 1, pages 10--17.


\bibitem{Kurt}
N.~Kurt (2009).
Maximum and entropic repulsion for a Gaussian membrane model in the critical dimension. 
\textit{Ann. Probab.} \textbf{37}, no.~2, 687--725.


\bibitem{Lawler}
G.F.~Lawler (1991). \textit{Intersections of Random Walks}. 
Birkh\"auser Boston, Inc., Boston, MA.

\bibitem{Lawler-Limic}
G.F.~Lawler and V.~Limi\'c (2010).
\textit{Random walk: a modern introduction}. 
Cambridge Studies in Advanced Mathematics, vol.~123. Cambridge University Press, Cambridge, xii+364.

\bibitem{MRV}
T.~Madule, R.~Rhodes and~V. Vargas (2013). 
Glassy phase and freezing of log-correlated Gaussian potentials.
arXiv:1310.5574

\bibitem{RV-review}
R.~Rhodes and V.~Vargas (2014).
Gaussian multiplicative chaos and applications: a review.
\textit{Probab. Surveys}~\textbf{11}, 315--392.


\bibitem{Stoehr}
A.~St\"ohr (1950). 
Uber einige lineare partielle Differenzengleichungen mit konstanten Koeffizienten~III.
\textit{Math. Nachr.}~\textbf{3}, 330--357.

\bibitem{Subag-Zeitouni}
E.~Subag and O.~Zeitouni (2014). Freezing and decorated Poisson point processes. arXiv:1404.7346

\end{thebibliography}
\end{document}